\documentclass{article}
\usepackage[utf8]{inputenc}
\usepackage[english]{babel}
\usepackage{pdfpages}
\usepackage[margin=3cm]{geometry}
\usepackage{graphicx}
\usepackage{amsmath,amsthm,amssymb}
\usepackage{mathtools,enumerate,float}
\usepackage{wasysym}
\usepackage{microtype}
\usepackage{hyperref}
\usepackage[capitalise]{cleveref}

\usepackage{todonotes}
\makeatletter
\providecommand\@dotsep{5}
\renewcommand{\listoftodos}[1][\@todonotes@todolistname]{%
	\@starttoc{tdo}{#1}}
\makeatother

\newenvironment{cproof}{\begin{proof}[Proof of the claim]}{\end{proof}}

\title{Quantitative measure equivalence between amenable groups}
\author{Thiebout Delabie\footnote{supported by grant P2NEP2 181564 of the Swiss National Science Foundation}, Juhani Koivisto, François Le Maître\footnote{Research 
		partially supported by Projet ANR-17-CE40-0026 AGRUME, Projet 
		ANR-19-CE40-0008 AODynG and IdEx Université de Paris, ANR-18-IDEX-0001.}  and Romain Tessera}
\date{\today}

\newtheorem{theorem}{Theorem}[section]
\newtheorem{lemma}[theorem]{Lemma}
\newtheorem{proposition}[theorem]{Proposition}
\newtheorem{corollary}[theorem]{Corollary}

\newtheorem{thmintro}{Theorem}
\newtheorem{corintro}[thmintro]{Corollary}

\theoremstyle{definition}
\newtheorem{example}[theorem]{Example}

\newtheorem*{convention}{Convention}
\newtheorem{definition}[theorem]{Definition}
\newtheorem{remark}[theorem]{Remark}
\newtheorem{clai}[theorem]{Claim}

\newtheorem{ques}[theorem]{Question}
\newtheorem*{ack}{Acknowledgments}

\newcommand{\abs}[1]{\left\lvert #1\right\rvert}
\newcommand{\norm}[1]{\left\lVert #1\right\rVert}
\newcommand{\la}{\left\langle}
\newcommand{\ra}{\right\rangle}

\newcommand{\vol}{\operatorname{Vol}}

\renewcommand{\leq}{\leqslant}
\renewcommand{\geq}{\geqslant}
\renewcommand{\le}{\leqslant}
\renewcommand{\ge}{\geqslant}
\renewcommand{\epsilon}{\varepsilon}

\newcommand{\asdim}{\operatorname{asdim}}

\newcommand{\BS}{\operatorname{BS}}

\newcommand{\supp}{\operatorname{supp}}

\newcommand{\PSL}{\operatorname{PSL}}
\newcommand{\N}{\mathbb{N}}
\newcommand{\Z}{\mathbb{Z}}
\newcommand{\Q}{\mathbb{Q}}
\newcommand{\R}{\mathbb{R}}

\newcommand{\Heis}{\mathbb{H}}
\newcommand{\eps}{\varepsilon}
\newcommand{\LL}{\mathrm{L}}
\newcommand{\act}{\curvearrowright}
\newcommand{\inv}{^{-1}}
\newcommand{\dcof}{\rho}

\newcommand{\actom}{*}
\newcommand{\actx}{\cdot}

\makeatletter
\newcommand{\AlignRight}[1]{\ifmeasuring@#1\else\omit\hfill$\displaystyle#1$\fi\ignorespaces}
\makeatother

\begin{document}
	
	\maketitle

\begin{abstract}	We initiate a quantitative study of measure equivalence (and orbit equivalence) between finitely 
	generated groups, which extends the classical setting of $\LL^p$ measure equivalence. In this paper, our main focus will be on amenable groups, for which we prove both rigidity and flexibility results. 
	
	On the rigidity side, we prove a general monotonicity property satisfied by the isoperimetric profile, which implies in particular its invariance under $\LL^1$ measure equivalence. This yields explicit  ``lower bounds" on how integrable a measure coupling between two amenable groups can be. This result also has an unexpected application to geometric group theory:  the isoperimetric profile turns out to be monotonous under coarse embedding between amenable groups. This has various applications, among which the existence of an uncountable family of $3$-solvable groups which pairwise do not coarsely embed into one another. 
	
	On the flexibility side, we construct explicit orbit equivalences between amenable groups with prescribed integrability conditions. Our main tool is a new notion of Følner tiling sequences. We show in a number of instances that the bounds derived from the isoperimetric profile are sharp up to a logarithmic factor. We also deduce from this study that two important quasi-isometry invariants are {\it not} preserved under $\LL^1$ orbit equivalence: the asymptotic dimension and finite presentability.
\end{abstract}
	\tableofcontents

\section{Introduction}

Gromov introduced measure equivalence between countable groups as a measured analogue of quasi-isometry. A classical instance of a pair of measure equivalent groups is given by lattices in a common locally compact group. Another source of examples is given by orbit equivalent groups. Recall that two groups $\Gamma$ and $\Lambda$ are orbit-equivalent if they admit free measure-preserving actions on a same standard probability space $(X,\mu)$ which share the same orbits, i.e.\ such that for almost every $x\in X$, $\Gamma\cdot x=\Lambda\cdot x$.

The notion of measure equivalence has been extensively studied over the past 20 years, and we refer the reader to \cite[\S 2]{gaboriauExamplesGroupsThat2005} for an overview of its main properties as well as its tight connections with invariants such as cost or $\ell^2$ Betti numbers.  Various rigidity phenomenons have also been uncovered. 
A famous example is Furman's superrigidity results for lattices in higher rank semi-simple Lie groups \cite{furmanGromovMeasureEquivalence1999}, which implies for instance that any countable group which is measure-equivalent to a lattice in $\PSL_3(\R)$ is commensurable up to finite kernel to another lattice in $\PSL_3(\R)$. Another example is provided by Kida's work on mapping class groups of surfaces: he showed that most surfaces can be reconstructed from the measure equivalence class of their mapping class group \cite{kidaMappingClassGroup2008}, and that every group which is measure equivalent to a mapping class group must actually be commensurable up to finite kernel to it \cite{kidaMeasureEquivalenceRigidity2010}.

In the opposite direction of flexibility, a  celebrated result of Ornstein and Weiss  implies that all infinite countable amenable groups are orbit equivalent and hence measure equivalent \cite{ornsteinErgodicTheoryAmenable1980}.
So most coarse geometric invariants (such as volume growth) are not preserved under orbit equivalence. 
Also, it is known that the class of groups measure equivalent to lattices in $\PSL_2(\R)$ is very diverse and contains groups that are not virtually isomorphic to lattices of the latter (for instance, all free products of infinite amenable groups belong to this class, see \cite[{$\mathbf{P}_{\mathrm{ME}}6$}]{gaboriauExamplesGroupsThat2005}). 
But as we will now see, measure equivalence admits natural refinements which capture meaningful coarse geometric invariants.

Assume for simplicity that we are given an \emph{orbit equivalence coupling} of two finitely generated groups $\Gamma=\la S_\Gamma\ra$ and $\Lambda=\la S_\Lambda\ra$ over a probability space $(X,\mu)$, i.e. two measure-preserving free actions of $\Gamma$ and $\Lambda$ on $(X,\mu)$ which share the same orbits.  Then we can equip the space $(X,\mu)$ with the Schreier graph metrics $d_{S_\Gamma}$ and $d_{S_\Lambda}$, and consider for each $\gamma\in \Gamma$ and each $\lambda\in\Lambda$ the following \emph{distance maps}
$$x\mapsto d_{S_\Lambda}(x,\gamma\cdot x)\text{ and } x\mapsto d_{S_\Gamma}(x,\lambda\cdot x).$$
The fact that the two actions share the same orbits means that these functions do not take the value $+\infty$, in other words they belong to the space $\LL^0(X,\mu,\R)$ of measurable maps
$X\to \R$.
We then say that $\Gamma$ and $\Lambda$ are $\LL^p$ orbit equivalent when all these functions are in $\LL^p(X,\mu;\R)$. A similar definition can be given for measure equivalence, yielding the notion of $\LL^p$ measure equivalence (see Section \ref{sec: integrability conditions}).

In the past ten years, $\LL^1$ measure equivalence has been intensely investigated. In the context of non-amenable groups, Bader, Furman and Sauer showed that any group that is $\LL^1$ measure equivalent to a lattice in SO$(n,1)$ for a given $n\geq 2$ must be virtually isomorphic to another such lattice \cite{baderIntegrableMeasureEquivalence2013}. In particular, there is a $\LL^1$-measure equivalence rigidity phenomenon for lattices in $\PSL_2(\R)$, as opposed to the pure measure equivalence case.   
Somewhat more surprising is a result of Austin in the context of amenable groups: he showed that $\LL^1$ measure equivalent virtually nilpotent groups have bi-Lipschitz asymptotic cones \cite{austinIntegrableMeasureEquivalence2016}. Another important result is due to Bowen, who showed in the appendix of Austin's aforementioned paper that volume growth is invariant under $\LL^1$ measure equivalence.\\  

This paper is in the direct continuation of such results, but we also aim at a deeper understanding by going in two new directions: finer integrability notions and asymmetric versions of measure equivalence.  For this introduction, let us however only give the definitions in the context of orbit equivalence (see Sec. \ref{sec: integrability conditions} for the measure equivalence versions and its asymmetric counterparts). 

Assume again that we are given two orbit equivalent actions of two finitely generated groups $\Gamma=\la S_\Gamma\ra$ and $\Lambda=\la S_\Lambda\ra$. 
Given any two unbounded increasing positive function $\varphi,\psi\colon (0,\infty)\to (0,\infty)$, we say that we have {\it a $(\varphi,\psi)$-integrable orbit equivalence coupling from $\Gamma$ to $\Lambda$} if for each $\gamma\in\Gamma$ and each $\lambda\in\Lambda$, there are constants $c_\gamma,c_\lambda>0$ such that the associated distance functions satisfy the following conditions:
\[\int_{X}\varphi\left(\frac{d_{S_\Lambda}(x,\gamma\cdot x)}{c_\gamma}\right)d\mu(x)<\infty \text{ and }\int_{X}\psi\left(\frac{d_{S_\Gamma}(x,\lambda\cdot x)}{c_\lambda}\right)d\mu(x)<\infty.\]

Two  remarks are in order: since $X$ has finite measure, the $(\varphi,\psi)$-integrability of the distance functions is only sensitive to the speed at which $\varphi$ and $\psi$ tend to infinity. The constants $c_\gamma,c_\lambda>0$ are partly motivated by the fact that  we want  our notion of $(\varphi,\psi)$-integrability to be independent of the choice of generating subset.  We address these technical points  in Section \ref{sec: integrability conditions}.

For consistency with the literature, we will often abbreviate integrability conditions as follows: for $p\in (0,\infty)$ we write $\LL^p$ instead of $(t\mapsto t^p)$-integrable. By extension, $\LL^{\infty}$ means that the distance maps are essentially bounded. 
On the opposite, if no assumption is made on the distance maps we shall write $\LL^0$. 
For instance, an $(\LL^\infty,\LL^0)$  orbit coupling  from $\Gamma$ to $\Lambda$ is an orbit equivalence coupling such that for all $\gamma\in \Gamma$, the function $x\mapsto d_{S_\Lambda}(x,\gamma\cdot x)$ is essentially bounded on $X$.

Saying that two groups are $\LL^p$ orbit equivalent for some $p\geq 1$ means in our terminology that there exists an $(\LL^p,\LL^p)$ orbit equivalence coupling between them. Note that for $p<1$, what one would like to call $\LL^p$ orbit equivalence fails to be a transitive relation, so we will refrain from using such a terminology. Nevertheless, such a notion actually defines a very natural pseudo-distance on the space of all finitely generated groups (see Sec. \ref{secIntro:Space}), and we will see for instance that the Heisenberg group and $\Z^4$ are at distance zero, while the distance between $\Z^n$ and $\Z^m$ is equal to $\abs{\log n-\log m}$.


The rest of this introduction contains a summarized description of our main results, and is organized in three subsections.  Section \ref{secIntro:isop} contains a general rigidity result related to the isoperimetric profile, while Section \ref{secIntro:Constructions} has a flexibility flavor as it deals with constructions of concrete orbit equivalences. 
Finally, in Section \ref{secIntro:Space} we briefly discuss a point of view that emerges from the notion of $(\varphi,\psi)$-integrability (this can be seen as an introduction to the results of Section \ref{sec: prelim}).

\subsection{Monotonicity of the isoperimetric profile (and applications)}\label{secIntro:isop}

\paragraph{The isoperimetric profile.}
Bowen's aforementioned result allows to distinguish groups with different growth functions up to  $\LL^1$ measure equivalence. However this fails to address the case of solvable groups that are not virtually nilpotent, as these have exponential growth. By contrast the 
isoperimetric profile offers a fine invariant of amenable groups with exponential growth. Its asymptotic behavior has notably been computed for a wide class of solvable groups.  Let us start recalling its definition.

Let $\Gamma$ be a finitely generated group, and let 
$S_\Gamma$ be a finite generating subset. Given a function $f \colon \Gamma 
\rightarrow \mathbb{R}$ and $p>0$ we define the $\LL^p$-norm of its left 
gradient by the equation 
\[\norm{\nabla^l_{S_\Gamma} f}_p^p=\sum_{g\in \Gamma, s\in S}|f(sg)-f(g)|^p.\]
Given a function $f\colon \Gamma\to\R$, its support is the set 
$\supp f=\{g\in\Gamma\colon f(g)\neq 0\}$.
Following \cite{coulhonRandomWalksGeometry2000,nowakExactnessIsoperimetricProfiles2007,tesseraLargeScaleSobolev2008a}, we define for every $0< p\leq  \infty$, the $\ell^p$-isoperimetric profile\footnote{Note that in \cite{brieusselSpeedRandomWalks2021}, the isoperimetric profile is defined as $I_{p,\Gamma}=\frac{1}{j_{p,\Gamma}}$.} of $\Gamma$ as
the non-decreasing function
\[j_{p,\Gamma}(n) = \sup_{|\supp f|\leq n} \frac{\norm f_p}{\norm{\nabla^l_{S_\Gamma} 
		f}_p}.\]
Of course this quantity depends on a choice of generating subset, but we will be interested in the asymptotic behavior of $j_{p,\Gamma}$, which does not.

Given two monotonous real-valued functions $f$ and $g$ we say that $f$ is \emph{asymptotically less} than $g$ and write $f\preccurlyeq g$ if there exists a positive constant $C$ such that $f(n)=O(g(Cn))$ as $n\to \infty$. We say that $f$ and $g$ are \emph{asymptotically equivalent} and write $f\approx g$ if $f\preccurlyeq g$ and $g\preccurlyeq f$. The asymptotic behavior of $f$ is its equivalence class modulo $\approx$.  

Recall that for $p=1$, the $\ell^1$ isoperimetric profile 
is simply called the isoperimetric profile because of the following geometric interpretation \cite{coulhonRandomWalksGeometry2000}:
\[j_{1,\Gamma}(n)\approx \sup_{|A|\leq n} \frac{|A|}{|\partial A|},\]
where $\partial A=S_\Gamma A\bigtriangleup A$.
Certain authors prefer to work with the  F\o lner 
function, defined for all $k\geq 1$ as follows
\[\text{F\o l}_\Gamma(k)= \inf\left\{|A|\colon  \frac{|\partial A|}{|A|}\leq 
1/k\right\}.\]
Note that the isoperimetric profile and the F\o lner 
function are generalized inverses of one another.

For $p=2$, the asymptotic behavior of $j_{2,\Gamma}$ is intimately related to the probability of return of the simple random walks on the Cayley graph $(\Gamma,S_{\Gamma})$ as described in \cite{coulhonIsoperimetriePourGroupes1993}.

Note that a group is amenable if and only if its isoperimetric profile is unbounded. So the isoperimetric profile can be interpreted as a measurement of amenability: the faster it tends to infinity, the ``more amenable'' the group is.
This is well illustrated by the following examples. For all $p\geq 1$ we have\footnote{The notion of wreath product of groups (denoted by $G\wr H$) is recalled in Section \ref{sec: Zlamp}.
}: 

\begin{itemize}
	\item $j_{p,\Gamma}(n)\approx n^{1/d}$ for a group of polynomial growth of degree $d$ (e.g.\ for $\Gamma=\Z^d$) (see  \cite{coulhonRandomWalksGeometry2000});
	\item $j_{p,\Gamma}(n)\preccurlyeq \log n$ for all amenable groups with exponential growth \cite{coulhonIsoperimetriePourGroupes1993};
	\item $j_{p,\Gamma}(n)\approx \log n$ for a wide class of solvable groups with exponential growth including those that are polycyclic, or the so-called ``lamplighter'' groups, i.e.\ group of the form $F\wr \Z$, where $F$ is a finite non-trivial group \cite{coulhonGeometricApproachOndiagonal2001,pittetFolnerSequencesPolycyclic1995,pittetIsoperimetricProfileHomogeneous2000,tesseraLargeScaleSobolev2008a};
	\item $j_{1,\Gamma}(n)\approx (\log n)^{1/d}$ for a group of the form $F\wr \Sigma$, where $F$ is a non-trivial finite group and $\Sigma$ has polynomial growth of degree $d$ \cite{erschlerIsoperimetricProfilesFinitely2003}.
\end{itemize}

\paragraph{A general monotonicity result.}

Our first main result is the following theorem. We refer to Definition \ref{def:phi-couplings} for the notion of $(\varphi,\psi)$-integrable measure equivalence coupling. 

\begin{thmintro}\label{thmintroIsop}
	Assume that there exists $(\varphi,\LL^0)$ measure equivalence coupling from a finitely generated group $\Gamma$ to a finitely generated group $\Lambda$. Then 
	\begin{itemize}
		\item if $\varphi(t)=t^p$ for some $p\geq 1$, then \[j_{p,\Gamma}\succcurlyeq j_{p,\Lambda};\]
		\item if $\varphi$ and $t\mapsto t/\varphi(t)$ are increasing then
		\[
		j_{1,\Gamma}\succcurlyeq \varphi\circ j_{1,\Lambda}.
		\]
	\end{itemize}
\end{thmintro}
In particular, the $\ell^p$-isoperimetric profile is stable under $\LL^p$ 
measure equivalence for all $p\geq 1$. 
As a concrete application we deduce for instance that $\Z\wr \Z$ is not $\LL^1$ measure equivalence to $\Z/2\Z \wr \Z$, nor to any polycyclic group. Indeed, for these groups, the isoperimetric profile is $\succcurlyeq\log n$, while for $\Z\wr \Z$ it is $\approx \log n/\log \log n$ \cite{erschlerIsoperimetricProfilesFinitely2003}.

A special case of interest is when $\Lambda=\Z$: since $j_{1,\Lambda}\approx n$, the second statement of Theorem~\ref{thmintroIsop} says in this case that $\varphi \preccurlyeq j_{1,\Gamma}$. This triggers the following question asking for some kind of converse.

\begin{ques}\label{qu: optimal coupling?}
	Given an amenable finitely generated group $\Gamma$,
	does there exist a $(j_{1,\Gamma},\LL^0)$ 
	orbit equivalence coupling from $\Gamma$ to $\Z$?
\end{ques}
We shall see in the next 
subsection that this holds up to a logarithmic error in various examples.
Let us also mention that since the  F\o lner function and the $\ell^1$-isoperimetric profile are generalized inverses of each other,  Theorem \ref{thmintroIsop} can be reformulated as follows for $p=1$.
\begin{corintro}\label{corintroFoln}
	Assume that there exists $(\varphi,\LL^0)$ measure equivalence coupling from a finitely generated group $\Gamma$ to a finitely generated group $\Lambda$. Then 
	\begin{itemize}
		\item if $\varphi(t)=t$, then \[\text{F\o l}_\Gamma\preccurlyeq \text{F\o l}_\Lambda;\]
		\item if $\varphi$ and $t\mapsto t/\varphi(t)$ are increasing then
		\[ \text{F\o l}_\Gamma \circ \varphi\preccurlyeq \text{F\o l}_\Lambda .\]
	\end{itemize}
\end{corintro}

\begin{remark} In particular we get that if $\Gamma$ and $\Lambda$ are $\LL^1$ measure equivalent, then $\text{F\o l}_\Lambda\approx  \text{F\o l}_\Gamma$.
	This strengthens the following unpublished result obtained by Bowen in 2013: 
	assuming 
	$\Gamma$ and $\Lambda$ are $\LL^1$ orbit equivalent, he proved that there exists $r>1$ such 
	that $\text{F\o l}_\Gamma(k)\leq \text{F\o l}_\Lambda(r^k)$. 
\end{remark}

\paragraph{Application to regular maps.}
Theorem \ref{thmintroIsop} turns out to be a special case of a more general 
monotonicity property of the $\ell^p$-isoperimetric profile (Theorem \ref{thm: 
	monotonicity of isop under Lp} and Theorem \ref{thm: monotonicity of isop under phi}) 
which encompasses its well-known monotonicity under subgroups and quotients \cite[Thm.~1]{tesseraLargeScaleSobolev2008a}. 
Writing down this statement motivated us to introduce and study natural 
``measured generalizations" of the notions of subgroup,  quotient, and subgroup of 
a quotient (see Sec.~\ref{sec: asymmetric couplings}).

Let us mention here a striking application of this 
monotonicity result.  Recall that a regular map $f\colon \Gamma\to \Lambda$ between two finitely generated groups  is 
a Lipschitz map such that $\sup_{\lambda\in \Lambda}|f^{-1}(\{\lambda\})|<\infty$. 
Particular cases of  regular maps are Lipchitz injective maps and 
coarse embeddings. We recall that a coarse embedding is a map $f\colon \Gamma\to \Lambda$ such that there exists two proper non-decreasing functions $\rho_-,\rho_+\colon[0,\infty)\to [0,\infty)$ satisfying \[\rho_-(d_\Gamma(\gamma,\gamma'))\leq d_\Lambda(f(\gamma),f(\gamma'))\leq \rho_+(d_\Gamma(\gamma,\gamma')),\]
for all $\gamma,\gamma'\in \Gamma$. Note that since the distance on $\Lambda$ is defined as the Cayley graph distance associated to some generating set, we get that such a map is actually 
$\rho_+(1)$-Lipschitz. Moreover, for all $\lambda\in \Lambda$, $f^{-1}(\{\lambda\})$ is contained a ball of radius $r$, for any $r$ such that $\rho_-(r)>0$:  hence a coarse embedding is duly a regular map.

A trick essentially due to Shalom allows us to interpret the existence of a regular map from an amenable group $\Gamma$ to a group $\Lambda$ as ``$\Gamma$ being an  $\LL^{\infty}$ measure subgroup of $\Lambda$'' (see Sec.~\ref{sec: integrability conditions}). This enable us to deduce from Theorem~\ref{thm: monotonicity of isop under Lp}  the following rather 
unexpected fact.

\begin{thmintro}[see Cor. \ref{cor:iso profile monotonous reg 
		emb}]\label{thmintroIsopregular}
	For every $1\leq p\leq \infty$, the $\ell^p$-isoperimetric profile is monotonous under regular maps between amenable groups. 
\end{thmintro}
The assumption that the groups are amenable is not an artefact of the proof: the statement becomes false in general if the embedded group is not amenable.  For instance the group $\Z/2\Z\wr \Z$ admits a bi-Lipschitz embedded 4-regular tree (see for instance \cite{Woess-LAMPLIGHTERS_HARMONIC_FUNCTIONS}). Since the 4-regular tree can be seen as the standard Cayley graph of the free group  $F_2$ on two generators, we deduce that $F_2$ bi-Lipschitz embeds into $\Z/2\Z\wr \Z$. But $F_2$ is non-amenable and therefore its isoperimetric profile is bounded, while $\Z/2\Z\wr \Z$ is obviously amenable, so its isoperimetric profile tends to infinity.

Only very few geometric invariants are known to be monotonous under coarse embedding: these are the volume growth, the asymptotic dimension, the separation profile \cite{benjaminiSeparationProfileInfinite2012}, and more recently the Poincaré profiles \cite{humePoincareProfilesGroups2020}. All of these turn out to be monotonous under regular map as well (see \cite{benjaminiSeparationProfileInfinite2012} for the case of asymptotic dimension). For solvable groups of exponential growth, the asymptotic dimension is generally infinite, and the Poincaré profiles have only been computed in very specific examples.   By contrast, our theorem provides us with a powerful obstruction. As a concrete application, it prevents the existence of a regular map from  $\Z\wr \Z^a$ to $\Z\wr \Z^b$ as soon as $a>b$ (note that these groups have infinite asymptotic dimension). Combining it with a construction of Erschler and Zheng \cite[Cor. 3.3]{erschlerIsoperimetricInequalitiesShapes2017} (see also \cite[Thm. 1.1]{brieusselSpeedRandomWalks2021}) we obtain the following result.

\begin{thmintro}[see Cor. \ref{cor:3solvable}]\label{thm:continuumsolvable}
	There exists an uncountable family of 3-step solvable groups $\Gamma_i$ that do not pairwise regularly embed into one another, and that are pairwise non $\LL^1$ measure equivalent. Moreover, one can assume that these groups all have asymptotic dimension one.
\end{thmintro}
We note that finding an uncountable family of groups that do not pairwise regularly embed into one another can be done by other means.  Indeed, one can produce an uncountable  familiy of groups of intermediate growth whose growth functions are pairwise incomparable (this can be done exploiting \cite[Thm.~7.1]{grigorchukDegreesGrowthFinitely1985}, or more directly invoking \cite[Thm.~7.2]{brieusselGrowthBehaviorsRange2014}). Moreover by Lewis Bowen's theorem (see Theorem \ref{thm:Bowen}), these groups  are pairwise non $\LL^1$ measure equivalent.
Other examples of uncountable families of groups that  pairwise do not coarsely embed into one another are due to Hume \cite[Thm.~1.2]{humeContinuumExpanders2017}: they involve (non-amenable) groups that contain isometrically embedded families of expanders. 
However none of these other methods can yield families of solvable (nor even elementary amenable) groups.

\paragraph{No quantitative version of Orstein-Weiss.}

In a recent groundbreaking work, Brieussel and Zheng managed to construct amenable groups with prescribed isoperimetric profiles \cite{brieusselSpeedRandomWalks2021}. In particular, they show that for every increasing unbounded function $\varphi$, there exists an amenable group whose isoperimetric profile does not dominate $\varphi$. Combining this with Theorem \ref{thmintroIsop}, we deduce that Orstein-Weiss' orbit equivalence coupling between amenable groups can be as poorly integrable as possible. More precisely, we obtain the following corollary.

\begin{corintro}[{see Sec.~\ref{sec:OW}}]\label{corintroOrnstein}
	For every amenable finitely generated group $\Lambda$ and every increasing unbounded function $\varphi$, there exists an amenable finitely generated group $\Gamma$ such that 
	there are  no $(\varphi,\LL^0)$ measure equivalence coupling from $\Gamma$ to $\Lambda$.
\end{corintro}


\paragraph{Stability of Liouville property among lamplighter groups.}
Let $G$ be a countable infinite group and let $\mu$ be a probability measure on $G$ whose
support generates $G$ (as a group). A function $f:G\to \mathbb{R}$ is called harmonic if for all $g\in G$, \[f(g) =\sum_{g'\in G}\mu(g g').\]
If all bounded harmonic functions are constant, then we say that $(G, \mu)$ has the Liouville property. 
In what follows, to simplify the discussion, we shall say that a measure $\mu$ is admissible if it is symmetric, and its support is finite and generates $G$. We refer to \cite{kaimanovichRandomWalksDiscrete1983} for an introduction to this topic.

A well-known open question due to Itai Benjamini is whether Liouville property is ``invariant under quasi-isometry''. More precisely
does there exist $(G_1,\mu_1)$ and $(G_2,\mu_2)$ such that the measures $\mu_i$ are admissible, the groups $G_i$ are quasi-isometric, and $(G_1,\mu_1)$ is Liouville while $(G_2,\mu_2)$ is not? 
Intuitively, being Liouville is a property of ``small groups'', so we might even expect it to be stable under coarse (or even regular) embeddings. However, this is false in general as shown by the following counterexample: being non-amenable, the free group on two generators is non-Liouville (with respect to any $\mu$), while $(\Z/2\Z\wr \Z,\mu)$ is Liouville for any admissible measure $\mu$ (see \cite{kaimanovichRandomWalksDiscrete1983}). But as noticed earlier, the free group coarsely embeds into $\Z/2\Z\wr \Z$. 

By contrast, Theorem \ref{thm: monotonicity of isop under Lp} can be used to prove that being Liouville is stable under regular embedding (and under $L^1$ measure equivalence) among lamplighter groups. More precisely, we obtain:

\begin{theorem}[See Thm.~\ref{thm:SUBQLiouville}]\label{thm:CoarseLiouville}
	For $i=1,2$, let $F_i$ be a non-trivial finite group, and let $H_i$ be a finitely generated group. Let $\mu_i$ be an admissible probability measure on $\Gamma_i=F_i \wr H_i$. Assume either that 
	\begin{itemize}
		\item $\Gamma_1$ regularly embeds into $\Gamma_2$, 
		\item or there exists a $(L^1,L^0)$ measure equivalence equivalence coupling from $\Gamma_1$ to $\Gamma_2$. 
	\end{itemize}
	Then if $(\Gamma_2,\mu_2)$ is Liouville, then so is $(\Gamma_1,\mu_1)$.
\end{theorem}

\subsection{Construction of measure couplings between amenable groups}\label{secIntro:Constructions}

\paragraph{F\o lner tilings.}

Most of the known examples of $\LL^p$ measure equivalence couplings come from lattices in semi-simple Lie groups.  Here we introduce a general method for constructing orbit equivalent couplings with prescribed integrability conditions between amenable groups. Our main tool is a notion of Følner tiling, developed in Section \ref{section:Folner tilings}. A Følner tiling sequence in a group $\Gamma$ is a sequence of finite subsets $T_n$ satisfying conditions:
\begin{enumerate}[(i)]
	\item $(T_n)$ is a left Følner sequence for $\Gamma$,
	\item each $T_n$ is obtained as a union of right translates of $T_{n-1}$.
\end{enumerate} 
We note that Følner sequences $T_n$ such that $\Gamma$ can be obtained as a 
disjoint union of right translates of $T_n$ do exist in all elementary amenable 
groups and all residually finite amenable groups thanks to the work of Weiss 
\cite{weissMonotileableAmenableGroups2001}. Our second condition appears to be 
new, and is crucial in our construction. An easy example of a Følner tiling 
sequence is provided by $T_n=\{1,...,2^n\}$ in $\Z$.

Starting from a Følner tiling sequence for a group $\Gamma$, we construct a free measure preserving action of $\Gamma$ on the infinite product space $\prod_n (F_n,\nu_n)$, where $F_n$ is a sequence of finite subsets of $\Gamma$ such that $T_{n}=\bigsqcup_{\gamma\in F_{n}}T_{n-1}\gamma$, and where $\nu_n$ is the renormalized counting measure on $F_n$. We show that 
the equivalence relation generated by the action is the co-finite equivalence relation (in the above example the action is given by the dyadic odometer): two sequences are equivalent if they coincide except possibly on finitely many indices. 
It follows from our construction that if another group $\Lambda$ admits a Følner tiling sequence whose basic tiles $T'_n$ satisfy $|T_n|=|T'_n|$, then the corresponding action of $\Lambda$ is orbit equivalent to that of $\Gamma$. Moreover, the degree of integrability of this orbit equivalence is controlled by the properties of the two Følner sequences $(T_n)$ and $(T'_n)$ (see Proposition \ref{prop:tilingOEquantitative}).


Theorem~\ref{thmintroIsop} provided us with an obstruction for finding $\varphi$-integrable couplings with certain functions $\varphi$ between two amenable groups. We will now show that in many cases of interest (and especially in the case of couplings from certain amenable groups to $\Z$) this obstruction is close to being optimal.

Let us start with the case of groups with polynomial growth. We deduce from a straightforward extension of  Bowen's theorem on volume growth (see Corollary \ref{cor:BowenPolynomial}), or alternatively from the monotonicity of the isoperimetric profile (Corollary \ref{cor:isoperimetricPolynomialGrowth}), that $\Z^m$ and $\Z^{n}$ do not admit $\LL^p$ measure equivalence couplings for $p>m/n$ if $n<m$. We show that this threshold is sharp:
\begin{thmintro}[see Thm. \ref{thm:OE between Zd and Zd'}]\label{thmi: Zn Zm}
	For every pair of positive integers $n<m$, there exists an orbit equivalence coupling between $\Z^m$ and $\Z^n$ that is $\LL^p$ for every $p<n/m$.
\end{thmintro}
Let us mention that these couplings between $\Z^m$ and $\Z^n$ were independently discovered by Matthieu Joseph (unpublished result). 

Note that the growth function of $\Z^n$ being slower than that of $\Z^m$, we expect more constraint on the integrability condition from $\Z^m$ to $\Z^n$ than from $\Z^n$ to $\Z^m$.
This is reflected by the more precise statement of Theorem \ref{thm:OE between Zd and Zd'}: the orbit equivalence coupling which is constructed from $\Z^m$ to $\Z^n$ is $(\LL^p,\LL^{p'})$ for all $p<n/m$ and $p'<m/n$. 

Finally, we point out that the $L^p$-integrability from $\Z^m$ to $\Z^n$ is almost sharp, as Corollary \ref{cor:BowenPolynomial} implies that there is no $(\LL^p,\LL^0)$ measure equivalence coupling from $\Z^m$ to $\Z^{n}$ for $p>n/m$. The only case that remains unclear is when $p=n/m$. More precisely, 
Theorem \ref{thmi: Zn Zm} leaves the following question open (see also Question \ref{qu: refined optimal coupling for Z^n}).
\begin{ques}\label{qu: optimal coupling for Z^n}
	Let $n<m$, are $\Z^m$ and  $\Z^n$ $\LL^{n/m}$ measure equivalent?
\end{ques}
We may also consider the following asymmetric version of this question: is there a $(\LL^{n/m},\LL^0)$ measure coupling from $\Z^m$ to  $\Z^n$?

Regarding groups with same degree of polynomial growth, Austin proved that $\LL^1$ 
measure equivalent groups of polynomial growth have bi-Lipschitz asymptotic 
cones \cite{austinIntegrableMeasureEquivalence2016}. Combined with a famous 
theorem of Pansu \cite[Thm. 
3]{pansuMetriquesCarnotCaratheodoryQuasiisometries1989}, this implies for 
instance that a non-virtually abelian nilpotent group cannot be $\LL^1$ measure 
equivalent to an abelian group. The following result shows that this rigidity statement cannot be extended to $\LL^p$ for $p<1$.
\begin{thmintro}[{see Thm. \ref{thm:OE between Heis and Z4}}]\label{thmi: 
		heisenberg}
	There exists an orbit equivalence coupling between $\Z^4$ and $\Heis(\Z)$ that is $\LL^p$ for every $p<1$.
\end{thmintro}

In this paper, the actions we construct by means of Følner tiling sequences turn out to be (explicit) profinite actions. To the best of our knowledge, this way of producing orbit equivalences between amenable groups is new, and quite different from the much more abstract quasi-tiling machinery of Ornstein and Weiss \cite{ornsteinEntropyisomorphismtheorems1987}.  Note that latter allows them to prove that \emph{all} free measure-preserving ergodic actions of \emph{all} amenable groups are orbit equivalent. This brings us to the complementary problem of studying quantitatively orbit equivalence between actions of a fixed amenable group. This is obviously a vast program, so we shall  simply mention the following natural question, which is already interesting for the group $\Z$.
\begin{ques}
	Let $\Gamma$ be an amenable finitely generated group.
	Does there exist a single unbounded increasing function $\varphi$ such that any two free measure-preserving ergodic actions actions of $\Gamma$ are $\varphi$ orbit equivalent? 
\end{ques}

\paragraph{Orbit equivalence and wreath products.}
Many interesting solvable groups can be constructed by means of (iterated) wreath products. By the work of Erschler  \cite{erschlerIsoperimetricProfilesFinitely2003}, their isoperimetric profile can be estimated precisely. Hence those groups offer a useful playground to test the optimality of the obstruction provided by Theorem \ref{thmintroIsop}. In Section \ref{sec:wreathproduct}, we define a natural notion of ``wreath product of orbit equivalence couplings''. This construction allows us to obtain orbit equivalence couplings between wreath product of groups whose integrability conditions are best possible (see Corollary \ref{cor: wreath product} for a precise statement). 

In this paragraph, we will focus on some concrete applications of this general result.

\begin{thmintro}\label{thmintro:LamplighteroverZ^d}
	For every pair of positive integers $a<b$, and every non-trivial finite group $F$ there exists an $(\LL^p,\LL^{1/p})$ orbit equivalence coupling from $F\wr \Z^b$ to $F\wr\Z^a$ for every $p<a/b$. 
\end{thmintro}
We refer to Corollary \ref{cor:iteratedwreathZ^ab} for an iterated version of this result.
By \cite{erschlerIsoperimetricProfilesFinitely2003}, the isoperimetric profile of $F\wr \Z^a$ is $\approx (\log n)^{1/a}$.
Hence we deduce from the second part of Theorem \ref{thmintroIsop} that the existence of a $(\LL^p,\LL^0)$ measure equivalence coupling from $F\wr \Z^a$ to $F\wr\Z^b$ forces the inequality $p\leq a/b$, which is shown to be (almost) sharp by Theorem \ref{thmintro:LamplighteroverZ^d}.

We now focus on the existence of couplings with $\Z$. We first consider couplings from the lamplighter group to $\Z$. Since the isoperimetric profile of the lamplighter group grows like $\log n$, we know that a $(\varphi,\LL^0)$ coupling from the latter to $\Z$ must satisfy $\varphi(t)\preccurlyeq \log t$. This turns out to be sharp up to a logarithmic error. 
\begin{thmintro}[{see Prop. \ref{prop:couplingZLamp}}]
	For every finite group $F$, there exits an orbit equivalence coupling from $F\wr\Z$ to $\Z$ which is $(\varphi_{\eps},\exp)$-integrable for every $\eps>0$, where $\varphi_\eps(x)=\frac{\log (x)}{\log(\log (x))^{1+\eps}}$.
\end{thmintro}
This result is obtained by means of F\o lner tilings. Once again, the asymmetry of the condition reflects the intuition that $\Z$ is a ``small'' group compared to $F\wr\Z$.

More generally, one can consider iterated wreath products as follows. Let $G_0=\Z$, and for all $n\geq 1$, $G_n=F\wr G_{n-1}$. By \cite{erschlerIsoperimetricProfilesFinitely2003}, we know that the isoperimetric profile of $G_k$ is asymptotically equivalent to $\log^{\circ n}$, namely the $n$-times iteration of $\log$ with itself. Similarly, using our wreath product construction of orbit equivalences, we show that this is almost sharp: see  Corollary \ref{cor:iteratedwreath}.

\paragraph{Exponentially integrable couplings.}

The lamplighter group $\Z/k\Z\wr \Z$ and the Baumslag-Solitar group $\BS(1,k)$ are known to have very similar geometries. Indeed the first one can be seen as a horocyclic product of two $(k+1)$-regular trees  \cite{Woess-LAMPLIGHTERS_HARMONIC_FUNCTIONS}, while the second is quasi-isometric to a horocyclic product of a $(k+1)$-regular tree with a hyperbolic plane \cite[\S 3]{MR1608595}. A consequence of this description is that they have isometric asymptotic cones (this follows from the analysis in \cite[\S 9]{Cornulier_dimcone}). They are also known to have same isoperimetric profile \cite{coulhonGeometricApproachOndiagonal2001} 
(or \cite[Thm.~4]{tesseraLargeScaleSobolev2008a}). However they are not quasi-isometric, as $\BS(1,k)$ is finitely presented, while $\Z/k\Z\wr \Z$ is not. 

In Section \ref{sec:FP_Unstable}, we construct a very natural orbit coupling between these groups, which satisfies a very strong integrability condition.
\begin{thmintro}[{see Thm. \ref{thm:Lamplighter et BS}}]
	For every $k\geq 2$, there exists an $(\LL^{\infty},\exp)$ orbit equivalence coupling from $\Z/k\Z \wr \Z$ to $\BS(1,k)$. 
\end{thmintro}
The asymmetry is not an artifact of the proof: indeed we could not have an $(\LL^{\infty},\exp)$-orbit equivalence coupling  from $\BS(1,k)$ to $\Z/k\Z\wr \Z$  as the asymptotic dimension is monotonous under $(\LL^{\infty},\LL^0)$ orbit equivalence (Corollary \ref{cor:asdimLinfty}), and    $\BS(1,k)$ has asymptotic dimension 2, while $\Z/k\Z\wr \Z$ has asymptotic dimension 1. 

The previous discussion also shows that the asymptotic dimension is not preserved under $\LL^p$ orbit equivalence for all $p<\infty$. Considering slightly more sophisticated examples, we similarly show that the finiteness of the asymptotic dimension is not preserved either.


\subsection{A sort of ``asymmetric distance" between finitely generated groups}\label{secIntro:Space}

We may view the functions $\varphi$  and $\psi$ for which there exists a $(\varphi,\psi)$ measurable coupling from a group $\Gamma$ to a group $\Lambda$ as an asymmetric way of quantifying how geometrically close these groups are from one another. 
This is of course only interesting within a class of groups that are measure equivalent, such as lattices (uniform and non-uniform ones) in a given locally compact group, or among amenable groups. 

Let us illustrate this point of view in the following special case: given two finitely generated groups $\Gamma$ and $\Lambda$, let $\alpha(\Gamma,\Lambda)$ be the infimum  of $-\log p$ over all  $p\leq 1$ such that there exists an $(\LL^p,\LL^p)$ integrable measure equivalence coupling between  $\Gamma$ and $\Lambda$. It follows from Proposition \ref{prop:CombineIntCouplingsConcave} that $\alpha$ satisfies the triangle inequality. Therefore, it defines a pseudo-distance between (isomorphism classes of) finitely generated groups. Note that Theorem~\ref{thmi: heisenberg} yields $\alpha(\Z^4,\operatorname{Heis}(\Z))=0$, while Theorem~\ref{thmi: Zn Zm} and the remark that precedes imply $\alpha(\Z^n,\Z^m)=\abs{\log n-\log m}$.

It is interesting to note that for $p\geq 1$ (and more generally for convex functions $\varphi$), things behave differently, as admitting an $(\LL^p,\LL^p)$ integrable measure equivalence coupling is an equivalence relation. This means that balls of radius $0$ for $\alpha$ can be equipped with an ultrametric distance $\beta(\Gamma,\Lambda)$ defined  for instance by the infimum of $1/p$ over all  $p\geq1$ such that the groups are $\LL^p$ measure equivalent. In other words, we can distinguish two scales: a ``large" scale measured by the pseudo distance $\alpha$, and a finer scale, measured by  the ultrametric pseudo-distance $\beta$.

More generally, we will see in Section \ref{sec:compositioncouplings} that if the functions $\varphi$ and $\psi$ are concave, then $(\varphi,\psi)$-integrable measurable equivalence couplings satisfy a natural composition rule, which we may view as analogous to a ``triangle inequality". 
By contrast, for functions that grow faster than any polynomial, e.g.\ $t\to \exp (t^a)$ for $a>0$, there is no clear composition rule. This led us to introduce a stronger version of the integrability condition (which is automatic for $\LL^p$) to make it well-behaved under composition: see Definition \ref{def:strong}. As in the $L^p$ case discussed above, we then obtain a composition rule that is analogous to an ``ultrametric inequality" when the functions are at least linear.

As a matter of fact, these composition rules are better stated in an even more asymmetric situation, where the measure coupling is no longer a measure {\it equivalence} coupling. If we drop on one side the  finiteness of the measure of the fundamental domain, we obtain a natural notion of measured subgroup, while if we drop the freeness for one of the actions, we obtain a measurable notion of quotient. Combining them, we get a measurable notion of sub-quotient. In these situations, an integrability condition only makes sense in one direction, so that we obtain notions of $\varphi$-integrable measurable subgroups, quotients or sub-quotients, which we develop in the next section. To see how these notions can be useful, note that Theorem~\ref{thmintroIsop} is a corollary of the monotonicity of the isoperimetric profile under $\LL^1$ measure quotient, while Theorem \ref{thmintroIsopregular} involves the notion of $\LL^{\infty}$ measure subgroup. 

\subsection{Plan of the paper}

In Section \ref{sec: prelim}, we set the general framework of this paper. 
More precisely, in Section \ref{sec:smooth} we introduce new notions of measured 
subgroups, quotients and sub-quotients that are defined via asymmetric couplings 
where the actions of both groups are smooth but where the freeness or the 
finite covolume assumption are dropped for one of the actions. Such notions 
satisfy natural composition rules that are gathered in Proposition 
\ref{prop:CombineCouplings}. In Section \ref{sec: integrability conditions}, we make 
these notions quantitative by adding integrability conditions. The behavior of 
these integrability conditions under composition is studied in Section 
\ref{sec:compositioncouplings}. Finally,  Section \ref{sec: OE couplings} is 
dedicated to similar notions in the orbit equivalence setting.

In Section \ref{sec:Bowen}, we prove a general monotonicity result for the 
volume growth under measured sub-quotients satisfying certain integrability 
conditions. The proof is largely inspired (although different) from Bowen's 
original proof of the invariance of volume growth under $\LL^1$ measure 
equivalence. 

Section \ref{sec:isoperimetricprofile} focuses on two monotonicity results 
regarding the isoperimetric profile, namely Theorems \ref{thm: monotonicity of 
	isop under Lp} and \ref{thm: monotonicity of isop under phi}, which generalize 
Theorem~\ref{thmintroIsop}. In Section \ref{sec:overviewisop} a quick overview of 
their proofs is given. Sections \ref{sec:GammaGradient} to \ref{sec:L1 measure 
	quotient iso profile down} are dedicated to the actual proofs of these 
theorems. This section ends with Section \ref{sec:OW} where we prove Corollary 
\ref{corintroOrnstein}. 

In Section \ref{sec:Linfty}, we show that the existence of a regular map from an amenable finitely generated group $\Gamma$ to a finitely generated group $\Lambda$ can be used to construct an $\LL^{\infty}$ subgroup coupling from $\Gamma$ to $\Lambda$. This  allows us to deduce Theorem~\ref{thmintroIsopregular} from 
Theorem~\ref{thm: monotonicity of isop under Lp}. We then exploit this result  
in Section \ref{sec:continuum3solvable} to establish Theorem~
\ref{thm:continuumsolvable}.

In Section \ref{sec:Ftiling} we show how  F\o lner tilings can be used to 
produce orbit equivalent probability measure-preserving actions of groups with a control on the 
integrability properties of the corresponding cocycles. In Section 
\ref{sec:polynomialgrowth} we give applications to groups with polynomial 
growth, and in Section \ref{sec: Zlamp} we treat the case of $\Z$ and the 
Lamplighter group. 

In Section \ref{sec:wreathproduct}, we study the behavior of orbit equivalence 
under 
wreath products, and we show that the integrability conditions pass through. 
This enables us to construct couplings between $\Z$ and iterated wreath 
products with nearly sharp integrability conditions.

Section \ref{sec: unstable} deals with properties that are ``very" unstable.
In Section \ref{sec:FP_Unstable}, we construct a very explicit exponential orbit equivalence coupling between $BS(1,n)$ and $\Z/n\Z\wr \Z$, proving that finite presentability is not preserved. Section \ref{cor:finitePresentabilityUnstable}, we construct another explicit coupling between two solvable groups, one with finite asymptotic dimension and the other not. 

Appart from 
Section \ref{sec: prelim}, which contains all the relevant notions used throughout the paper, each other section can be read independently.

\begin{ack} We thank Matthieu Joseph for many helpful conversations around this 
	project. 
	We also thank Yves 
	Cornulier and Amandine Escalier for their useful remarks on a preliminary version of this work. We
	thank Lewis Bowen for 
	sharing with us his unpublished note on the behavior of the F\o lner function 
	under $\LL^1$ orbit equivalence. Finally, we are  indepted to the anonymous referee for their numerous remarks and corrections which greatly improved the quality of the exposition.
\end{ack}

\section{Variations on measure equivalence}\label{sec: prelim}

\begin{convention}
	Throughout the paper, we allow metrics and pseudo-metrics to take the value $+\infty$. 
\end{convention}

\subsection{Smooth actions and fundamental domains}\label{sec:smooth}


A \textbf{standard Borel space} $(\Omega, \mathcal B(\Omega))$ is a measurable space whose $\sigma$-algebra $\mathcal B(\Omega)$ consists of the the Borel subsets coming from some Polish (separable and completely metrizable) topology on $\Omega$. The elements of $\mathcal{B}(\Omega)$ are called \textbf{Borel}, and so are called measurable maps between standard Borel spaces.
A \textbf{standard measure space} is a standard Borel space $(\Omega,\mathcal B(\Omega))$ equipped with a nonzero $\sigma$-finite measure $\mu$ on $\mathcal B(\Omega)$. We simply denote such a space by $(\Omega,\mu)$.  A Borel subset $\Omega_0 \subseteq \Omega$ is said to be \textbf{conull} or of \textbf{full measure} if $\mu(\Omega \setminus \Omega_0)=0$ and a property that holds for all $\omega \in \Omega_0$ is said to hold \textbf{almost surely} or \textbf{for almost every} $\omega\in\Omega$. The space $(\Omega,\mu)$ is said to be a standard probability space if $\mu(\Omega)=1$.

A \textbf{measure-preserving action} of a discrete countable group $\Gamma$ on $(\Omega,\mu)$, for short $\Gamma \curvearrowright (\Omega,\mu)$, is a $\Gamma$-action on $\Omega$ such that the action map $(\gamma,x)\mapsto \gamma\cdot x$ is Borel and that $\mu( \gamma\cdot  E) = \mu(E)$ for all $\gamma \in \Gamma$ and all Borel $E\subseteq \Omega$. A measure-preserving action on a standard probability space is said to be probability measure-preserving. 

Since the groups that we are dealing with are countable, 
if we are given $\Gamma\curvearrowright(\Omega,\mu)$,
then every full measure Borel subset $A\subseteq \Omega$ contains a $\Gamma$-invariant full measure Borel subset, 
namely the set $\bigcap_{\gamma\in \Gamma} \gamma\cdot A$. Moreover, since any full measure Borel subset of a standard measure space is a standard measure space, 
there will be no harm in considering that some properties that hold almost everywhere actually hold everywhere, so we will often do so without explicitly restricting to a full measure Borel subset. 
For instance, one says that a measure-preserving $\Gamma$-action on $(\Omega,\mu)$ is \emph{free} if 
for all $\gamma\in\Gamma\setminus\{e_\Gamma\}$, we have $\gamma\cdot x\neq x$ for almost every $x\in \Omega$. We may as well assume that the latter holds for \emph{every} $x\in \Omega$ since this becomes true once we restrict the action to a full measure $\Gamma$-invariant Borel set.

Finally, given $\Gamma\act(\Omega,\mu)$, the \textbf{full pseudo-group} of the action  $\Gamma\act\Omega$ is the set of all partially defined Borel bijections $\varphi\colon A\to B$, where $A$ and $B$ are Borel subsets of $\Omega$, such that for all $x\in A$, we have $\varphi(x)\in\Gamma\cdot x$. Every such element is measure-preserving, see 
\cite[Prop. 2.1]{kechrisTopicsOrbitEquivalence2004}. 

\begin{definition} 
	A \textbf{fundamental domain} for $\Gamma \curvearrowright (\Omega,\mu)$ is a Borel set $X_{\Gamma} \subseteq \Omega$ which intersects almost every $\Gamma$-orbit at exactly one point: there is a full measure $\Gamma$-invariant Borel set $\Omega_0\subseteq \Omega$ such that for all $x\in \Omega_0$ we have that the intersection of the $\Gamma$-orbit of $x$ with $X_\Gamma$ is a singleton
\end{definition}

Equivalently, a Borel set $X_{\Gamma}\subseteq\Omega$ is a fundamental domain if and only if after throwing away a null set, the quotient map $\pi\colon X_{\Gamma}\to \Omega/\Gamma$ is a bijection.
Note that since a fundamental domain for $\Gamma \curvearrowright (\Omega,\mu)$ intersects almost every orbit, the union of its translates has full measure, so every fundamental domain must have strictly positive measure whenever this is the case with $\Omega$.	

Moreover, the existence of a Borel set $X_\Gamma\subseteq\Omega$ that intersects every $\Gamma$-orbit exactly once is equivalent to the fact that $\Omega/\Gamma$ is standard Borel, or in other words that the quotient map $\pi\colon X_\Gamma\to \Omega/\Gamma$ is a Borel bijection between standard Borel spaces (see \cite[Prop.~6.3]{kechrisTopicsOrbitEquivalence2004}). In this case, the orbit equivalence relation is said to be smooth, and we make the following definition. 

\begin{definition}
	A measure-preserving action of a countable group $\Gamma$ on a standard measured space $(\Omega,\mu)$ is \textbf{smooth} if it admits a fundamental domain. 
\end{definition}

Given a smooth action $\Gamma\act(\Omega,\mu)$, if $X$ is a fundamental domain for the $\Gamma$-action, we denote by $\pi_X$ the map which takes (almost) every $\omega\in\Omega$ to the unique element of the $\Gamma$-orbit of $\omega$ which belongs to $X$. Observe that by definition, if $X_1$ and $X_2$ are two measure-fundamental domains, then the restriction of $\pi_{X_1}$ to $X_2$ is an element of the full pseudo-group of the $\Gamma$-action whose inverse is the restriction of $\pi_{X_2}$ to $X_1$. 

In particular, $X_1$ and $X_2$ have the same measure, and so given any smooth action $\Gamma\act(\Omega,\mu)$, we can unambiguously endow the quotient space $\Omega/\Gamma$ with the measure obtained by identifying it with one of the fundamental domains. We will still denote this measure by $\mu$. 

Finally, given a fundamental domain $X$, we denote by $\iota_{X}$ the inverse of the projection map $X\to \Omega/\Gamma$.

\begin{convention}
	We shall use the notation ``$ \gamma\actom  x$" instead of ``$ \gamma\actx x$" for smooth $\Gamma$-actions. This distinction will prove useful later on since we will also have induced actions on fundamental domains, which we will denote by $\cdot$.
\end{convention}

\subsection{Asymmetric couplings}\label{sec: asymmetric couplings}
We start by recalling the notion of measure equivalence coupling, introduced by Gromov.
\begin{definition}
	Let $\Gamma$ and $\Lambda$ be countable groups, a \textbf{measure equivalence coupling} from $\Gamma$ to $\Lambda$ is a quadruple $(\Omega, X_\Gamma, X_\Lambda,\mu)$, where $(\Omega,\mu)$ is a standard Borel measure space equipped with commuting measure-preserving smooth $\Gamma$ and $\Lambda$ actions such that 
	\begin{enumerate}[(i)]
		\item both the $\Gamma$-action and the $\Lambda$-action are free;
		\item $X_\Lambda$ (resp.\ $X_\Gamma$) is a fixed fundamental domain for the $\Lambda$-action (resp.\ for the $\Gamma$-action);
		\item $X_\Lambda$ and $X_\Gamma$ have finite measures.
	\end{enumerate}
	When there exists such a coupling, we say that $\Gamma$ and $\Lambda$ are \textbf{measure equivalent}.
\end{definition}

The most basic example of a measure equivalence coupling is the following. Given 
a countable group $\Lambda$ viewed as a standard measured space equipped with the counting measure $c$,
we take $\Gamma=\Lambda$ and we have a measure equivalence coupling between $\Gamma$ and $\Lambda$
by making $\Gamma$ act on $\Lambda$ by left translation, and $\Lambda$ act on itself by right translation. A finite measure fundamental domain for both actions is given by $X_\Gamma=X_\Lambda=\{e\}$, so $(\Gamma,\{e_\Lambda\},\{e_\Lambda\},c)$ is a measure equivalence coupling
between $\Gamma$ and $\Lambda$. Motivated by this example, we see that if $\Gamma$ was only an \emph{infinite index subgroup} of $\Lambda$, then the action cannot have a finite measure fundamental domain,
although it has an infinite measure fundamental domain.
We thus first relax the notion of measure equivalence coupling as the following asymmetric version.

\begin{definition}
	Let $\Gamma$ and $\Lambda$ be countable groups.
	A \textbf{measure subgroup coupling} from $\Gamma$ to $\Lambda$ is a triple $(\Omega, X_\Lambda,\mu)$, where $(\Omega,\mu)$ is a standard Borel measure space equipped with commuting measure-preserving smooth  $\Gamma$ and $\Lambda$ actions such that  
	\begin{enumerate}[(i)]
		\item both actions are free;
		\item $X_\Lambda$ is a fixed fundamental domain for the $\Lambda$-action;
		\item $X_\Lambda$ has finite measure.
	\end{enumerate}
	When there exists such a coupling, we say that $\Gamma$ is a \textbf{measure subgroup} of $\Lambda$.
\end{definition}
As explained above, the basic example of a measure subgroup
coupling is  obtained when $\Gamma\leq\Lambda$: it is the triple $(\Lambda,\{e_\Lambda\},c)$ making $\Gamma$ act by left-translation on $\Lambda$ and $\Lambda$ by right-translation on itself.

\begin{remark} 
	Being symmetric in $\Gamma$ and $\Lambda$, measure equivalence couplings require the data of two fundamental domains $X_\Gamma$ and $X_\Lambda$. By contrast, for the above asymmetric notion
	and its latter variants, only an explicit fundamental domain for the $\Lambda$-action is required.  
	The latter will be crucial when making these notions quantitative (see Section \ref{sec: integrability conditions}).
\end{remark}

Let
us now tackle the case when $\Gamma$ is a quotient of $\Lambda$. Here we can make $\Gamma$
act on itself by left translation, and $\Lambda$ act on $\Gamma$ by right translation via the quotient
map. Compared to measure subgroup couplings, we have lost the freeness of the $\Lambda$-action, but we gained that
$\Gamma$ has a finite measure fundamental domain, arriving at the following other asymmetric definition.

\begin{definition}
	Let $\Gamma$ and $\Lambda$ be countable groups. 
	A \textbf{measure quotient coupling} from $\Gamma$ to $\Lambda$ is a triple $(\Omega, X_\Lambda,\mu)$, where $(\Omega,\mu)$ is a standard Borel measure space equipped with commuting measure-preserving smooth  $\Gamma$ and $\Lambda$ actions such that  
	\begin{enumerate}[(i)]
		\item the $\Gamma$-action is free and admits a finite measure fundamental domain;
		\item $X_\Lambda$ is a fixed fundamental domain for the $\Lambda$-action;
		\item $X_\Lambda$ has finite measure.
	\end{enumerate}
	When there exists such a coupling, we say that $\Gamma$ is a \textbf{measure quotient} of $\Lambda$.
\end{definition}

Finally, one can find a common denominator to measure subgroup couplings
and measure quotient couplings as follows.

\begin{definition}
	Let $\Gamma$ and $\Lambda$ be countable groups. 
	A \textbf{measure sub-quotient coupling} from $\Gamma$ to $\Lambda$ is a triple $(\Omega, X_\Lambda,\mu)$, where $(\Omega,\mu)$ is a standard Borel measure space equipped with commuting measure-preserving smooth  $\Gamma$ and $\Lambda$ actions such that  
	\begin{enumerate}[(i)]
		\item the $\Gamma$ action is free;
		\item $X_\Lambda$ is a fixed fundamental domain for the $\Lambda$-action;
		\item $X_\Lambda$ has finite measure.
	\end{enumerate}
	When there exists such a coupling, we say that $\Gamma$ is a \textbf{measure sub-quotient} of $\Lambda$.
\end{definition}

It is a good exercise for the reader to describe how such a coupling arises when
$\Gamma$ is a subgroup of a quotient of $\Lambda$.

\begin{remark} Every measure equivalence coupling yields two measure sub-quotient couplings (one from $\Gamma$ to $\Lambda$, the other one from $\Lambda$ to $\Gamma$) which are both subgroup and quotient  couplings. \end{remark}

To lighten the presentation, and since there will be no ambiguity, we shall frequently remove the word “measure”
in front of the word “coupling” from here on.
Let us end this section by pointing a few nontrivial examples of couplings:

\begin{itemize}
	
	\item We start with a concrete example: let $G$ be a locally compact Polish group with a Haar measure $\mu$, $\Gamma\leq G$ be a discrete subgroup and $\Lambda\leq G$ be lattice. Let $\Lambda$ and $\Gamma$ act respectively by left and right translations on $G$, and let $X_\Lambda$ be a fundamental domain for the $\Lambda$-action. Then $(G,X_\Lambda,\mu)$ is a measure subgroup coupling from $\Gamma$ to $\Lambda.$
	
	\item An important instance of subgroup coupling arises via regular embeddings,
	using amenability. It will be described in Section \ref{sec:Linfty}. 
	
	\item Other examples arising from the notion of orbit equivalence will be given in Section \ref{sec: OE couplings}.
	
	\item 
	Another way to produce more sophisticated examples of couplings is to compose them, as we will see
	now.
\end{itemize}


\subsection{Composition of asymmetric measure couplings}
Our goal here is to extend the (classical) notion of composition of measure equivalence couplings to our asymmetric setting (see for instance \cite{baderIntegrableMeasureEquivalence2013} for compositions of measure equivalence couplings). Starting from a triplet of groups $\Gamma$, $\Lambda$ and $\Sigma$, a coupling from $\Gamma$ to $\Lambda$ and a coupling from $\Lambda$ to $\Sigma$, we would like to define a coupling from $\Gamma$ to $\Sigma$. Such a composition rule should be coherent with those holding for usual notions of subgroups, quotients... For instance we will get that a measure subgroup of a measure subgroup is again a measure subgroup. Similarly, we will obtain that a measure subgroup of a measure quotient is a measure sub-quotient (see Proposition \ref{prop:CombineCouplings}). 
However, it is not true in general that the composition of a measured quotient with another measure quotient remains a measure quotient (see Remark \ref{rmk: prob with quotients}).

\

\noindent{\bf Composition of couplings.}	
Let $\Gamma$, $\Lambda$ and $\Sigma$ be three countable groups  let $(\Omega_1,X_{1,\Lambda},\mu_1)$ be a subgroup coupling from $\Gamma$ to $\Lambda$ and let $(\Omega_2,X_{2,\Sigma},\mu_2)$ be a sub-quotient coupling from $\Lambda$ to $\Sigma$. The \textbf{composition} of these two couplings $(\Omega_3,X_{3,\Sigma},\mu_3)$ is the sub-quotient coupling from $\Gamma$ to $\Sigma$ obtained as follows. 
\begin{itemize}
	\item We consider the diagonal action $\Lambda\act(\Omega_1\times \Omega_2,\mu_1\otimes\mu_2),$ which is smooth and commutes with the $\Gamma$-action on the first coordinate and the $\Sigma$-action on the second coordinate. 
	\item Then the measured space of the coupling is $\Omega_3\coloneqq(\Omega_1\times\Omega_2)/\Lambda$ equipped with the induced $\Gamma$ and $\Sigma$-actions. 
	\item Identifying $\Omega_3$ with a $\Lambda$-fundamental domain of $\Omega_1\times \Omega_2$, we define the measure $\mu_3$ as the restriction of the product measure.
	
	\item We let $X_{3,\Sigma}=\pi_{\Omega_3}(X_{1,\Lambda}\times X_{2,\Sigma})$, where $\pi_{\Omega_3}$ is the projection $\Omega_1\times\Omega_2\to \Omega_3$.
	
	\item
	Furthermore, letting $X_{1,\Gamma}$ (resp.\ $X_{2,\Lambda})$ be fundamental domains for the action of $\Gamma$ on $\Omega_1$ (resp.\ for the action of $\Lambda$ on $\Omega_2$), the subset   
	$X_{3,\Gamma}=\pi_{\Omega_3}(X_{1,\Gamma}\times X_{2,\Lambda})$ of $\Omega_3$ defines a fundamental domain for the action of $\Gamma$. 
\end{itemize}


\begin{proposition}\label{prop:CombineCouplings}
	Let $\Gamma$, $\Lambda$ and $\Sigma$ be three countable groups, let $(\Omega_1,X_{1,\Lambda},\mu_1)$ be a subgroup coupling from $\Gamma$ to $\Lambda$ and let $(\Omega_2,X_{2,\Sigma},\mu_2)$ be a sub-quotient coupling from $\Lambda$ to $\Sigma$. The composition of these two couplings is a sub-quotient coupling from $\Gamma$ to $\Sigma$. 
	
	If both couplings are subgroup (resp.  equivalence) couplings, then their composition is also a subgroup (resp. equivalence) coupling. 
	
	Finally if the coupling from $\Gamma$ to $\Lambda$ is a measure equivalence coupling and the coupling from $\Lambda$ to $\Sigma$ is a quotient coupling, then their composition is a quotient coupling.
\end{proposition}
Before proceeding to the proof, we need to introduce some notation. 
Given a sub-quotient coupling $(\Omega,X_\Lambda,\mu)$ from $\Gamma$ to $\Lambda$, we denote their two commuting actions by $*$. We also have a natural action of $\Gamma$ on $\Omega/\Lambda$ which we call the \textbf{induced action} and denote by $\cdot$. Through the natural identification of $\Omega/\Lambda$ to $X_\Lambda$, this induced action can be described as follows: for all $x\in X_\Lambda$, $\gamma\in\Gamma$,
$$\{\gamma\cdot x\}=\Lambda\actom \gamma\actom x\cap X_\Lambda.$$
Note that the induced $\Gamma$-action defines elements of the pseudo full group of the $\Lambda$-action on $\Omega$, so it is a measure-preserving action. 

\begin{proof}[Proof of Proposition \ref{prop:CombineCouplings}]
	First, the diagonal $\Lambda$-action on $\Omega_1\times \Omega_2$ is smooth because $X_{1,\Lambda}\times \Omega_2$ is a fundamental domain for it. We then define $\Omega_3\coloneqq (\Omega_1\times\Omega_2)/\Lambda$, and we denote by $\mu_3$ the measure induced on $\Omega_3$ from the product measure on $\Omega_1\times \Omega_2$. Finally, we denote by $\star$ the induced actions of both $\Gamma$ and $\Sigma$ on $\Omega_3$.
	Through the identification of $\Omega_3$ with the  fundamental domain $X_{1,\Lambda}\times \Omega_2$ we see that the induced $\Sigma$-action is given by: 
	\[\sigma\star (x,\omega)=(x,\sigma\actom \omega),\]
	for all $(x,\omega)\in X_{1,\Lambda}\times \Omega_2$ and all $\sigma\in\Sigma.$ In particular, we see that $X_{3,\Sigma}\coloneqq\pi_{\Omega_3}(X_{1,\Lambda}\times X_{2,\Sigma})$ is a fundamental domain for it. If the $\Sigma$-action on $\Omega_2$ is free then so is the $\Sigma$-action on $\Omega_3$. Note that $X_{3,\Sigma}$ has finite measure because $X_{1,\Lambda}$ and $X_{2,\Sigma}$ do. 
	Finally, if we fix a fundamental domain $X_{2,\Lambda}$ for the $\Lambda$-action on $\Omega_2$, then we also have a natural identification of $\Omega_3$ with $\Omega_1\times X_{2,\Lambda}$. This shows that the $\Gamma$-action on $\Omega_3$ is free. 
	
	All the properties stated in the proposition can now directly be inferred from the two previous paragraphs. 
\end{proof}

\begin{remark}\label{rmk: induced action}
	Keeping the notation of the proof, we can also describe the $\Gamma$-action on $\Omega_3$ when the latter is identified with $X_{1,\Lambda}\times \Omega_2$ as follows:
	$$\gamma\star (x,\omega)=(\gamma\cdot x,\alpha(\gamma,x)\actom \omega),$$
	where $\cdot$ denotes the induced $\Gamma$-action on $X_{1,\Lambda}$, and where $\alpha\colon \Gamma\times X_{1,\Lambda} \to \Lambda$ is the cocycle given by the equation $\alpha(\gamma,x)\ast\gamma\ast x = \gamma\cdot x$ (we shall come back to this important notion in the next section). Note that the induced $\Gamma$-action (denoted by $\bullet$) on the fundamental domain $X_{3,\Sigma}$, identified to $X_{1,\Lambda}\times X_{2,\Sigma}$, is then given by: 
	\[
	\gamma\bullet(x,y)=(\gamma\cdot x, \alpha(\gamma,x)\cdot y).
	\]
	This point of view will be useful when we explore how integrability conditions behave under composition of couplings.\end{remark}

\begin{remark}\label{rmk: prob with quotients}
	The composition of a quotient coupling with any kind of coupling (even a measure equivalence coupling) may not be itself a sub-quotient coupling, as shown by the following example:
	Let $\Omega_1=\{1\}$ be the quotient coupling from the trivial group $\Gamma=\{1\}$ to $\Lambda=\Z$ such that $\Gamma$ acts trivially on $\{1\}$, and let $\Omega_2=\R$  be the measure equivalence coupling from $\Lambda$ to $\Sigma=\Z$ where the $\Lambda$-action is by translation by $1$, and the $\Sigma$-action is by translation by $\sqrt{2}$. The composition of these couplings is $\Omega_3=\{1\}\times \R/\Z$, where the action of $\Sigma$ is trivial on $\Omega_1$, and by irrational translation on $\R/\Z$. This is not a quotient coupling as the $\Sigma$-action is not smooth.
\end{remark}
\begin{remark}
	It is possible to relax the notions of $\Gamma$ being a measure sub-quotient (resp.\ subgroup, quotient) of $\Lambda$, by replacing the assumption that the $\Lambda$-action has a fundamental domain of finite measure by the weaker assumption that the action has finite co-measure: there exists a subset $A$ of finite measure such that $\Lambda \actom A$ has full measure. Let us refer to this as $\Gamma$ being a weak measured subquotient of $\Lambda$.
	The proof of Proposition \ref{prop:CombineCouplings} shows that composition makes sense between weak measured sub-quotients, and gives rise to a weak measured sub-quotient. In particular a composition of two measured sub-quotients is itself a weak measured sub-quotient.
	Weak measured sub-quotients seem to be interesting objects of study. 
	Let us illustrate this by an example: let $G$ be a locally compact Polish unimodular group, $\Gamma$ be a discrete subgroup of $G$, and $\Lambda$ be a dense subgroup. 
	One easily checks that $G$ equipped with its Haar measure defines a weak measure sub-quotient from $\Lambda$ to $\Gamma$, letting $\Lambda$ act by left-translation and $\Gamma$ by right-translation. 
	We decided not pursue the study of weak couplings here because we were not able to prove Theorem~\ref{thm: monotonicity of isop under Lp} under this generality (i.e.\ replacing sub-quotient by its weak counterpart). 
	In particular, when $G$ is a unimodular locally compact Polish group with $\Gamma\leq G$ discrete and $\Lambda\leq G$ dense, we do not know whether the isoperimetric profile of $\Gamma$ is necessarily asymptotically smaller than that of $\Lambda$.
\end{remark}

\subsection{Integrability conditions}\label{sec: integrability conditions}

We now come to the central notion studied in this paper. Our goal is to add quantitative constraints on a coupling that extend the well-known $\LL^p$-integrability condition for measure equivalence couplings. Classically (for measure equivalence couplings), this is done via the notion of cocycle. Given a subgroup coupling  $(\Omega, X_\Lambda,\mu)$ from $\Gamma$ to $\Lambda$, we can define the cocycle
$\alpha\colon \Gamma\times X_{\Lambda} \to \Lambda$ via by the equation \[\alpha(\gamma,x)\ast\gamma\ast x = \gamma\cdot x.\]
It is non-ambiguously defined for a.e. $x\in X_\Lambda$, as we assume that the $\Lambda$-action is free. 

\begin{definition}\label{def:intCocycle}
	Given any non-decreasing map $\varphi\colon \R^+\to \R^+$ and a finite generating subset $S_\Lambda\subseteq \Lambda$, we say that the subgroup coupling $(\Omega, X_\Lambda,\mu)$ is $\varphi$-integrable if for all $\gamma\in \Gamma$ there is $c_\gamma>0$ such that 
	$$\int_{X_\Lambda} \varphi\left(\frac{|\alpha(\gamma,x)|_{S_\Lambda}}{c_\gamma}\right)d\mu(x) <+\infty.$$
\end{definition}
This definition is very convenient but only makes sense when the $\Lambda$ action is free. So we need to find a substitute for measure quotients or measure sub-quotients.   
To that purpose, it will be useful to work with Schreier graph metrics. 
\begin{definition} Let $\Lambda$ be a finitely generated group  acting smoothly on a standard measured space $(\Omega,\mu)$. We denote by $d_{S_\Lambda}$ the \textbf{Schreier graph metric} on the $\Lambda$-orbits, namely for $y\in \Lambda\actom x$, we let 
	$$d_{S_\Lambda}(x,y)=\min\{n\in\N\colon \exists s_1,...,s_n\in S_\Lambda, y=s_1\cdots s_n\actom x\}.$$
\end{definition}

Observe that if we are given $\Lambda\act(\Omega,\mu)$ and two finite generating sets $S_1$ and $S_2$ for $\Lambda$, then there is $C>0$ such that for all $x\in\Omega$ and all $y\in \Lambda\actom x$,
\begin{equation}\label{ineq: integrability well def}
	\frac 1 Cd_{S_1}(x,y)\leq d_{S_2}(x,y)\leq Cd_{S_1}(x,y).
\end{equation}

We recall that given a smooth action $\Lambda\act(\Omega,\mu)$ and a measure fundamental domain $X$, the map $\iota_X$ is the inverse of the (bijective) projection $\pi_{\Omega/\Lambda}\colon X\to \Omega/\Lambda$. 
\begin{definition}\label{def: Lp metric for fundamental domains}
	Given any non-decreasing map $\varphi\colon \R^+\to \R^+$ and two fundamental domains $X_1$ and $X_2$ of a smooth $\Lambda$-action on
	$(\Omega,\mu)$, we say that they are $\boldsymbol\varphi$\textbf{-equivalent} if there is $c>0$ such that 
	$$ \int_{\Omega/\Lambda}\varphi\left(\frac{d_{S_{\Lambda}}(\iota_{X_1}(x),\iota_{X_2}(x))}{c}\right)d\mu(x)<+\infty.$$
\end{definition}
Note that this does not depend on the  choice of the symmetric generating set $S_\Lambda$ by virtue of inequality \eqref{ineq: integrability well def}. 

\begin{remark}\label{rmk: no c in phi equivalence}
	If $\varphi$ satisfies that for every $c>0$, there is a constant $C>0$ such that for all $x\geq 0$, $\varphi(cx)\leq C\varphi(x)$, then $X_1$ and $X_2$ are $\varphi$-equivalent if and only if
	$$ \int_{\Omega/\Lambda}\varphi\left({d_{S_{\Lambda}}(\iota_{X_1}(x),\iota_{X_2}(x))}\right)d\mu(x)<+\infty,$$
	which is then also equivalent to: for \emph{every} $c>0$ we have
	$$ \int_{\Omega/\Lambda}\varphi\left(\frac{d_{S_{\Lambda}}(\iota_{X_1}(x),\iota_{X_2}(x))}c\right)d\mu(x)<+\infty.$$
	This is the case if $\varphi(x)=x^p$ for some $p>0$, or if $\varphi$ is subadditive. An example where this is not true is when $\varphi$ is the exponential map.
\end{remark}

In order to show that $\varphi$-equivalence is indeed an equivalence relation, we introduce the following quantity: given a smooth $\Lambda$-action and two fundamental domains $X_1$ and $X_2$, we let 
$$c_{\varphi,S_\Lambda}(X_1,X_2)=\inf\left\{c>0\colon \int_{\Omega/\Lambda}\varphi\left(\frac{d_{S_{\Lambda}}(\iota_{X_1}(x),\iota_{X_2}(x))}{c}\right)d\mu(x)<+\infty \right\}.$$

\begin{proposition}\label{prop: c is pseudo-metric}
	The map $c_{\varphi,S_\Lambda}$ is a pseudo-metric on the set of fundamental domains (recall that pseudo-metrics are allowed to take the value $+\infty$). 
\end{proposition}
The proof relies on the following elementary observation. 
\begin{lemma}
	Let $\varphi\colon\R^+\to\R^+$ be a non-decreasing function, let $a,b\geq 0$ and $c,d>0$. Then 
	\[
	\varphi\left(\frac{a+b}{c+d}\right)\leq \varphi\left(\frac a c\right)+\varphi\left(\frac b d\right).
	\]
\end{lemma}
\begin{proof}
	By symmetry, we may as well assume that $\frac ac\geq \frac b d$, in which case $\frac a c\geq \frac{a+b}{c+d}$, so we have $\varphi\left(\frac{a+b}{c+d}\right)\leq\varphi\left(\frac a c\right)\leq \varphi\left(\frac a c\right)+\varphi\left(\frac b d\right)$.
\end{proof}
\begin{proof}[Proof of Proposition \ref{prop: c is pseudo-metric}]
	The map $c_{\varphi,S_\Lambda}$ is clearly symmetric and satisfies $c_{\varphi,S_\Lambda}(X,X)=0$ for every fundamental domain $X$, so we only need to check that it satisfies the triangle inequality.
	
	To this end, let $X_1$, $X_2$ and $X_3$ be fundamental domains, let $c_1>c_{\varphi,S_\Lambda}(X_1,X_2)$ and $c_2>c_{\varphi,S_\Lambda}(X_2,X_3)$.
	We have for every $x\in \Omega/\Lambda$ that
	\begin{align*}
		\frac {d_{S_\Lambda}(\iota_{X_1}(x),\iota_{X_3}(x))} {c_1+c_2} &\leq \frac {d_{S_\Lambda}(\iota_{X_1}(x),\iota_{X_2}(x))+d_{S_\Lambda}(\iota_{X_2}(x),\iota_{X_3}(x))} {c_1+c_2}.
	\end{align*}
	By the previous lemma, we thus have 
	\[
	\varphi\left(\frac {d_{S_\Lambda}(\iota_{X_1}(x),\iota_{X_3}(x))} {c_1+c_2}\right)\leq \varphi\left(\frac{d_{S_\Lambda}(\iota_{X_1}(x),\iota_{X_2}(x))}{c_1}\right)+\varphi\left(\frac{d_{S_\Lambda}(\iota_{X_2}(x),\iota_{X_3}(x))}{c_2}\right).
	\]
	By integrating and using our assumptions on $c_1$ and $c_2$, we then deduce 
	that \[\int_{\Omega/\Lambda} \varphi\left(\frac 
	{d_{S_\Lambda}(\iota_{X_1}(x),\iota_{X_3}(x))} {c_1+c_2}\right) 
	d\mu(x)<+\infty,\] and so $c_{\varphi,S_\Lambda}(X_1,X_3) \le 
	c_{\varphi,S_\Lambda}(X_1,X_2)+c_{\varphi,S_\Lambda}(X_2,X_3)$ as wanted.
\end{proof}

\begin{corollary}
	The notion of $\varphi$-equivalence is an equivalence relation between fundamental domains.
\end{corollary}

We are now ready to define our notion of $\varphi$-integrability for sub-quotient couplings. 

\begin{definition}\label{def:phi-couplings}
	Let $\varphi,\psi\colon \R^+\to \R^+$ be non-decreasing maps. \begin{itemize}
		\item A sub-quotient (resp. subgroup, quotient) coupling $(\Omega,X_\Lambda,\mu)$ from $\Gamma$ to $\Lambda$ is a called $\boldsymbol\varphi$\textbf{-integrable}  if for every $\gamma\in\Gamma$ we have that $X_\Lambda$ and $\gamma\ast X_\Lambda$ are $\varphi$-equivalent as fundamental domains of the $\Lambda$-action. In that case, we call $\Gamma$  a $\boldsymbol{\varphi}\textbf-$\textbf{integrable measure sub-quotient} (resp. \textbf{subgroup} or \textbf{quotient}) of $\Lambda$. 
		
		\item A measure equivalence coupling $(\Omega,X_\Gamma,X_\Lambda,\mu)$ from $\Gamma$ to $\Lambda$ is called  $\boldsymbol(\boldsymbol\varphi\boldsymbol,\boldsymbol\psi\boldsymbol)$\textbf{-integrable} when the coupling $(\Omega,X_\Lambda,\mu)$ from $\Gamma$ to $\Lambda$ is $\varphi$-integrable and the coupling $(\Omega,X_\Gamma,\mu)$ from $\Lambda$ to $\Gamma$ is $\psi$-integrable. We then say that $\Gamma$ and $\Lambda$ are $\boldsymbol(\boldsymbol\varphi\boldsymbol,\boldsymbol\psi\boldsymbol)$\textbf{-integrable measure equivalent}, or simply that they are $(\varphi,\psi)$-measure equivalent.
		
	\end{itemize}
\end{definition}

\begin{remark}\label{rmk: concrete phi equivalence}
	By spelling out what $\varphi$-integrability means, we see that a measure sub-quotient coupling $(\Omega,X_\Lambda,\mu)$ from $\Gamma$ to $\Lambda$  is $\varphi$-integrable if and only if for every $\gamma\in\Gamma$, there is $c_\gamma>0$ such that 
	\[
	\int_{X_\Lambda} \varphi\left(\frac{d_{S_\Lambda}(\gamma\actx x,\gamma\actom x)}{c_\gamma}\right)d\mu(x)<+\infty,
	\]
	where $\actom$ denotes the $\Gamma$-action on $\Omega$ and $\cdot$ denotes the induced $\Gamma$-action on $X_\Lambda$.
	We deduce that a measure subgroup coupling from $\Gamma$ to $\Lambda$ is a $\varphi$-coupling if and only if the cocycle $\alpha\colon\Gamma\times X_\Lambda\to \Lambda$ is $\varphi$-integrable in the sense of Definition~\ref{def:intCocycle}.
\end{remark}

\begin{proposition}\label{prop: only check integrability on generators}
	A sub-quotient coupling $(\Omega,X_\Lambda,\mu)$  from $\Gamma=\la S_\Gamma\ra$ to $\Lambda$ is $\varphi$-integrable if and only if for every $s\in S_\Gamma$ we have that $X_\Lambda$ and $s\ast X_\Lambda$ are $\varphi$-equivalent.
\end{proposition}
\begin{proof}
	Since the $\Gamma$-action commutes with the $\Lambda$-action, it must preserve the equivalence relation of $\varphi$-equivalence, so for every $\gamma\in\Gamma$ and every $s\in S_\Gamma$ we have that $\gamma\actom X_\Lambda$ and $\gamma s\actom X_\Lambda$ are $\varphi$-equivalent. From there, the statement follows by induction from Proposition \ref{prop: c is pseudo-metric}.
\end{proof}


\begin{remark}
	It is not hard to see that being a $\varphi$-integrable sub-quotient coupling only depends on the asymptotic behavior of $\varphi$. More precisely, if $\varphi\preccurlyeq \psi$, then every  $\psi$-integrable sub-quotient coupling is also a $\varphi$-integrable.
\end{remark}

For convenience and consistency with the literature, for $p\in (0,\infty)$ we talk about $\LL^p$ couplings instead of $x\mapsto x^p$-integrable couplings.
We also say that we have an \textbf{$\boldsymbol{\LL^\infty}$ sub-quotient coupling} from $\Gamma$ to $\Lambda$ when the $\Gamma$-action satisfies for every $\gamma\in\Gamma$ that the map $\Omega/\Lambda \to \Lambda\colon x\mapsto d_{S_\Lambda}(\gamma\actx x,\gamma\actom x)$ is essentially bounded. Note that every $\LL^\infty$ sub-quotient coupling is  $\varphi$-integrable for any increasing map $\varphi\colon \R^+\to \R^+$.\\

Let us now explain how various established notions fit into our asymmetric framework.

\begin{itemize}
	\item Two finitely generated groups $\Gamma$ and $\Lambda$ are  $L^p$ \emph{measure equivalent} in the sense of \cite{baderIntegrableMeasureEquivalence2013} when there is an $(\LL^p, \LL^p)$ measure equivalence coupling from $\Gamma$ to $\Lambda$. 
	\item Two finitely generated groups $\Gamma$ and $\Lambda$ are  \emph{uniform measure equivalent} in the sense of \cite{shalomHarmonicAnalysisCohomology2004} when there is an $(\LL^\infty, \LL^\infty)$ measure equivalence coupling between $\Gamma$ and $\Lambda$.
	\item Two finitely generated groups $\Gamma$ and $\Lambda$ are \emph{bounded measure equivalent} in the sense of \cite{sauerMathrmLInvariantsGroups2002} when there is an $(\LL^\infty,\LL^\infty)$ measure equivalence coupling from $\Gamma$ to $\Lambda$ which is cobounded in both directions. 
\end{itemize}

\subsection{Composition of \texorpdfstring{$\varphi$}{phi}-integrable 
	couplings}\label{sec:compositioncouplings}

We study how integrability conditions behave under composition of couplings. We first consider the case where $\varphi\colon\R^+\to\R^+$ is subadditive, e.g. when $\varphi$ is concave. Our arguments will follow closely those from \cite[Sec. A.2]{baderIntegrableMeasureEquivalence2013}.

\begin{lemma}\label{lem:ConcavePhiEquivalence}
	Let $\Gamma$ and $\Lambda$ be two finitely generated groups, let $\varphi\colon \R^+\to \R^+$ be a non-decreasing subadditive map and let $(\Omega,X_\Lambda,\mu)$ be a $\varphi$-integrable sub-quotient coupling from $\Gamma$ to $\Lambda$.
	Then there is a constant $C>0$ such that for every $\gamma\in\Gamma$
	$$\int_{X_\Lambda}\varphi\left(d_{S_{\Lambda}}(\gamma\actx x,\gamma\actom x)\right)d\mu(x)\le C\abs{\gamma}_{S_\Gamma}.$$
\end{lemma}
\begin{proof} Assume $(\Omega,X_\Lambda,\mu)$ is a $\varphi$-integrable sub-quotient coupling from $\Gamma$ to $\Lambda$.
	Given two fundamental domains $X_1$ and $X_2$ for the $\Lambda$-action, we define their $\varphi$-distance $d_{\varphi,S_\Lambda}(X_1,X_2)$ by
	\[
	d_{\varphi,S_\Lambda}(X_1,X_2)=\int_{\Omega/\Lambda}\varphi\left({d_{S_{\Lambda}}(\iota_{X_1}(x),\iota_{X_2}(x))}\right)d\mu(x)
	\]
	Note that this distance is symmetric and satisfies the triangle inequality, so it is a pseudometric as soon as $\varphi(0)=0$. Moreover, using Remark \ref{rmk: no c in phi equivalence}, we have for every $\gamma\in \Gamma$ that 
	\[
	d_{\varphi,S_\Lambda}(X_\Lambda,\gamma\actom X_\Lambda)<+\infty,
	\]
	while by Remark \ref{rmk: concrete phi equivalence} we have $d_{\varphi,S_\Lambda}(X_\Lambda,\gamma\actom X_\Lambda)=\int_{X_\Lambda}\varphi\left(d_{S_{\Lambda}}(\gamma\actx x,\gamma\actom x)\right)d\mu(x)$.

	Let $C=\max_{\gamma\in S_\Gamma}d_{\varphi,S_\Lambda}(X_\Lambda,\gamma\actom X_\Lambda)$. The $\Gamma$-action on the set of $\Lambda$-fundamental domains preserves $d_{\varphi,S_\Lambda}$ because it commutes with the $\Lambda$-action, so for every $s_1,...,s_n\in S_\Gamma$ we have by the triangle inequality
	\begin{align*}
		d_{\varphi,S_\Lambda}(s_1\cdots s_n\actom X_\Lambda,X_\Lambda)&\leq 
		\sum_{i=1}^n d_{\varphi,S_\Lambda}(s_1\cdots s_i\actom X_\Lambda, s_1\cdots s_{i-1}\actom X_\Lambda)\\
		&\leq \sum_{i=1}^n d_{\varphi,S_\Lambda}(s_i\actom X_\Lambda,X_\Lambda)\\
		&\leq Cn,
	\end{align*}
	which yields the desired result.
\end{proof}
\begin{remark}
	Note that when $\varphi$ is subadditive and non-decreasing, we always have $\varphi(x)\leq x\varphi(1)+\varphi(1)$, in particular every $\LL^1$ sub-quotient coupling is  $\varphi$-integrable.
\end{remark}
Let us now study how couplings compose in the subadditive regime.

\begin{proposition}\label{prop:CombineIntCouplingsConcave}
	Let $\varphi,\psi\colon \R^+\to \R^+$ be non-decreasing subadditive maps with $\varphi$ moreover concave and let $\Gamma$, $\Lambda$ and $\Sigma$ be three finitely generated groups. Let $(\Omega_1,X_{1,\Lambda},\mu_1)$ be a $\varphi$-integrable subgroup coupling from $\Gamma$ to $\Lambda$ and let $(\Omega_2,X_{2,\Sigma},\mu_2)$ be a $\psi$-integrable sub-quotient coupling from $\Lambda$ to $\Sigma$. 
	Then the composition of these two couplings is a $\varphi\circ\psi$-integrable sub-quotient coupling from $\Gamma$ to $\Sigma$.
\end{proposition}
\begin{proof}
	Thanks to Lemma \ref{lem:ConcavePhiEquivalence} we find $C>0$ such that
	\begin{equation}\label{eq:ConcaveInt}
		\int_{X_{2,\Sigma}}\psi\left(d_{S_{\Sigma}}(\lambda\cdot x,\lambda\actom x))\right)d\mu(x)\le C|\lambda|_{S_\Lambda} \text{ for every } \lambda\in\Lambda.
	\end{equation}
	By scaling the measure $\mu_2$ we may assume that $\mu_2(X_{2,\Sigma})=1$. Denote by $\alpha\colon \Gamma\times X_{1,\Lambda}\to\Lambda$ the cocycle defined by the equation  $\alpha(\gamma,x)\actom \gamma\actom  x=\gamma\cdot x $. By Remark \ref{rmk: induced action} and the definition of the composition of our two couplings, we need to show that the following quantity is finite:
	\[
	\int_{X_{1,\Lambda}}\int_{X_{2,\Sigma}} \varphi\circ \psi\left(d_{S_\Sigma}(\alpha(\gamma,x)\actx y, \alpha(\gamma,x)\actom y  )\right) d\mu_2(y)d\mu_1(x).
	\]
	Now by Jensen's inequality, this is at most
	\[
	\int_{X_{1,\Lambda}}\varphi\left(\int_{X_{2,\Sigma}} \psi\left(d_{S_\Sigma}(\alpha(\gamma,x)\actx y, \alpha(\gamma,x)\actom y  )\right) d\mu_2(y)\right)d\mu_1(x),
	\]
	which by inequality \eqref{eq:ConcaveInt} is bounded above by
	$\int_{X_{1,\Lambda}}\varphi(C\abs{\alpha(\gamma,x)}_{S_\Lambda})d\mu_1(x)$. The 
	latter is indeed finite
	by our assumption on the first coupling and Remark \ref{rmk: no c in phi equivalence}.
\end{proof}

The above result can be combined with Proposition \ref{prop:CombineCouplings} to obtain a $\varphi\circ\psi$-integrable
subgroup coupling or a $\varphi\circ\psi$-integrable quotient coupling  by composition. 

\begin{remark}
	Given two finitely generated groups $\Gamma$ and $\Lambda$, one could define
	\[
	\alpha(\Gamma,\Lambda)=
	-\log\left(\sup\{p\leq 1\colon \Gamma\text{ and }\Lambda\text{ have an }\LL^p\text{ measure equivalence coupling}\}\right)
	\]
	The previous proposition implies that this is a pseudo-metric on the space of isomorphism classes of finitely generated groups. It would be interesting to understand this pseudo-metric further. For instance, Theorem~\ref{thm:OE between Zd and Zd'} implies that $\alpha(\Z^{n},\Z^{m})=\abs{\log n-\log m}$.
\end{remark}

For non subadditive maps $\varphi$, we need a stronger notion of $\varphi$-integrability so that it behaves well with respect to composition.
\begin{definition}\label{def:strong}
	Let $\varphi\colon \R^+\to \R^+$ be an increasing map.
	We say that a coupling $(\Omega,X_\Lambda,\mu)$ from $\Gamma$ to $\Lambda$ is  \textbf{strongly $\varphi$-integrable} or $\varphi^\diamond$-integrable if for every $\varepsilon>0$ there are $\delta>0$ and $C>0$ such that for every $\gamma\in\Gamma$,
	\[\int_{X_\Lambda} \varphi\left(\delta\, d_{S_{\Lambda}}(\gamma\actx x,\gamma\actom x)\right) d\mu(x) \le C\varphi\left(\varepsilon |\gamma|_{S_\Gamma}\right)\]
\end{definition}
Note again that thanks to inequality \eqref{ineq: integrability well def} strong integrability does not depend on the choice of the finite generating set  $S_\Lambda$. However, the above condition has to be checked on \emph{every} element of $\Gamma$.
\begin{proposition}\label{prop:CombineIntCouplings}
	Let $\varphi\colon \R^+\to \R^+$ be an increasing map and let $\Gamma$, $\Lambda$ and $\Sigma$ be three finitely generated groups.
	Let $(\Omega_1,X_{1,\Lambda},\mu_1)$ be an strongly $\varphi$-integrable  subgroup coupling from $\Gamma$ to $\Lambda$ and let $(\Omega_2,X_{2,\Sigma},\mu_2)$ be a strongly $\varphi$-integrable sub-quotient coupling from $\Lambda$ to $\Sigma$. 
	Then the composition of these two couplings is a strongly $\varphi$-integrable sub-quotient coupling from $\Gamma$ to $\Sigma$.
\end{proposition}
\begin{proof}
	Let $\varepsilon>0$. As $(\Omega_1,X_{1,\Lambda},\mu_1)$ is strongly $\varphi$-integrable, there are $\delta_\Lambda>0$ and $C_\Gamma>0$ such that
	\begin{equation}\label{eq: strong ineq for gamma}
		\int_{X_{1,\Lambda}} \varphi\left(\delta_\Lambda d_{S_{\Lambda}}(\gamma\actx x,\gamma\actom x)\right) d\mu_1(x) \leq C_\Gamma\,\varphi\left(\varepsilon \abs{\gamma}_{S_\Gamma}\right)\text{ for every $\gamma\in\Gamma$.}
	\end{equation}
	and as $(\Omega_2,X_{2,\Sigma},\mu)$ is strongly $\varphi$-integrable, there exist $\delta_\Sigma>0$ and $C_\Lambda>0$ such that
	\begin{equation}\label{eq: strong ineq for lambda}
		\int_{X_{2,\Sigma}} \varphi\left(\delta_\Sigma d_{S_{\Sigma}}(\lambda\actx x,\lambda\actom x)\right) d\mu_2(x) \leq C_\Lambda\,\varphi\left(\delta_\Lambda |\lambda|\right)\text{ for every $\lambda\in\Lambda$.}
	\end{equation}
	By Remark \ref{rmk: induced action} and the definition of the composition of our two couplings, we need to estimate the following quantity:
	\[
	\int_{X_{1,\Lambda}}\int_{X_{2,\Sigma}} \varphi\left(\delta_\Sigma d_{S_\Sigma}(\alpha(\gamma,x)\actx y, \alpha(\gamma,x)\actom y  )\right) d\mu_2(y)d\mu_1(x).
	\]
	By inequality \eqref{eq: strong ineq for lambda}, this is bounded above by 
	\[
	\int_{X_{1,\Lambda}}C_\Lambda \varphi(\delta_\Lambda \abs{\alpha(\gamma,x)}_{S_\Lambda})\leq C_\Lambda C_\Gamma\,\varphi\left(\varepsilon |\gamma|_{S_\Gamma}\right)
	\]
	as wanted, where the last inequality is a consequence of inequality \eqref{eq: strong ineq for gamma}, and the fact that by definition $\abs{\alpha(\gamma,x)}_{S_\Lambda}=d_{S_\Lambda}(\gamma\cdot x,\gamma\actom x)$.
\end{proof}

For some maps $\varphi$ we can weaken the strong integrability condition. Most 
notably for $\LL^p$ couplings, where $p\geq 1$. 
\begin{proposition}
	Let $p\ge 1$. Every $\LL^p$ sub-quotient coupling from $\Gamma$ to $\Lambda$ is a strongly $\LL^p$ sub-quotient coupling.
\end{proposition}
\begin{proof}
	We follow an approach similar to that of  Lemma~\ref{lem:ConcavePhiEquivalence}. For two fundamental domains $X_1$ and $X_2$ for the $\Lambda$-action, we define their $\LL^p$ distance by
	\[
	d_{\LL^p,S_\Lambda}(X_1,X_2)=\left(\int_{\Omega/\Lambda}\left({d_{S_{\Lambda}}(\iota_{X_1}(x),\iota_{X_2}(x))}\right)^pd\mu(x)\right)^{1/p}
	\]
	It is not hard to check that $d_{\LL^p,S_\Lambda}$ is a metric, and since the $\Gamma$-action commutes with the $\Lambda$-action, the group $\Gamma$ acts on the set of $\Lambda$-fundamental domains by isometries. Let  $X_\Lambda$ be one of them. As in the proof of Lemma~\ref{lem:ConcavePhiEquivalence}, we find $C_\Gamma>0$ such that for every $\gamma\in\Gamma$, $d_{\LL^p,S_\Lambda}(X_\Lambda,\gamma\actom X_\Lambda)\leq C_\Gamma\abs{\gamma}_{S_\Gamma}$. This means that for all $\gamma\in\Gamma$, we have
	\[
	\int_{X_\Lambda} d_{S_{\Lambda}}(\gamma\actx x,\gamma\actom x)^p d\mu(x) \leq C_\Gamma^p|\gamma|_{S_\Gamma}^p,
	\]
	from which the result easily follows:  given $\varepsilon>0$, we take 
	$\delta=\varepsilon$, $C=C_\Gamma^p$ and note that 
	\[
	\int_{X_\Lambda} \left(\delta d_{S_{\Lambda}}(\gamma\actx x,\gamma\actom 
	x)\right)^p d\mu(x)=\delta^p	\int_{X_\Lambda} 
	d_{S_{\Lambda}}(\gamma\actx 
	x,\gamma\actom x)^p d\mu(x) \leq 
	C_\Gamma^p(\varepsilon|\gamma|_{S_\Gamma})^p\] as wanted.
\end{proof}

Say for $p\geq 1$ that two finitely generated groups are $\LL^p$ \textbf{measure equivalent} when there is an $(\LL^p,\LL^p)$ measure equivalence coupling from $\Gamma$ to $\Lambda$. We deduce from the two previous results that $\LL^p$ measure equivalence is an equivalence relation between finitely generated groups, as proven by Bader, Furman and Sauer in \cite[Lemma A.2.]{baderIntegrableMeasureEquivalence2013}. Note however that this is not true anymore for $p< 1$; counter-examples are provided by Theorem~\ref{thm:OE between Zd and Zd'}.

For exponentially integrable couplings, the next proposition states that 
it is enough to find a single $\eps$ so
as to witness strong exponential integrability.
\begin{proposition}
	Let $\Gamma$ and $\Lambda$ be two finitely generated groups and let $(\Omega,X_\Lambda,\mu)$ be a sub-quotient coupling from $\Gamma$ to $\Lambda$. If  there are $\varepsilon'>0$, $\delta'>0$ and $C'>0$ such that
	\[\int_{X_\Lambda} \exp\left(\delta'\, d_{S_{\Lambda}}(\gamma\actx 
	x,\gamma\actom x)\right) d\mu(x) \le C'\exp\left(\varepsilon' |\gamma|_{S_\Gamma}\right)\text{ for every $\gamma\in\Gamma$,}\]
	then the coupling is strongly $\exp$-integrable.
\end{proposition}
\begin{proof} 
	Let $\varepsilon>0$. If $\varepsilon\ge \varepsilon'$, we take $C=C'$ and $\delta=\delta'$
	and check that Definition \ref{def:strong} holds with these constants. 
	If $\varepsilon< \varepsilon'$, then take $C={C'}^\frac{\varepsilon}{\varepsilon'}$ and $\delta=\frac{\delta'\varepsilon}{\varepsilon'}$. By rescaling the measure if necessary, we can assume that $\mu(X_\Lambda)=1$. Then, applying Jensen's inequality, we have for every $\gamma\in\Gamma$ that
	\begin{align*}
		\int_{X_\Lambda} \exp\left(\delta\, d_{S_{\Lambda}}(\gamma\actx 
		x,\gamma\actom x)\right) d\mu(x) & = \int_{X_\Lambda} \exp\left(\delta'\, d_{S_{\Lambda}}(\gamma\actx 
		x,\gamma\actom x)\right)^\frac{\varepsilon}{\varepsilon'} d\mu(x)\\
		& \le \left(\int_{X_\Lambda} \exp\left(\delta'\, d_{S_{\Lambda}}(\gamma\actx 
		x,\gamma\actom x)\right) d\mu(x)\right)^\frac{\varepsilon}{\varepsilon'}\\
		& \le \left(C'\exp(\varepsilon'|\gamma|)\right)^\frac{\varepsilon}{\varepsilon'}\\
		& = C \exp(\varepsilon|\gamma|).
	\end{align*}
	Thus, the coupling is strongly $\exp$-integrable as wanted.
\end{proof}

Interesting examples of strongly $(\exp,\exp)$-integrable measure equivalence couplings will be provided in Section \ref{sec: unstable}.

\subsection{Variations on orbit equivalence} \label{sec: OE couplings}
\par We now turn our attention to orbit equivalence of groups, which implies measure equivalence. To make the connection, we need to introduce measure-preserving equivalence relations (see \cite{kechrisTopicsOrbitEquivalence2004} for details). 

\begin{definition}\label{def: pmp eq rel}
	Given a measure-preserving action of a countable group $\Lambda$ on a standard probability space $(X,\mu)$, we associate to it a \textbf{measure-preserving equivalence relation} $\mathcal R_\Lambda$ defined by 
	$$\mathcal R_\Lambda\coloneqq \{(x,\lambda\cdot x)\colon x\in X,\lambda\in\Lambda\}.$$
\end{definition}

A key property of measure-preserving equivalence relations is that they can be endowed with a natural $\sigma$-finite measure $M$. Namely, for $\mathcal R$ a measure-preserving equivalence relation, and $A$ a Borel subset of $\mathcal R$, we let
$$M(A)=\int_X\abs{A_x}d\mu(x),$$
where $A_x=\{y\in X\colon (x,y)\in A\}$. Such a measure is invariant under the flip map $(x,y)\mapsto (y,x)$ (see \cite[p. 34]{kechrisTopicsOrbitEquivalence2004}), which means that
$$M(A)=\int_X\abs{A^y}d\mu(y),$$ 
where $A^y=\{x\in X\colon (x,y)\in A\}$.

\begin{definition}
	Let $\Gamma$ and $\Lambda$ be two finitely generated groups.
	
	\begin{itemize}	
		\item	An \textbf{orbit sub-quotient coupling} from $\Gamma$ to $\Lambda$ is  a triple $(X,Y,\mu)$  where
		\begin{itemize}
			\item[(i)]	 $(Y,\mu)$ is a standard $\sigma$-finite space equipped with a measure-preserving $\Lambda$-action, 
			\item[(ii)] $X$ is a Borel subset of $Y$ of measure $1$  equipped with a free measure-preserving $\Gamma$-action; 
			\item[(iii)] and finally for almost every $x\in X$ we have that $\Gamma\cdot x \subseteq \Lambda\cdot x$. 
		\end{itemize}
		\item  If the $\Lambda$-action is also free, we say that $(X,Y,\mu)$ is an \textbf{orbit subgroup} coupling of $\Gamma$ with $\Lambda$.
		If $X=Y$ and for almost every $x\in X$ we have that $\Gamma\cdot x = \Lambda\cdot x$, we say that $(X,Y,\mu)$  is an \textbf{orbit quotient coupling}, which we then simply write as $(X,\mu)$. Finally, $(X,\mu)$  is an \textbf{orbit equivalence} coupling of $\Gamma$ with $\Lambda$ if it is both an orbit quotient and an orbit subgroup coupling.
	\end{itemize}
\end{definition}

\begin{remark}
	We note that admitting an orbit equivalence coupling agrees with the definitions of orbit equivalence between countable groups in the literature.
\end{remark}
\begin{remark}
	An orbit sub-quotient coupling $(X,Y,\mu)$  from $\Gamma$ to $\Lambda$ such that for almost every $x\in X$ we have that $\Gamma\cdot x = \Lambda\cdot x$ restricts to an orbit quotient coupling $(X,X,\mu)$.
\end{remark}

\begin{remark}
	To every orbit sub-quotient coupling is naturally associated a measure sub-quotient coupling. We will see later that this statement admits some kind of converse (see Proposition \ref{prop:OrbitSubquotient}). 
	\begin{itemize}
		\item The coupling space is $\Omega\coloneqq\mathcal R_\Lambda\cap(X\times Y)$ equipped with the measure induced by $M$.
		\item The commuting actions are defined as follows: for every $\gamma\in\Gamma$, $\lambda\in\Lambda$ and every $(x,y)\in\mathcal R_\Lambda$,
		\[
		\gamma\actom (x,y)=(\gamma\actx x,y)\text{ and }\lambda\actom(x,y)=(x,\lambda\actx y).
		\]
		\item The chosen $\Lambda$-fundamental domain is the diagonal: $X_\Lambda=\{(x,x)\colon x\in X\}$. 
		\item The $\Gamma$-action is smooth. Indeed, a Borel fundamental domain can be obtained as  the intersection with $X\times Y$ of a disjoint union of graphs of \emph{Borel choice functions} for the subequivalence relation $\mathcal (\mathcal R_\Gamma\cup\Delta_Y) \leq\mathcal R_\Lambda$ (see \cite[Sec. 2.(A)]{ioanaSubequivalenceRelationsPositivedefinite2009}).
	\end{itemize}
	Note furthermore that $(X_\Lambda,M)$ is naturally isomorphic to $(X,\mu)$ via $x\mapsto(x,x)$, and that the induced $\Gamma$-action on $X_\Lambda$ is conjugate via the inverse of this map to the original action on $(X,\mu)$. 
\end{remark}

\begin{definition} An orbit sub-quotient coupling from $\Gamma$ to $\Lambda$ is $\varphi$-integrable when it is $\varphi$-integrable as a measure sub-quotient coupling. 
\end{definition}
Note that $\varphi$-integrability for an orbit sub-quotient coupling as above means that for all $\gamma\in\Gamma$, there is $c_\gamma>0$ such that 

\[
\int_X\varphi\left(\frac{d_{S_\Lambda}(x,\gamma\cdot x)}{c_\gamma}\right)d\mu(x)<+\infty.
\]
\begin{remark}
	When $\varphi(x)=x^p$ for some $p\geq 1$, this means that $\Gamma$ is contained in the $\LL^p$ \emph{full group} of the $\Lambda$-action, as defined in 
	\cite{lemaitremeasurableanaloguesmall2018}.
\end{remark}

\begin{remark} Every orbit equivalence (resp. subgroup, quotient) coupling gives rise to a measure equivalence (resp. subgroup, quotient) coupling.  We can thus define similarly $(\varphi,\psi)$-integrable orbit equivalence couplings. Finally, we can also define strong $\varphi$-integrability conditions for orbit couplings. 
\end{remark}

There is a well-known connection between orbit equivalence couplings and measure equivalence couplings of the form $(\Omega,X_\Gamma,X_\Lambda,\mu)$ with $X_\Gamma=X_\Lambda\coloneqq X$ of measure $1$. This is not exactly a one-to-one correspondence as the actions of $\Lambda$ and $\Gamma$ on $X$ may not be free for such a mesure coupling. This mild issue is taken care of by the following proposition.

\begin{proposition}\label{prop: freeness for induced}
	Let $(\Omega,X_\Lambda, \mu)$ be a measure sub-quotient coupling from $\Gamma$ to $\Lambda$, let $(Y,\nu)$ be a standard probability space equipped with a free $\Gamma$-action. Then $(\Omega\times Y, X_\Lambda\times Y, \mu\otimes\lambda)$ is a measure sub-quotient coupling from $\Gamma$ to $\Lambda$, where $\Gamma$ acts diagonally and $\Lambda$ acts on the first coordinate. Moreover, the induced $\Gamma$-action on $X_\Lambda\times Y$ is free, the coupling $(\Omega\times Y, X_\Lambda\times Y, \mu\otimes\lambda)$ is $\varphi$-integrable if and only if $(\Omega,X_\Lambda,\mu)$ was, and if $X_\Gamma$ was a fundamental domain for $\Gamma\act\Omega$, then $X_\Gamma\times Y$ is a fundamental domain for $\Gamma\act \Omega\times Y$. 
\end{proposition}
\begin{proof}
	It is clear that $X_\Lambda\times Y$ is a fundamental domain for the $\Lambda$-action. Since the $\Gamma$-action on $\Omega$ is free, $X_\Gamma\times Y$ is a fundamental domain for the new diagonal action. The induced $\Gamma$-action on $X_\Lambda\times Y$ is the diagonal action obtained from its induced action on $X_\Lambda$ and its action on $Y$. We deduce that this action is free. Finally, the statement about $\varphi$-integrability follows directly from the fact that for every $y\in Y$ and every $\omega,\omega'\in\Omega$, we have  $d_{S_\Lambda}((\omega,y),(\omega',y))=d_{S_\Lambda}(\omega,\omega')$.
\end{proof}
\begin{remark}
	Note that the above lemma can be applied to any countable group $\Gamma$: if $\Gamma$ is infinite, one can take  $(Y,\nu)$ as a Bernoulli shift of $\Gamma$, and if $\Gamma$ is finite, one can take $Y=\Gamma$ acted upon by left translation, equipped with the normalized counting measure.
\end{remark}

\begin{proposition}\label{prop:OrbitQuotient}
	Let $\Gamma$ and $\Lambda$ be countable groups. If there is a $(\varphi,\psi)$ measure equivalence (resp. measure quotient) coupling from $\Gamma$ to $\Lambda$ where the two fundamental domains coincide, then there is a $(\varphi,\psi)$ orbit equivalence (resp. orbit quotient) coupling from $\Gamma$ to $\Lambda$. 
\end{proposition}
\begin{proof}
	Let $(\Omega,X_\Gamma,X_\Lambda,\mu)$ be a $(\varphi,\psi)$-integrable measure equivalence coupling from $\Gamma$ to $\Lambda$, and denote by $X=X_\Gamma=X_\Lambda$ the common fundamental domain. Up to rescaling the measure (which does not impact the integrability conditions), we may as well assume that $X$ has measure $1$. Using the previous remark, we can apply the above proposition twice and see that without loss of generality, we can also assume that the induced $\Gamma$- and $\Lambda$-actions on $X$ are free. We denote them by $\cdot$.
	
	Now observe that for every $\gamma\in\Gamma$ and every $x\in X$, we have $\gamma\cdot x\in \Lambda\actom \gamma\actom x$ so there is $\lambda\in\Lambda$ such that $\gamma\cdot x= \lambda\actom \gamma\actom x=\gamma\actom\lambda\actom x$. In particular $\lambda\cdot x=\gamma\cdot x$, so we conclude that $\Gamma\cdot x\subseteq\Lambda\cdot x$. By symmetry, we also have $\Lambda\cdot x\subseteq\Gamma\cdot x$, so we conclude that $(X,\mu)$ is an orbit equivalence coupling. Finally, the map $(\gamma\cdot x,x)\mapsto\gamma\actom x$ is a $\Gamma\times\Lambda$-equivariant bijection from $\mathcal R_\Gamma=\mathcal R_\Lambda$ to $\Omega$ which takes $\{(x,x)\colon x\in X\}$ to $X$, and thus the $(\varphi,\psi)$-integrability of the orbit coupling $(X,\mu)$ follows from that of $(\Omega,X_\Gamma,X_\Lambda,\mu)$.
	
	The statement for orbit quotient couplings follows from the same argument, except we apply the above proposition only once so as to make the induced $\Gamma$-action free. 
\end{proof}

Similar comparisons can be made for orbit subgroup and sub-quotient couplings. Observe that an orbit sub-quotient coupling yields a measure sub-quotient coupling such that $X_\Lambda$ intersects every $\Gamma$-orbit at most once. We have the following converse.
\begin{proposition}\label{prop:OrbitSubquotient}
	Let $\Gamma$ and $\Lambda$ be two finitely generated groups. If there is a $\varphi$-integrable measure sub-quotient (resp. measure subgroup) coupling $(\Omega,X_\Lambda,\mu)$ from $\Gamma$ to $\Lambda$ such that $X_\Lambda$ intersects every $\Gamma$-orbit at most once, then there is a $\varphi$-integrable orbit sub-quotient (resp. orbit subgroup) coupling from $\Gamma$ to $\Lambda$. 
\end{proposition}
\begin{proof}
	We first assume that $(\Omega,X_\Lambda,\mu)$ is a $\varphi$-integrable measure sub-quotient coupling from $\Gamma$ to $\Lambda$ such that $X_\Lambda$ intersects every $\Gamma$-orbit at most once. Up to scaling the measure and using Proposition \ref{prop: freeness for induced}, we may as well assume that $\mu(X_\Lambda)=1$ and that the induced $\Gamma$-action on $X_\Lambda$ is free.
	
	We define $Y\coloneqq\Omega/\Gamma$ and $X\coloneqq\Omega/\Lambda$, equipped with their respective actions of $\Lambda$ and $\Gamma$. By assumption the restriction of $\pi_{\Omega/\Gamma}$ to $X_\Lambda$ is injective, so identifying $X$ with $\pi(X_\lambda)$ we may as well view $X$ as a subset of $Y$. To prove that this defines an orbit sub-quotient coupling, it remains to show that $\Gamma\cdot x\subseteq \Lambda\cdot x$ for every $x\in X$. 
	For every $\gamma\in\Gamma$ and $x\in X_\Lambda$ there is $\lambda\in\Lambda$ such that $\gamma\cdot x=\lambda\ast\gamma\ast x \in X_\Lambda$. Then $\gamma\ast(\lambda\ast x)\in X_\Lambda$, so $\lambda \cdot x=\gamma\ast(\lambda\ast x)=\gamma\cdot x$.
	Thus, $\Gamma\cdot x\subseteq \Lambda\cdot x$, and $(X,Y,\mu)$ is indeed an orbit sub-quotient coupling. 
	
	We now check that the map $\lambda\actom x\mapsto(x,\lambda\cdot x)$ from $\Omega$ to $\mathcal R_\Lambda\cap X \times Y$ is well-defined and bijective. First, if $\lambda\actom x=\lambda'\actom x'$ for some $x,x'\in X_\Lambda$ and $\lambda,\lambda'\in\Lambda$, we have $x=x'$ since $X_\Lambda$ is a fundamental domain, and then by definition of the induced action we must have $\lambda\cdot x=\lambda'\cdot x'$: our map is well-defined. In order to show it is injective, assume that $\lambda\actx x=\lambda'\actx x$, then $\lambda\inv\lambda'\actx x=x$, so there is $\gamma\in\Gamma$ such that $\lambda\inv\lambda'\actom x=\gamma\actom x$, so $\gamma\cdot x=x$. By freeness of the induced $\Gamma$-action, we deduce that $\gamma=e_\Gamma$, and so $\lambda\actom x=\lambda'\actom x$ as wanted. Surjectivity is clear from the definition.
	Note that this map can be more formally defined as $(\pi_{\Omega/\Lambda},\pi_{\Omega/\Gamma})$. Its $\Gamma\times\Lambda$-equivariance is obvious from that description.
	We conclude that the orbit sub-quotient coupling that we obtained is $\varphi$-integrable if and only if our original coupling was.	
	
	For the orbit subgroup case we proceed as before, except we first make sure that both induced actions are free by using Proposition \ref{prop: freeness for induced} twice.
\end{proof}

\begin{remark}
	The two previous propositions show that orbit couplings can be composed as in Proposition \ref{prop:CombineCouplings}. This is shown by combining Proposition \ref{prop:CombineCouplings} with Proposition \ref{prop:OrbitQuotient} and Proposition \ref{prop:OrbitSubquotient}.
	The integrability of this composition coupling satisfies the same results as in Proposition \ref{prop:CombineIntCouplingsConcave} and Proposition \ref{prop:CombineIntCouplings}.
\end{remark}

\section{Revisiting Lewis Bowen's monotonicity theorem for the volume growth}\label{sec:Bowen}

We let $\Gamma$ be a finitely generated group equipped with a finite generating subset $S_\Gamma$. By a slight abuse of notation, we denote $B_{\Gamma}(\gamma,n)$ the ball of radius $n$ centered at $\gamma\in \Gamma$ for the left-invariant word metric $d_{S_\Gamma}$ (therefore omitting the mention of $S_\Gamma$). We also denote $V_\Gamma(n)=|B_{\Gamma}(\gamma,n)|$.
The volume growth of $\Gamma$ is the asymptotic behavior of $V_\Gamma$ (which does not depend on a specific choice of $S_\Gamma$). 
In \cite[Thm.~B2]{austinIntegrableMeasureEquivalence2016}, Lewis Bowen proves that the volume growth is invariant under $\LL^1$ measure equivalence\footnote{Bowen actually proves a more general statement: he shows that the volume growth is monotonous under ``integrable-embedding". His notion seems strongly related to that of $L^1$-measure subgroup, although we did not investigate this further.}. We strengthen his result as follows.

\begin{theorem}\label{thm:Bowen}
	Let  $\varphi$ be an increasing, subadditive function such that $\varphi(0)=0$ and let $\Gamma$ and $\Lambda$ be finitely generated groups. If $\Gamma$ is a $\varphi$-integrable measure sub-quotient of $\Lambda$, then 
	\[V_\Gamma(n)\preccurlyeq V_\Lambda(\varphi^{-1}(n)),\]
	where $\varphi\inv$ denotes the inverse function of $\varphi$. 
\end{theorem}

The following key lemma is stated in \cite{austinIntegrableMeasureEquivalence2016} for free actions, but the proof works as well for non-free actions.

\begin{lemma}[{\cite[Lemma~B.11]{austinIntegrableMeasureEquivalence2016}}] \label{lem:Bowen}
	Let $\Gamma$ be a finitely generated group. Let $\Gamma\curvearrowright(X, \mu)$ be a measure-preserving action on a standard probability space, and let $X_0\subseteq X$. For $x\in X_0$ let $R_{X_0}(x) \colon= \{\gamma\in \Gamma\colon \gamma\cdot x \in X_0\}$ be its associated \textbf{return time set}. Then, for every $n \in \mathbb{N}$ we have
	\[\int_{X_0} \frac{|R_{X_0}(x)\cap B_\Gamma(e_\Gamma,n)|}{V_\Gamma(n)} d\mu(x) \ge 2\mu(X_0) - 1.\]
\end{lemma}

\begin{proof}[Proof of Theorem~\ref{thm:Bowen}]
	Let $(\Omega,X_\Lambda,\mu)$ be a sub-quotient coupling from $\Gamma$ to $\Lambda$ such that for all $s\in S_{\Gamma}$,  \[\int_{X_\Lambda}\varphi(d_{S_\Lambda}(s\actx 
	x,s\actom x)) d\mu(x)<\infty.\]
	The set $X_\Lambda$ has finite measure,  
	so by replacing $\mu$ by $\frac \mu{\mu(X_{\Lambda})}$, we may as well assume that $\mu(X_\Lambda)=1$.
	We let $X_\Gamma$ be a fundamental domain for the action of $\Gamma$. 
	Since $\Gamma\actom X_\Gamma = \Omega$ there exists $R>0$ such that
	the subset $X_0=X_\Lambda\cap B_\Gamma(e_\Gamma,R)\actom X_\Gamma$ satisfies $\mu(X_0)\ge \frac{9}{10}$. 
	
	By Lemma~\ref{lem:ConcavePhiEquivalence}  there exists $C>0$ such that for every $\gamma\in \Gamma$, 
	\begin{equation}\label{ineq: length bound}
		\int_{X_\Lambda} \varphi(d_{S_\Gamma}(\gamma\actx x,\gamma\actom x))d\mu(x) \leq C |\gamma|.
	\end{equation}
	
	We then fix $n\in\N$ and define
	\[X_1 = \left\{x\in X_\Lambda\colon \sum_{\gamma\in B_{\Gamma}(e_\Gamma,n)}\frac{\varphi(d_{S_\Lambda}(\gamma\actx x,\gamma\actom x))}{|\gamma|}\leq {10CV_\Gamma(n)}\right\}.\]
	Note that by applying Markov's inequality and then inequality \eqref{ineq: length bound}, we have
	\[\mu(X_\Lambda\setminus X_1) \le 
	\frac{1}{10CV_\Gamma(n)}\int_{X_\Lambda}
	\sum_{\gamma\in B_{\Gamma}(e_\Gamma,n)}\frac{\varphi(d_{S_\Lambda}(\gamma\actx x,\gamma\actom x))}{|\gamma|} d\mu(x) \le \frac{1}{10}.\]
	For every $x \in X_\Lambda$, we let $$\Gamma_x =\{\gamma \in B_\Gamma(e_\Gamma,n) \colon \varphi(d_{S_\Lambda}(\gamma\actx x,\gamma\actom x)) \le 60C|\gamma|\}.$$
	Applying Markov's inequality to the finite probability space $B_{\Gamma}(e_{\Gamma},n)$ equipped 
	with the normalized counting measure and the definition of $X_1$, we then have for all $x\in X_1$
	\begin{equation}\label{eq:Gammax}
		|\Gamma_x|\ge \frac{5}{6} V_\Gamma(n).
	\end{equation}
	Defining $X_2 = X_1\cap X_0$, we have
	\begin{equation}\label{eq:X2}
		\mu(X_2)\ge \frac{8}{10}.
	\end{equation} We deduce from Lemma \ref{lem:Bowen} that
	\begin{equation}\label{eq:RX2}\int_{X_2} \frac{|R_{X_2}(x)\cap B_{\Gamma}(e_\Gamma,n)|}{V_\Gamma(n)}\ge 2\mu(X_2) - 1 \ge \frac{6}{10}.
	\end{equation}
	As $\Gamma_x\subseteq B_\Gamma(e_\Gamma,n)$ we have that
	\begin{align*}
		\int_{X_2} |R_{X_2}(x) \cap \Gamma_x| d\mu(x) & \ge  \int_{X_2} \big(|R_{X_2}(x) \cap B_\Gamma(e_\Gamma, n)| + |\Gamma_x| - V_\Gamma(n)\big) d\mu(x)\\
		& \ge V_\Gamma(n)\left(\frac{6}{10} + \frac{5}{6}\mu(X_2) -\mu(X_2)\right) \quad \text{by (\ref{eq:Gammax}) and (\ref{eq:RX2})}\\
		& \ge \frac{13}{30}V_\Gamma(n).  
	\end{align*}
	On the other hand
	\begin{align*}
		\int_{X_2} |R_{X_2}(x) \cap \Gamma_x| d\mu(x) & \le \int_{X_2} \left|\{\gamma\in R_{X_2}(x) \colon  \varphi(d_{S_\Lambda}(\gamma\actx x,\gamma\actom x)) \le 60C n\}\right| d\mu(x)\\
		& \le \sum_{\lambda\in B_\Lambda\left(e_\Lambda,\varphi^{-1}(60C n)\right)} \int_{X_2} |\{\gamma\in R_{X_2}(x) \colon \lambda \actom \gamma\actom x\in X_\Lambda \}| d\mu(x) \\
		& \le V_\Lambda\left(\varphi^{-1}(60C n)\right)V_\Gamma(R) \mu(X_2),
	\end{align*}
	where the last inequality results from the inclusion $X_2\subseteq X_0$ together with the following claim  (recall that $X_0=X_\Lambda\cap B_\Gamma(e_\Gamma,R)\actom X_\Gamma$).
	
	\begin{clai} Given $\lambda\in \Lambda$ and $x\in X_\Lambda$,  the cardinality of 
		$ \{\gamma\in R_{X_0}(x) \colon \lambda\actom \gamma\actom x\in X_\Lambda \}$ is at most $V_\Gamma(R).$
	\end{clai}
	\begin{proof} Indeed, let $\gamma$  be in the above set. We have $\lambda\actom \gamma\actom x=\gamma\actx x\in X_0$. In particular $\lambda\actom \gamma\actom x\in B_\Gamma(e_\Gamma,R)\actom X_\Gamma$, from which we deduce that  $\gamma\in B_\Gamma(e_\Gamma,R)\actom \lambda^{-1} \actom X_\Gamma$.
		The conclusion follows by freeness of the $\Gamma$-action and the fact that $\lambda^{-1}\actom X_\Gamma$ is a fundamental domain for it. 
	\end{proof}
	Exploiting our upper and lower bounds for $\int_{X_2} |R_{X_2}(x) \cap \Gamma_x| d\mu(x)$ we deduce that for all $n\in \N$,
	\[ \frac{13}{30}V_\Gamma(n)	\le V_\Lambda\left(\varphi^{-1}(60C n)\right)V_\Gamma(R),\]	
	which yields the conclusion of  the theorem.
\end{proof}

\begin{corollary}\label{cor:BowenPolynomial}
	Assume that two groups $\Gamma$ and $\Lambda$ satisfy $V_\Gamma(r)\approx r^a$ and $V_\Lambda(r)\approx r^b$, with $a<b$. Then $\Lambda$ is not an $\LL^p$ measured sub-quotient of $\Gamma$ if $p>a/b$.
\end{corollary}
\begin{proof}
	Since $a/b<1$, and $\LL^p$-integrability gets stronger as $p$ increases, it is enough to check it for $p\leq 1$. In this case, Theorem~\ref{thm:Bowen} applies as $x\mapsto x^p$ is subadditive.
\end{proof}

\section{Monotonicity of the isoperimetric profile}\label{sec:isoperimetricprofile}

In this section we state and prove a general monotonicity result satisfied by the isoperimetric profile. We obtain two different conclusions depending on whether the coupling is $\LL^p$ for some $p\geq 1$, or if it is $\varphi$-integrable for a sublinear function $\varphi$. However the first statement for $p=1$ and the second one for $\varphi(t)=t$ have the same content. 
Roughly speaking the following theorem says that the isoperimetric profile behaves well under measure quotient (and so in particular under measure equivalence). 
In order to obtain a similar statement for measure sub-quotients, and in particular for measure subgroups, we will need the following technical assumption.
\begin{definition}
	Given $m\in \N$, we say that a sub-quotient (resp. subgroup, resp quotient) coupling $(\Omega,X_\Lambda,\mu)$ from $\Gamma$ to $\Lambda$ is \textbf{at most $\boldsymbol m$-to-one} if for every $x\in X_\Lambda$ the map $\gamma\mapsto \gamma\inv \actom ( \gamma\actx x)\in \Lambda\actom  x$ has pre-images of size at most $m$.\end{definition}
\begin{remark}
	Equivalently, this condition says that for every $\lambda\in \Lambda$ and  $x\in X_\Lambda$, there are at most $m$ elements $\gamma$ such that $\gamma\actom x \in \lambda \actom X_{\Lambda}$. 
	Indeed, by definition of the $\Gamma$-action on $X_\Lambda$,  $\gamma\actom x \in \lambda \actom X_{\Lambda}$ if and only if $\gamma\actx x=\lambda^{-1}\actom \gamma\actom x$, which is also equivalent to $\lambda\actom x=\gamma\inv \actom ( \gamma\actx x)$.
\end{remark}

\begin{theorem}\label{thm: monotonicity of isop under Lp}Let $p\geq 1$, let $\Gamma$ and $\Lambda$ be finitely
	generated groups.
	Assume that we are in one of the following situations. 
	\begin{itemize}
		\item[(i)] $\Gamma$ is an $\LL^p$ measure quotient of $\Lambda$;
		\item[(ii)] $\Gamma$ is an at most $m$-to-one $\LL^p$ measure 
		sub-quotient of $\Lambda$.
	\end{itemize}
	Then their $\ell^p$ isoperimetric profiles satisfy the following inequality: \[j_{p,\Gamma}\succcurlyeq j_{p,\Lambda}.\]
\end{theorem}
\begin{theorem}\label{thm: monotonicity of isop under phi}
	Let $\varphi\colon(0,\infty)\to (0,\infty)$ be a function  such that $\varphi$ and $t\mapsto t/\varphi(t)$ are non-decreasing, let $\Gamma$ and $\Lambda$ be finitely generated groups. Assume that we are in one of the following situations. 
	\begin{itemize}
		\item[(i)]  $\Gamma$ is a $\varphi$-integrable measure quotient of $\Lambda$;
		\item[(ii)] $\Gamma$ is an at most $m$-to-one $\varphi$-integrable 
		measure  sub-quotient of $\Lambda$
	\end{itemize}
	Then their isoperimetric profiles satisfy the following inequality:
	\[j_{1,\Gamma}\succcurlyeq \varphi\circ j_{1,\Lambda}.\]
\end{theorem}
\begin{remark}\label{remark:Folner}
	Formulated with the F\o lner function instead of the isoperimetric profile, we obtain $\text{F\o l}_\Gamma\circ \varphi\preccurlyeq \text{F\o l}_\Lambda $.
\end{remark}

\begin{remark}
	We do not know whether the condition of being ``at most $m$-to-one'' is necessary. Fortunately, we shall see in Section \ref{sec:Linfty} that it is satisfied in a case of interest: when $\Gamma$ admits a regular map to $\Lambda$, then $\Gamma$ turns out to be an at most $m$-to-one $\LL^{\infty}$ measure subgroup of $\Lambda$. 
\end{remark}


We deduce the following corollary.

\begin{corollary}\label{cor:isoperimetricPolynomialGrowth}
	Assume that $\Gamma$ and $\Lambda$ have polynomial growth of degree $b$ and $a$ respectively, with $b>a$, and let $F$ and $K$ be non-trivial finite groups. If $p>a/b$, then
	\begin{itemize}
		\item $\Gamma$ is not an $\LL^p$ measured quotient (nor an at most $m$-to-one $\LL^p$ measured sub-quotient) of $\Lambda$;
		\item $F\wr \Gamma$ is not an $\LL^p$ measured quotient (nor an at most $m$-to-one $\LL^p$ measured  sub-quotient) of $K\wr \Lambda$. 
	\end{itemize}
\end{corollary}

\begin{proof}
	We recall the following estimates. \begin{itemize}
		\item $j_{1,\Gamma}(n)\approx n^{1/d}$ for a group of polynomial growth of degree $d$ (e.g.\ for $\Gamma=\Z^d$) (see  \cite{coulhonRandomWalksGeometry2000});
		\item $j_{1,\Gamma}(n)\approx (\log n)^{1/d}$ for a group of the form $F\wr \Sigma$, where $F$ is a non-trivial finite group, and $\Sigma$ has polynomial growth of degree $d$ \cite{erschlerIsoperimetricProfilesFinitely2003}.
	\end{itemize} 
	We obtain both statements by confronting Theorem~\ref{thm: monotonicity of isop under phi}(ii) to these estimates (and using the monotonicity of $p$-integrability).
\end{proof}

\begin{remark}
	The first item is slightly weaker than the conclusion of Corollary \ref{cor:BowenPolynomial} obtained by comparing the volume growths (as the latter does not require the at most $m$-to-one assumption). 
\end{remark}

We end this section discussing some important situations where the condition of being at most $m$-to-one is satisfied.

\begin{proposition}
	A measure sub-quotient coupling $(\Omega,X_\Lambda,\mu)$  from $\Gamma$ to $\Lambda$ is at most $1$-to-one if and only if $X_\Lambda$ intersects every $\Gamma$-orbit at most once.
\end{proposition}
\begin{proof}
	Suppose first that $X_\Lambda$ intersects each $\Gamma$-orbit at most once. Let $x\in X_\Lambda$, suppose that $\gamma_1\inv\actom\gamma_1\actx x=\gamma_2\inv\actom\gamma_2\actx x$.  Then $\gamma_1\cdot x$ and $\gamma_2\cdot x$ are two elements of the same $\Gamma$-orbit which belong to $X_\Lambda$. So by our assumption they are equal. Since the $\Gamma$-action on $\Omega$ is free, we conclude that $\gamma_1=\gamma_2$ as wanted.
	
	Conversely, suppose that the coupling is at most $1$-to-one. Let $x\in X_\Lambda$, suppose that $\gamma\actom x\in X_\Lambda$ for some $\gamma\in\Gamma$.
	Then $\gamma\cdot x=\gamma\actom x$, so $\gamma\inv\actom \gamma\actx x=x$, and so by injectivity $\gamma=e_\Gamma$. We conclude that $\gamma\actom x=x$, so $X_\Lambda$ intersects each $\Gamma$-orbit at most once as announced.
\end{proof}

An important example of at most $1$-to-one coupling is provided by the notion of \emph{orbit sub-quotient couplings} (see Section \ref{sec: OE couplings}). In fact, Proposition \ref{prop:OrbitSubquotient} implies that every at most $1$-to-one measure sub-quotient coupling can be turned into such a coupling after making the induced action free. In particular, we have

\begin{corollary}\label{cor:Orbit sub-quotient Isop}
	Let $\varphi\colon(0,\infty)\to (0,\infty)$ be a function  such that $\varphi$ and $t\mapsto t/\varphi(t)$ are non-decreasing. Assume that $\Gamma$ is a $\varphi$-integrable orbit  sub-quotient of $\Lambda$. Then  $j_{1,\Gamma}\succcurlyeq \varphi\circ j_{1,\Lambda}.$
	\qed
\end{corollary}


\subsection{Overview of the proof}\label{sec:overviewisop}
Let us start describing the main steps of the proof when the $\Lambda$-action 
is free.

The first step is to start from a finitely supported function $f$ on $\Lambda$ 
with small \emph{right} gradient and extend it to a function $\tilde{f}$ on the 
coupling space $\Omega$ through its identification with $X_{\Lambda}\times 
\Lambda$. In other words, we identify $\Omega$ to a field copies of $\Lambda$ 
which are pointed at $X_\Lambda$. The key 
observation is that the function $\tilde f$ has small gradient with respect 
to the $\Gamma$-action. This is because each element $\gamma\in\Gamma$ induces 
an 
isomorphism of $\Lambda$-spaces, only shifting the basepoint by an amount which 
is precisely controlled by our integrability condition: $\gamma$ takes the 
field of copies of $\Lambda$ pointed at $X_\Lambda$ to the field of copies of 
$\Lambda$ pointed at $\gamma\actom X_\Lambda$. The precise statement which 
allows to control the difference between  $\gamma\cdot\tilde f$ and $\tilde f$ 
is provided by Lemma~\ref{lem: pushing through 
	different points}.

Once this is established, we identify $\Omega$ with $\Gamma \times X_{\Gamma}$ 
and view $\tilde{f}$ as a family $(f^x)$ of finitely supported functions on 
$\Gamma$. By the previous fact, the gradient of these functions is well behaved 
on a large measure subset of $X_\Gamma$. We need to simultaneously control from 
below the $\ell^p$-norm of the function, and control from above the size of the 
support of these functions. 
Let us describe the simpler situation where the coupling is assumed to be at most $m$-to-one. This condition immediately implies that the size of the support of $f^x$ is (almost surely) at most $m$ times the size of the support of $f$. On the other hand, a pigeonhole argument allows us to find a set of 
positive measure of $X_\Gamma$ on which the lower bound on the norm is 
satisfied simultaneously with the upper bound on the gradient. This is enough to conclude for the existence of a subset of positive measure for which $f^x$ satisfies all three conditions simultaneously and therefore for the proof of the theorem.

If we do not assume that  the coupling is at most $m$-to-one, a difficulty 
arises that is overcome in Section \ref{sec:L1 measure quotient iso profile down}. What remains possible using pigeonhole arguments is the following: we can find a subset of $X_\Gamma$  with measure---say at least $1/2$---where an upper bound on both the size of the support of $f^x$ and the norm of its gradient hold simultaneously. However, the issue comes from the lower bound on the norm of $f^x$, which may be only valid on a subset of arbitrarily small measure. 
The solution that we find is rather subtle and consists in slightly modifying $f$ in such a way that the lower bound on the norm of $f^x$ is 
actually satisfied on a subset of $X_\Gamma$ of relatively large measure (see Remark \ref{rem:explanation} for a more detailed overview of the argument).  

Finally, when the action of $\Lambda$ is not free, we still find a way to define a function $\tilde{f}$ on $\Omega$ which in some sense extends $f$. The rest of 
proof is identical.

\subsection{\texorpdfstring{$\Lambda$}{Lambda}-gradients on 
	\texorpdfstring{$\Lambda$}{Lambda}-spaces}\label{sec:GammaGradient}

It will be convenient to define a notion of gradient on functions defined on measure spaces equipped with a measure-preserving group action. 
\begin{definition}
	Let $\Lambda$ be a finitely generated group equipped with a finite symmetric generating set 
	$S_\Lambda$. Given any standard measured $\Lambda$-space $(\Omega,\mu)$ and any function 
	$f\in \LL^p(\Omega,\mu)$, define its gradient to be the function 
	$\nabla_{S_\Lambda} f\colon S_{\Lambda} \times \Omega\to \R$ defined by 
	\[\nabla_{S_\Lambda} f(s,x)=f(x)-f(s\inv \actom x).\]
\end{definition}
\begin{remark}\label{rmk: fubini for gradient}
	We have
	$$\norm{\nabla_{S_\Lambda} f}_p^p=\int_\Omega\sum_{s\in S_\Lambda} \abs{f(\omega)-f(s\inv \actom \omega)}^pd\mu(\omega)=\sum_{s\in S_\Lambda}\norm{f-s\actom f}_p^p,$$
	where $s\actom f(x)=f(s\inv\actom x)$. 
	Note that by symmetry of $S_\Lambda$, we also have
	$$\norm{\nabla_{S_\Lambda} f}_p^p=\int_\Omega\sum_{s\in S_\Lambda} \abs{f(\omega)-f(s\actom \omega)}^pd\mu(\omega)$$
\end{remark}

There are two natural ways to view $\Lambda$ as a $\Lambda$-space (action by left or right multiplication), so we define the  \emph{left} $S_\Lambda$-gradient $f$ as the function $\nabla^l_{S_\Lambda} f\colon S_\Lambda\times \Lambda\to \R $ by the formula 
\[\nabla^l_{S_\Lambda} f(s,\lambda)=f(\lambda)-f(s\inv \lambda);\]
and the \emph{right} $S_\Lambda$-gradient $f$ by 
\[\nabla^r_{S_\Lambda} f(s,\lambda)=f(\lambda)-f(\lambda s).\]
The isoperimetric profile can be alternatively defined using the left or right 
gradient. This clearly does not change its value as the inverse map intertwines the actions by left and right multiplication. We will exploit this remark in our proof as follows: we will start with a function 
on $\Lambda$ with 
small \emph{right} gradient and induce on $\Gamma$ a function of comparable 
support with small \emph{left} gradient.

In the following lemma, we consider a $\Gamma$-space $X$. Starting with a function defined on $\Lambda$, we define a function on $X$ depending on some basepoint $x$. The main objective of the lemma is to control the $\LL^p$-norm of the difference between functions associated to different base points. It is reminiscent of the fact that right Reiter functions 
on a group $\Lambda$ can be pushed to Reiter functions on any equivalence 
relation induced by a $\Lambda$-action (see e.g. \cite[Prop. 
9.2]{kechrisTopicsOrbitEquivalence2004}).

\begin{lemma}\label{lem: pushing through different points}
	Let $X$ be a transitive $\Lambda$-set and let $x_0\in X$.
	Let $p\geq 1$. The map $\R^\Lambda\to \R^X$ which associates to every $f\in\R^\Lambda$ the function $f_{x_0}$ given by $$f_{x_0}(y)=\left(\sum_{\lambda\colon \gamma\actom x_0=y}|f(\lambda)|^p\right)^{1/p}$$
	satisfies $\norm{f_{x_0}}_p=\norm{f}_p$. Moreover, if $S_\Lambda$ is a 
	finite generating set for $\Lambda$ and if $f\in \ell^p(\Lambda)$ then for all $x_1\in X,$
	\begin{equation}\label{ineq: change point in orbit p>1}
		\norm{f_{x_0}-f_{x_1}}_p
		\leq d_{S_\Lambda}(x_0,x_1)\norm{\nabla^r_{S_\Lambda} f}_p.
	\end{equation}
\end{lemma}
\begin{proof}
	The verification that $\norm{f_{x_0}}_p=\norm{f}_p$ is immediate. To obtain the inequality \eqref{ineq: change point in orbit p>1}, note that for each $s\in S_\Gamma$ and $x_0\in X$, we have by the triangle difference inequality
	\begin{align*}\norm{f_{x_0}-f_{s\cdot x_0}}_p^p
		&=\sum_{y\in X}\abs{\left(\sum_{\lambda\colon \lambda\actom  x_0=y }|f(\lambda)|^p\right)^{1/p}-\left(\sum_{\lambda\colon \lambda s\actom x_0=y }|f(\lambda)|^p\right)^{1/p}}^p\\
		&=\sum_{y\in X}\abs{\left(\sum_{\lambda\colon \lambda\actom  x_0=y }|f(\lambda)|^p\right)^{1/p}-\left(\sum_{\lambda\colon \lambda\actom x_0=y }|f(\lambda s\inv)|^p\right)^{1/p}}^p\\
		& \leq \sum_{y\in X}\sum_{\lambda\colon \lambda\actom x_0=y }|f(\lambda)-f(\lambda s\inv)|^p\\
		&\leq \sum_{\lambda\in \Lambda}\abs{f(\lambda)-f(\lambda s\inv)}^p,
	\end{align*}
	so by the definition of the gradient we have $\norm{f_{x_0}-f_{s\actom  
			x_0}}_p\leq \norm{\nabla^r_{S_\Lambda} f}_p$.  The conclusion now follows 
	using the triangle inequality for the $\LL^p$-norm and the fact that 
	$\Gamma$ is acting by isometries on $\ell^p(X)$.
\end{proof}

We also need the following variant of Lemma~\ref{lem: pushing through different points}.

\begin{lemma}\label{lem: pushing through different points l1}
	Let $\varphi\colon(0,\infty)\to (0,\infty)$ be a function  such that both $\varphi$ and $t\mapsto t/\varphi(t)$ are non-decreasing. Given a transitive $\Lambda$-set $X$ and a point $x_0\in X$, the map $\R^\Lambda\to \R^X$ which associates to every $f\in\R^\Lambda$ the function $f_{x_0}$ given by $$f_{x_0}(y)=\sum_{\lambda\colon \lambda\actom x_0=y}|f(\lambda)|$$
	satisfies $\norm{f_{x_0}}_1= \norm{f}_1$. Moreover, if $S_\Lambda$ is a  finite generating set for $\Lambda$, $x_0,x_1\in X$ and $f$ is a finitely supported function on $\Lambda$ such that $\norm{\nabla^r_{S_\Lambda}f}_1=1$, then 
	\begin{equation}\label{ineq: change point in orbit p=1}
		\frac{ \norm{f}_1}{\norm{f_{x_0}-f_{x_1}}_1}\geq \frac{\varphi(\norm{f}_1)}{2\varphi(d_{S_\Lambda}(x_0,x_1))}.
	\end{equation}
\end{lemma}
\begin{proof}
	By triangle inequality, we obtain \[\norm{f_{x_0}-f_{x_1}}_1\leq 2\max_{i=0,1}\|f_{x_i}\|_1=2\norm{f}_1.\]
	Using the monotonicity of $t/\varphi(t)$ we have
	\[\norm{f_{x_0}-f_{x_1}}_1= \varphi(\norm{f_{x_0}-f_{x_1}}_1) \frac{\norm{f_{x_0}-f_{x_1}}_1}{\varphi(\norm{f_{x_0}-f_{x_1}}_1)}\leq \varphi(\norm{f_{x_0}-f_{x_1}}_1) \frac{2\norm{f}_1}{\varphi(2\norm{f}_1)}.\]
	Now, since $\norm{\nabla f}_1=1$, by applying  Lemma~\ref{lem: pushing through different points} for $p=1$ and using the monotonicity of $\varphi$ we get
	\[\norm{f_{x_0}-f_{x_1}}_1\leq  2\varphi(d_{S_\Gamma}(x_0,x_1))\frac{\norm{f}_1}{\varphi(\norm{f}_1)}\]
	So the lemma follows. 
\end{proof}

\subsection{The induction technique}\label{sec:induction}

In this section, we shall prove that if $p\geq 1$ and $\Gamma$ is an $\LL^p$ measure \emph{sub-quotient} of $\Lambda$, then every $\ell^p$ function on $\Lambda$ induces a function on $\Omega$ whose $S_\Gamma$-gradient is well behaved (and a similar statement for $\varphi$-integrable measure sub-quotient). To deal with sub-quotients, we need a non-free analogue of the fact that if $\Lambda$ is acting freely on $\Omega$ and $X_\Lambda$ is a fundamental domain, the map $(x,\lambda)\mapsto \lambda\actom  x$ is a measure-preserving bijection between $X_\Lambda\times \Lambda$ and $\Omega$.

\begin{lemma}\label{lem: Omega as fibered space}
	Suppose $X_\Lambda$ is a fundamental domain for a measure-preserving $\Lambda$-action on a standard measured space $(\Omega,\mu)$. Then for every Borel $A\subseteq \Omega$, we have 
	$$\mu(A)=\int_{X_\Lambda}\abs{( \Lambda\actom  x)\cap A}d\mu(x).$$
	In particular, for every measurable function $f\colon \Omega\to\R$, we have
	$$\int_\Omega f\;d\mu= \int_{X_\Lambda}\sum_{y\in \Lambda\actom  x} f(y)\;d\mu(x).$$
\end{lemma}
\begin{proof}
	Since every subset of $\Omega$ can be written as a countable disjoint union of $\Lambda$-translates of Borel subsets of $X_\Lambda$, it suffices to show that the right term defines a Borel measure on $\Omega$ which is $\Lambda$-invariant and coincides with $\mu$ when restricted to $X_\Lambda$. The fact that the formula
	$$m(A)=\int_{X_\Lambda}\abs{( \Lambda\actom  x)\cap A}d\mu(x).$$
	defines a Borel measure follows from the fact that the map $x\mapsto \abs{( \Lambda\actom  x)\cap A}$ is Borel. $\Lambda$-invariance is clear, and the fact that the two measures coincide when restricted on $X_\Lambda$ is also straightforward to check. 
\end{proof}

\begin{proposition}[Monotonicity of the 
	$\ell^p$-gradient under $\LL^p$ measured 
	sub-quotient]\label{prop:gradientLpSubquotient}
	Let $p\geq 1$.
	If $\Gamma$ is an $\LL^p$ measure sub-quotient of $\Lambda$ via a 
	coupling $(\Omega,X_\Lambda,\mu)$, and if 
	$f\in\ell^p(\Lambda)$, then the induced function $\tilde f$ on $\Omega$ 
	defined by $\tilde f(\omega)=\left(\sum_{\lambda\colon 
		\lambda\inv\actom\omega\in X_\Lambda}|f(\lambda)|^p\right)^{1/p}$ 
	satisfies \[\norm{\tilde{f}}_p^p=\norm{f}^p_p\mu(X_\Lambda),\] and
	\[
	\norm{\nabla_{S_\Gamma} \tilde f}_p^p\leq C \norm{\nabla^r_{S_\Lambda} f}_p^p,
	\]
	where $\displaystyle C=\abs{S_\Gamma}\max_{s\in S_\Gamma}\int_{X_\Lambda}d_{S_\Lambda}(s\actx x,s\actom x)^pd\mu(x)$.
\end{proposition}
\begin{proof} The first equality is obvious.
	For all $(x,y)\in \Omega^2$ such that $y\in \Lambda\actom x$, we let 
	$$f_x(y)=\left(\sum_{\lambda\colon \lambda\actom  x=y}|f(\lambda)|^p\right)^{1/p}.$$
	Identifying $\Omega$ with the set of pairs $(x,y)$ such that $x\in X_\Lambda$ and $y\in \Lambda\actom x$, we have
	\[\tilde{f}(x,y)=f_x(y).\]
	For every $\gamma\in \Gamma$, 
	\[\tilde f(\gamma\actom(x,y))=f_{\gamma\actx x}(\gamma\actom y).\]
	Note that since the actions of $\Gamma$ and $\Lambda$ commute, we have for all $\gamma\in \Gamma$
	\[ f_{\gamma\actom x}(\gamma\actom y)=f_x(y).\]
	Hence, denoting $\gamma\inv\actom \tilde f\colon\omega\mapsto \tilde f(\gamma\actom\omega)$, 
	\[\gamma\inv\actom \tilde f(x,y)-\tilde f(x,y)=f_{\gamma\actx x}(\gamma\actom y)-f_{\gamma\actom x}(\gamma\actom y).\]
	We can now invoke Lemma~\ref{lem: pushing through different points}: letting $s\in S_\Gamma$, we have
	\begin{eqnarray*}
		\sum_{y\in \Lambda\actom x} \abs{s\inv\actom\tilde f(x,y)-\tilde f(x,y)}^p 
		&=& \sum_{y\in \Lambda\actom x} \abs{f_{s\actx x}(s\actom y)-f_{s\actom x}(s\actom y)}^p\\
		& =& \sum_{y\in \Lambda\actom s\actom x} \abs{f_{s\actx x}(y)-f_{s\actom x}(y)}^p\\
		&\leq& d_{S_\Lambda}(s\actx x,s\actom x)^p\norm{\nabla^r_{S_\Lambda} f}_p^p
	\end{eqnarray*}
	We can now use Lemma~\ref{lem: Omega as fibered space} to compute for all $s\in S_\Gamma$
	\begin{align*}
		\norm{s\inv\actom\tilde f-\tilde f}_p^p&=\int_{X_\Lambda}\sum_{y\in \Lambda\actom x} \abs{s\inv\actom\tilde f(x,y)-\tilde f(x,y)}^pd\mu(x)\\
		&\leq \int_{X_\Lambda}d_{S_\Lambda}(s\actx x,s\actom x)^p\norm{\nabla^r_{S_\Lambda} f}_p^pd\mu(x)
	\end{align*}
	The desired inequality now follows by the definition of $\nabla_{S_\Gamma}\tilde f$ and the symmetry of $S_\Gamma$.
\end{proof}

\begin{proposition}[Monotonicity of the 
	$\ell^1$-gradient under $\varphi$-integrable measured sub-quotient] 
	\label{prop:gradientSubquotient(phi)}
	Let $\varphi\colon(0,\infty)\to (0,\infty)$ be a function such that $\varphi$ and $t/\varphi(t)$ 
	are non-decreasing.  If $\Gamma$ is a $\varphi$-integrable measure sub-quotient of $\Lambda$ via a coupling 
	$(\Omega,X_\Lambda,\mu)$, and if $f$ is a finitely supported function on $\Lambda$ such that $\norm{\nabla^r_{S_\Lambda} f}_1=1$, then the induced function $\tilde f$ on $\Omega$ defined by $\tilde f(\omega)=\sum_{\lambda\colon \lambda\inv\actom\omega\in X_\Lambda}f(\lambda)$ satisfies 
	\[
	\frac{\norm{\nabla_{S_\Gamma} \tilde f}_1}{\norm{\tilde f}_1}\leq \frac{2C}{\varphi(\norm{f}_1)},
	\]
	where	$\displaystyle C=\frac{\abs{S_\Gamma}}{\mu(X_\Lambda)}\max_{s\in S_\Gamma}\int_{X_\Lambda}\varphi(d_{S_\Lambda}(s\actx x,s\actom x))d\mu(x)$.
\end{proposition}
\begin{proof}
	The proof is similar to that of Proposition \ref{prop:gradientLpSubquotient}: using Lemma~\ref{lem: pushing through different points l1} applied to $f_x(y)\coloneqq \sum_{\lambda\colon \lambda\actom  x=y}|f(\lambda)|$, we obtain 
	\[
	\frac{\sum_{y\in \Lambda\actom x} \abs{s\inv\actom\tilde f(x,y)-\tilde f(x,y)}}{\norm{f}_1}\leq \frac{2\varphi(d_{S_\Lambda}(s\actx x,s\actom x))}{\varphi(\norm{f}_1)},
	\]
	which together with Lemma~\ref{lem: Omega as fibered space} and the fact that $\|\tilde f\|_1=\mu(X_\Lambda)\norm f_1$ yields the desired inequality.
\end{proof}
\begin{remark}\label{rmk: gradient rewritten for SG}
	Using Fubini's theorem and the natural identification of $X_\Gamma\times \Gamma$ with $\Omega$ given by $(x,g)\mapsto  g\actom  x$, we may rewrite 	$\norm{\nabla_{S_\Gamma} \tilde f}_p^p$ as
	\begin{align*}
		\norm{\nabla_{S_\Gamma} \tilde f}_p^p&=\int_{X_\Gamma}\sum_{s\in S_\Gamma}\sum_{\gamma\in \Gamma}\abs{\tilde f( \gamma\actom  x)-\tilde f(s \gamma\actom  x)}^pd\mu(x).
	\end{align*}
	Now for each $x\in X_\Gamma$, the function $\tilde f$ defines a function 
	$f^x$ on $\Gamma$ given by $f^x(\gamma)=\tilde f( \gamma\actom  x)$, and the previous 
	equality may thus be rewritten as
	\begin{align*}
		\norm{\nabla_{S_\Gamma} \tilde f}_p^p&=\int_{X_\Gamma}\norm{\nabla^l_{S_\Gamma}f^x}^p_pd\mu(x).
	\end{align*}
	So the conclusion of the two previous propositions is that under the right 
	assumption, any function $f$ on $\Lambda$ of small \emph{right} 
	$\ell^p$-gradient induces functions $f^x$ on $\Gamma$ which have on average 
	a small \emph{left} $\ell^p$ gradient. From there on, as we will see in the 
	next section, it is not hard to conclude the proof that the isoperimetric 
	profile goes down for $p\geq 1$ under the assumption that the coupling is 
	finite-to-one. 
	
	But without the finite-to-one assumption, we lose the uniform control on the size of the support of the functions $f^x$. We will circumvent this by simultaneously controlling the size of the support of $f^x$ and bounding  its norm from below on a large portion of the fundamental domain of $\Gamma$. This will be done in Section \ref{sec:L1 measure quotient iso profile down}.
	
\end{remark}

\subsection{Monotonicity under at most \texorpdfstring{$m$}{m}-to-one measure sub-quotients}

In this section we shall prove the second items of Theorem~\ref{thm: monotonicity of isop under Lp} and Theorem~\ref{thm: monotonicity of isop under phi}.

The common feature between these statements is that the coupling is supposed to 
be at most $m$-to-one. This has the following consequence. 

\begin{lemma}\label{lem:finite-to-one}
	Let $(\Omega,X_\Lambda,\mu)$ be an at most $m$-to-one measure 
	sub-quotient coupling between $\Gamma$ and $\Lambda$. 
	Let $X_\Gamma$ be a fundamental domain for the $\Gamma$-action.
	Let $f$ be a function 
	on $\Lambda$ whose support has cardinality at most $K$, and let 
	$(f^x)_{x\in X_\Gamma}$ be the family of functions on $\Gamma$ defined by
	$$f^x(\gamma)=\left(\sum_{\lambda\colon \lambda\actom  ( \gamma\actom  x)\in 
		X_\Lambda}|f(\lambda)|^p\right)^{1/p},$$
	for some $p>0$. 
	Then for each $x\in X_\Gamma$, the function $f^x$ has support of cardinality at most $mK$. 
\end{lemma}
\begin{proof}
	Let $x\in X_\Gamma$, let $\lambda_0\in \Lambda$ such that $\lambda_0\actom x\in X_\Lambda$. 
	By definition for every $\gamma\in \Gamma$ and $\lambda\in\Lambda$ we have $ \lambda\actom  ( \gamma\actom  x)\in X_\Lambda$ if and only if $ \lambda\actom ( \gamma\actom  x)=\gamma\cdot (\lambda_0\actom x)$, which since the $\Gamma$ and $\Lambda$-actions on $\Omega$ commute is in turn equivalent to $ \lambda\actom  x=\gamma\inv \actom(\gamma\cdot (\lambda_0\actom x))$. So we have 
	$$f^x(\gamma)=\sum_{\lambda\colon \lambda\actom  x=\gamma\inv \actom(\gamma\cdot (\lambda_0\actom x))}f(\lambda).$$
	The conclusion now follows from  our assumption that the map $\gamma\mapsto 
	\gamma\inv\actom(\gamma\cdot (\lambda_0\actom x))$ is at most $m$-to-one. 
\end{proof}

\begin{proof}[{Proof of Theorem~\ref{thm: monotonicity of isop under Lp}(ii)}]
	We  start with a function $f$ that realizes the $\LL^p$-isoperimetric profile of 
	$\Lambda$, and we consider the function $\tilde{f}$ on $\Omega$ defined in 
	Proposition \ref{prop:gradientLpSubquotient}. 
	By Proposition \ref{prop:gradientLpSubquotient}, there exists $C'$ only 
	depending on the coupling such that 
	\begin{equation}\label{eq:FinalGradient}
		\frac{\|\nabla_{S_\Gamma} \tilde f\|_p^p}{\|\tilde{f}\|_p^p}\leq C'\frac{\norm{\nabla^r_{S_\Lambda} f}_p^p}{\norm{f}_p^p}.
	\end{equation}
	This implies that on a set of positive measure, the function $f^x$ on $\Gamma$ satisfies
	\begin{equation}\label{eq:fxGradient}
		\frac{\|\nabla^l_{S_\Gamma} f^x\|_p^p}{\norm{f^x}_p^p}< 2C'\frac{\norm{\nabla^r_{S_\Lambda} f}_p^p}{\norm{f}_p^p}.
	\end{equation}
	Indeed, assume by contradiction that the reverse inequality holds on a subset $Z\subseteq X_\Gamma$ of full measure. This means that for all $x\in Z$,
	\[\|\nabla^l_{S_\Gamma} f^x\|_p^p\geq 2C'\frac{\norm{\nabla^r_{S_\Lambda} f}_p^p}{\norm{f}_p^p}{\norm{f^x}_p^p}.\]
	But integrating over $x\in Z$, we get a contradiction with (\ref{eq:FinalGradient}).
	On the other hand, by Lemma~\ref{lem:finite-to-one}, the support of $f^x$ has size at most $m|\supp(f)|$. 
	Hence we deduce that  $ j_{p,\Gamma}(mn)\geq \frac{j_{p,\Lambda}(n)}{2C'}$, so we are done.
\end{proof}

\begin{remark}
	Note that the previous argument does not provide any information on the measure of the set on which inequality (\ref{eq:fxGradient}) holds. 
\end{remark}

\begin{proof}[{Proof of Theorem~\ref{thm: monotonicity of isop under phi}(ii)}]
	As before we start with a function $f$ that realizes the $\LL^1$-isoperimetric profile of $\Lambda$. We then normalize $f$ such that $\norm{\nabla^r_{S_\Lambda} f}_1=1$  and consider the function $\tilde{f}$ on $\Omega$ defined in Proposition \ref{prop:gradientLpSubquotient}. 
	We deduce from Proposition \ref{prop:gradientSubquotient(phi)} that  
	\[\frac{\|\nabla_{S_\Gamma} \tilde f\|_1}{\|\tilde{f}\|_1}\leq \frac{C'}{\varphi(\norm{f}_1)},\]
	and the rest of the proof is identical. 
\end{proof}

\subsection{Monotonicity under measure quotients}\label{sec:L1 measure quotient iso profile down}
In this section we prove the first items of Theorem~\ref{thm: monotonicity of 
	isop under Lp} and Theorem~\ref{thm: monotonicity of isop under phi}. On
rescaling the measure, we may assume that $X_\Gamma$ has measure $1$.
Note that 
since $X_\Lambda$ intersects every $\Lambda$-orbit, we can then
find a finite subset $W$  of $\Lambda$ so that $Z\coloneqq (W\actom 
X_\Lambda)\cap X_\Gamma$ has measure at least $3/4$. Let us fix once and for 
all such a set $W$.

Given $n\geq 1$, we let $f_1$ be a function that realizes the (right) isoperimetric 
profile of $\Lambda$ at $n$. We start modifying it as follows. We define a function 
$f_2$ on $\Lambda$ by \[f_2(\lambda)=\left(\sum_{w\in 
	W}\abs{f_1(w\lambda)}^p\right)^{1/p}.\] 
The important fact 
is that $f_2$ ``almost" realizes the  isoperimetric profile of $\Lambda$: the 
support of $f_2$ has size at most $|W|$ times the support of $f_1$, and we have
$\|f_2\|_p\geq \|f_1\|_p$. Using the same inequalities as in the proof of 
Lemma~\ref{lem: pushing through different points} we obtain
\begin{equation}\label{eq:nablaf'}
	\|\nabla^r_{S_\Lambda} f_2\|_p^p\leq \abs W\|\nabla^r_{S_\Lambda} f_1\|_p^p.
\end{equation}

For  all $p>0$, we set
\[\tilde f_2(\omega)=\left(\sum_{\lambda\colon \lambda\actom \omega\in 
	X_\Lambda}f_2(\lambda)^p\right)^{1/p}.\]
As in the previous section, this provides for each $x\in X_\Gamma$ a function $f_2^x$ on $\Gamma$ given  by \[f_2^x(\gamma)=\tilde f_2(\gamma\actom x)=\left(\sum_{\lambda\colon\lambda\actom \gamma\actom x\in X_\Lambda}f_2(\lambda)^p\right)^{1/p}.\] 
Therefore, under the assumption that $p \geq 1$ (in order to prove Theorem~
\ref{thm: monotonicity of isop under Lp}.(i)), Proposition 
\ref{prop:gradientLpSubquotient} and Remark \ref{rmk: gradient rewritten for SG} imply  that for all $\eps>0$,
we find two subsets $V$ and $V'$ of $X_\Gamma$ of measure $1-\eps$, such that 
for all $x\in V$,
\[\norm{\nabla_{S_\Gamma}^l f_2^x}_p^p\leq \frac{C}{\eps}\norm{\nabla_{S_\Gamma}^r f_2}_p^p,\]
and for all $x\in V'$
\[|\supp(f_2^x)|\leq \frac{1}{\eps}\mu(X_\Lambda)|\supp(f_2)|.\]
For the proof of Theorem~\ref{thm: monotonicity of isop under phi}, we normalize $f_2$ such that $\norm{\nabla_{S_\Gamma} f_2}_1=1$. Using Proposition \ref{prop:gradientSubquotient(phi)} instead of Proposition \ref{prop:gradientLpSubquotient},
we find two subsets $V$ and $V'$ of $X_\Gamma$ of measure $1-\eps$, such that 
for all $x\in V$,
\[\frac{\norm{\nabla_{S_\Gamma}^lf_2^x}_1}{\norm{\tilde{f}_2}_1}\leq \frac{2C}{\eps\varphi(\norm{ f_2}_1)},\]
and for all $x\in V'$
\[|\supp(f_2^x)|\leq \frac{1}{\eps}\mu(X_\Lambda)|\supp(f_2)|.\]
In what follows, it will be sufficient to take $\eps=1/5$.
The proof of Items (i) of Theorem~\ref{thm: 
	monotonicity of isop under Lp} and Theorem~\ref{thm: monotonicity of isop under 
	phi} will be finished if we can find a subset $Y\subseteq X_\Gamma$ of measure $> 
2/5$ on which the $\ell^p$-norm 
of $f_2^x$ satisfies a uniform lower bound linear in $\|f_2\|_p$. Indeed, this ensures that for all $x$ in the non-empty set $V\cap V'\cap Y$, the function $f_2^x$ gives the required bound on the (left) isoperimetric profile of $\Gamma$.
The existence of such a subset is given by the following lemma, whose proof occupies the rest of the 
subsection. This is where we use our initial assumption that the set $Z=(W\actom 
X_\Lambda)\cap X_\Gamma$ has measure at least $3/4$.

\begin{lemma}\label{lem:LowerBoundnorm}
	There exists  a measurable subset $Y\subseteq X_\Gamma$ of measure at least $1/2$ such that for all $x\in Y$,
	\[\norm{f_2^x}_p^p\geq 
	\frac{1}{4}\norm{f_1}_p^p.\]
\end{lemma}

\begin{remark}\label{rem:explanation}
	Before proceeding to the proof of the lemma, let us briefly explain its main idea.
	The  issue that prevents us from obtaining such a lower bound directly from the lower bound on the norm of $\tilde{f_2}$ comes from the fact that we cannot rule out the possibility that in the integral
	\[\norm{\tilde{f_2}}_p^p=\int_{X_\Gamma}\norm{f_2^x}_p^pd\mu (x),\]
	most of the contribution comes from a very tiny portion of $X_\Gamma$ where 
	$\norm{f_2^x}_p$ is much bigger than $\norm{f_2}_p$. If we could 
	combine this lower bound with an uniform upper bound  $\norm{f_2^x}_p^p\leq 
	K\norm{f_2}_p^p$, then 
	this problem would not arise. The reason why such an upper bound is a priori 
	not available is due to the fact that the $\Gamma$-orbit of $x\in X_\Gamma$ 
	could potentially meet many times a same translate of $X_\Lambda$. In other 
	words, when writing $\norm{f_2^x}_p^p=\sum_\gamma\sum_{\lambda\colon \lambda\actom   
		\gamma\actom  
		x\in X_\Lambda} f_2(\lambda)^p=\sum_\lambda n_x(\lambda)f_2(\lambda)^p,$ the ``multiplicities'' $n_x(\lambda)\in \N$ might be very large.  
	Therefore, our strategy consists in replacing this expression by a sum where 
	all the multiplicities are at most one. To be more precise, we will replace it 
	by 
	$\sum_{\lambda : \lambda\cdot x\in Z}\abs{f_1(\lambda)}^p$. We first have to 
	make 
	sure that this new expression is controlled from above by 
	$\norm{\tilde{f_2}}_p^p$ (Inequality
	\ref{ineq: lower bound on f2}) and check that after integrating over 
	$X_\Gamma$, this new expression still captures a reasonable 
	fraction of $\norm{f_1}^p_p$ (Inequality \ref{clai:integraLlowerBound}). The 
	conclusion will then follow by confronting 
	this lower bound on the integral with the pointwise trivial upper bound 
	$\sum_{\lambda : \lambda\cdot x\in Z}\abs{f_1(\lambda)}^p\leq \norm{f_1}_p^p$ 
	(Claim \ref{clai:last}).  \end{remark}
\begin{proof} 
	We start by establishing the following upper bound.
	\begin{equation}\label{ineq: lower bound on f2}
		\sum_{\lambda : \lambda\cdot x\in Z}\abs{f_1(\lambda)}^p\leq 
		\norm{f_2^x}_p^p.
	\end{equation}
	
	First note that for all $x\in X_\Gamma$ and $\lambda\in\Lambda$ such that $\lambda\cdot x\in Z$, there is $\gamma\in \Gamma$  such that $ \lambda\actom \gamma\actom  x\in W\actom X_\Lambda$. 
	Indeed, this follows by inspection of the definitions: $\lambda\cdot x\in Z$ means that there exists $\gamma\in \Gamma$ such that $ \lambda\actom \gamma\actom  x\in Z$, and therefore that $ \lambda\actom\gamma\actom  x\in W\actom X_\Lambda$.
	
	We deduce the following crude inequality, which we then rewrite by exchanging orders of summation:
	\begin{align*}
		\sum_{\lambda\colon \lambda\cdot x\in Z}|f_1(\lambda)|^p& \leq \sum 
		_{\lambda}\sum_{\gamma\colon \lambda\actom   \gamma\actom  x\in W\actom 
			X_\Lambda}\abs{f_1(\lambda)}^p\\
		&\leq \sum_\gamma\sum_{\lambda\colon \lambda\actom   \gamma\actom  x\in W\actom X_\Lambda}\abs{f_1(\lambda)}^p.
	\end{align*}
	We can bound the  above sum by 
	\begin{align*}
		\sum_\gamma\sum_{w\in W}\sum_{\lambda\colon w^{-1} \lambda\actom   \gamma\actom  x\in X_\Lambda}\abs{f_1(\lambda)}^p
		&=\sum_\gamma\sum_{w\in W}\sum_{\lambda\colon \lambda\actom   \gamma\actom  x\in X_\Lambda}\abs{f_1(w \lambda)}^p\\
		&= \sum_\gamma\sum_{\lambda\colon \lambda\actom   \gamma\actom  
			x\in X_\Lambda} f_2(\lambda)^p\\
		&= \norm{f_2^x}_p^p.
	\end{align*}
	Putting this together with the previous inequality, we obtain \eqref{ineq: lower bound on f2}.
	
	We then have the following inequality, using our initial assumption that 
	the set $Z=(W\actom 
	X_\Lambda)\cap X_\Gamma$ has measure at least $3/4$:
	
	\begin{equation}\label{clai:integraLlowerBound}
		\int_{X_\Gamma}\sum_{\lambda : \lambda\cdot x\in Z}\abs{f_1(\lambda)}^p 
		d\mu(x)\geq \frac{3}{4}\norm{f_1}_p^p.
	\end{equation}
	Indeed, the following chain of inequalities hold:
	\begin{eqnarray*}
		\int_{X_\Gamma}\sum_{\lambda : \lambda\cdot x\in Z}\abs{f_1(\lambda)}^p d\mu(x) & \geq &  \int_{X_\Gamma}\sum_{\lambda\in \Lambda }\abs{f_1(\lambda)}^p1_{\lambda^{-1}\cdot Z}(x)d\mu(x)\\
		&  = &  \sum_{\lambda\in \Lambda }\left(\abs{f_1(\lambda)}^p\int_{X_\Gamma}1_{\lambda^{-1}\cdot Z}(x)d\mu(x)\right)\\
		&  =    & \norm{f_1}_p^p \mu(Z) \geq \frac{3}{4}\norm{f_1}_p^p, 
	\end{eqnarray*}
	so \eqref{clai:integraLlowerBound} is proved. We now use this inequality in 
	our final claim.
	\begin{clai}\label{clai:last}
		There is a subset $Y\subseteq X_\Gamma$ of measure at least $1/2$ such that 
		for all $x\in Y$, 
		\[
		\sum_{\lambda\colon \lambda\cdot x\in Z}|f_1(\lambda)|^p\geq \frac{1}{4}\norm{f_1}_p^p.
		\]
	\end{clai}
	\begin{proof}
		Indeed, assume by contradiction that 
		\[
		\sum_{\lambda\colon \lambda\cdot x\in Z}\abs{f_1(\lambda)}^p< \frac{1}{4}\norm{f_1}_p^p
		\]
		for all $x$ in a set $U\subseteq X_\Gamma$ of measure  strictly larger 
		than $1/2$. Then by Inequality \eqref{clai:integraLlowerBound}, we have
		\begin{eqnarray*}
			\frac{3}{4}\|f_1\|_p^p& \leq &  \int_{U}\sum_{\lambda\colon \lambda\cdot x\in Z}\abs{f_1(\lambda)}^p d\mu(x)+\int_{U^c}\sum_{\lambda\colon \lambda\cdot x\in Z}\abs{f_1(\lambda)}^p d\mu(x)\\
			&  <&  \frac{1}{4}\norm{f_1}_p^p+ \frac{1}{2}\norm{f_1}_p^p\\
			&  < & \frac{3}{4}\norm{f_1}_p^p,
		\end{eqnarray*}
		which is a contradiction.
	\end{proof}
	Putting the above claim together with Inequality \eqref{ineq: lower bound on f2} we deduce that for all $x\in Y$, 
	\[\sum_{\gamma\in \Gamma}\abs{f_2^x(\gamma)}^p \geq \sum_{\lambda\colon 
		\lambda\cdot x\in Z}\abs{f_1(\lambda)}^p\geq \frac{1}{4 }\norm{f_1}_p^p\]
	as wanted.
\end{proof}

\subsection{Can there be a quantitative version of the Ornstein-Weiss theorem?}\label{sec:OW}

In this section, we first observe that any two amenable groups admit a $(\varphi,\varphi)$-integrable orbit equivalence coupling for some $\varphi$ which grows slower than the logarithm, and then that there cannot be a universal such $\varphi$, thus proving Corollary \ref{corintroOrnstein} from the introduction.

\begin{proposition}
	For every orbit equivalence coupling between two finitely generated groups $\Gamma$ and $\Lambda$, there exists a concave increasing unbounded function $\varphi$ with $\varphi(0)=0$ and a $(\varphi,\varphi)$-integrable orbit equivalence coupling between them. Moreover one can assume that $\varphi(t)/\log t$ is non-increasing on $[1,\infty)$.
\end{proposition}\label{prop:orbitquantitative}
\begin{proof} Denote by $\alpha: \Gamma\times X\to \Lambda$ and $\beta:\Lambda\times X\to \Gamma$ the two cocycles associated to our orbit equivalence.
	Let \[k(t)=\mu\left(\left\{x\in X; \; \max_{s\in S_\Gamma}|\alpha(s,x)|_{S_\Lambda}\geq t\right\}\right)+ \mu\left(\left\{x\in X; \; \max_{s\in S_\Lambda}|\beta(s,x)|_{S_\Gamma}\geq t\right\}\right).\]
	Since $k$ tends to zero at infinity, there exists an increasing sequence of positive integers $(a_n)$ such that $a_0=0,$ and $k(a_n)\leq 2^{-n}$. Up to taking a subsequence, we can assume that $a_{n+1}-a_n$ and $\frac{\log a_n}{n}$ are non-decreasing. Let $\varphi$ be the continuous piecewise linear function with breakpoints $a_n$ satisfying $\varphi(a_n)=n$. Note that $\varphi$ is increasing, concave and such that $\varphi(t)/\log t$ is non-increasing on $[1,\infty)$. Moreover we have for all $s\in S_{\Gamma}$,
	\begin{eqnarray*} 
		\int \varphi(|\alpha(s,x)|_{S_\Lambda})d\mu(x) & \leq &  \sum_{n\geq 0} \varphi(a_{n+1})\mu\left(\left\{x\in X; \;|\alpha(s,x)|_{S_\Lambda}\geq a_n\right\}\right)\\
		& \leq &   \sum_{n\geq 0} (n+1) k(a_{n})\\
		& \leq & \sum_{n\geq 0} 2^{-n}(n+1) <\infty,
	\end{eqnarray*}
	The same computation shows that for all $s\in S_\Lambda$, \[\int \varphi(|\beta(s,x)|_{S_\Gamma})d\mu(x)<\infty.\] So we are done.
\end{proof}

We deduce the following quantitative version of Ornstein and Weiss's theorem. 

\begin{corollary}\label{cor: ow made quant}
	Let $\Gamma$ and $\Lambda$ be infinite finitely generated amenable groups. There exists a concave increasing unbounded function $\varphi$ satisfying $\varphi(0)=0$ and $\varphi(t)/\log t$ is non-increasing on $[1,\infty)$ such that there is a $(\varphi,\varphi)$-integrable orbit equivalence coupling between them. 
\end{corollary}

Recall the following theorem from \cite[Thm.~1.1]{brieusselSpeedRandomWalks2021}.

\begin{theorem}\label{thm:BZ}
	Let $F$ be an increasing function such that $F(t)/\log t$ is non-increasing on $[1,\infty)$. Then there exists a finitely generated group $\Gamma$ whose isoperimetric profile  satisfies $j_{1,\Gamma}\approx F$.
\end{theorem}
We deduce the following corollary, which is in sharp contrast with Corollary \ref{cor: ow made quant}.

\begin{corollary}\label{cor:Ornsteinsub-quotient}
	For every concave increasing unbounded function $\psi$ such that $\psi(t)/\log t$ is non-increasing on $[1,\infty)$ there exists a finitely generated amenable group  $\Gamma$ with the following property: for every concave function $\varphi$ such that $\varphi(0)=0$, if $\Gamma$ is a $\varphi$-integrable measure quotient (or an at most $m$-to-one measure sub-quotient) of $\Z$, then $\varphi\preccurlyeq \psi$. 
\end{corollary}
\begin{proof}
	Consider the function $F(t)=\psi(t)$, then $F(t)/\log t=\psi(t)/\log t$, so it is non-increasing on $[1,+\infty)$. We can then apply Theorem \ref{thm:BZ} and find a finitely generated group $\Gamma$ whose isoperimetric profile satisfies $j_{1,\Gamma}\approx F$. 
	
	Now let $\varphi$ be a concave non-decreasing function such that $\varphi(0)=0$, then by concavity we have that $t/\varphi(t)$ is non-decreasing. If there is a $\varphi$-integrable measure quotient coupling from $\Gamma$ to $\Z$, then we can apply Theorem \ref{thm: monotonicity of isop under phi}, and since $j_{1,\Z}(n)\approx n$ and $j_{1,\Gamma}\approx \psi$ we get that $\varphi\preccurlyeq \psi$ as wanted.
\end{proof}

\begin{proof}[Proof of  Corollary \ref{corintroOrnstein}]
	Let $\Gamma$  be a finitely generated amenable groups, and let $\varphi$ be a positive, unbounded increasing function. Reasoning as in the proof of  Corollary \ref{cor: ow made quant}, we  find $\varphi'\leq \varphi$
	that is increasing, unbounded, concave and such that $\varphi'(0)=0$. By Corollary \ref{cor: ow made quant}, there exists a concave increasing function $\theta$ and a $(\theta,\LL^0)$ measure equivalence coupling from $\Lambda$ to $\Z$.  Note that the two conditions: $f$ is increasing and $t/f(t)$ is non-decreasing are stable under composition. Hence the function $\psi=\log  \circ \varphi'\circ \theta$ satisfies those conditions. We now find a group $\Gamma$, which satisfies the conclusion of  Corollary \ref{cor:Ornsteinsub-quotient} for our function $\psi$. 
	
	Now assume by contradiction that there exists a $(\varphi,\LL^0)$-integrable measure equivalence coupling from $\Gamma$ to $\Lambda$. Such a coupling is also $(\varphi',\LL^0)$-integrable. Hence, composing these couplings yields a  $(\varphi'\circ \theta,\LL^0)$-integrable measure equivalence coupling from $\Gamma$ to $\Z$. Since $\psi$ is strictly slower than $\varphi'\circ \theta$, this contradicts  Corollary \ref{cor:Ornsteinsub-quotient}.
\end{proof}

\section{\texorpdfstring{$\LL^{\infty}$}{L-infinity} measure subgroups}\label{sec:Linfty}
In this section we discuss the notion of at most $m$-to-one 
$\LL^{\infty}$-measure subgroups. We shall see that for amenable groups this 
notion is equivalent to that of regular map, which will allow us to 
deduce various results stated in the introduction. 
\subsection{\texorpdfstring{$\LL^{\infty}$}{L-infinity} measure subgroups and 
	regular maps}
The following proposition can be extracted from the proof of \cite[Thm. 2.1.2]{shalomHarmonicAnalysisCohomology2004}, and provides a useful way of building measured subgroup couplings from \emph{Borel} sub-quotient couplings.

\begin{proposition}\label{prop: measured coupling from Borel}
	Let $\Omega$ be a standard Borel space, suppose that we have a $\Gamma\times \Lambda$-action on $\Omega$ with the following properties:
	\begin{itemize} 
		\item the $\Lambda$-action is free and admits a Borel  fundamental domain $X_\Lambda$,
		\item there exists a $\Gamma$-invariant probability measure $\mu$ for the induced action on $X_\Lambda$. 
	\end{itemize}
	Then $\mu$ has a unique $\Lambda$-invariant extension to a $\sigma$-finite invariant measure $m$ on $\Omega$ which is $\Gamma$-invariant as well. In particular, if the $\Gamma$-action on $\Omega$ is free and has a Borel fundamental domain, then $\Gamma$ is a measure subgroup of $\Lambda$.
\end{proposition}
\begin{proof}
	Since every Borel subset of $\Omega$ is a countable disjoint union of $\Lambda$-translates of Borel subsets of $X_\Lambda$, the $\Lambda$-invariant extension of $\mu$ is unique.
	We build it by letting
	$$m(A)=\sum_{\lambda\in \Lambda } \mu(( \lambda\actom  A)\cap X_\Lambda).$$
	By definition, the measure $m$ is $\Lambda$-invariant, let us show that it is $\Gamma$-invariant as well. Let $\gamma\in \Gamma$, if $A$ is a Borel subset of $X_\Lambda$, by the definition of the induced action on $X_\Lambda$ we have
	$$
	\gamma\cdot A=\bigsqcup_{\lambda\in \Lambda}( \lambda\actom   \gamma\actom  A)\cap X_\Lambda
	$$
	and so $m( \gamma\actom  A)=\mu( \gamma\actx A)=\mu(A)=m(A)$. Now for an arbitrary $\lambda\in\Lambda$ and a Borel subset $A\subseteq X_\Lambda$, we have $$m( \gamma\actom  ( \lambda\actom  A))=m( \lambda\actom ( \gamma\actom  A))=m( \gamma\actom  A)=m(A)=m( \lambda\actom  A),$$ and since every Borel subset of $\Omega$ is a countable disjoint union of $\Lambda$-translates of Borel subsets of $X_\Lambda$, we conclude that the $\Gamma$-action on $\Omega$ preserves $m$. 
\end{proof}

Recall the following definition from \cite{benjaminiSeparationProfileInfinite2012}.

\begin{definition}
	Let $\Gamma$ and $\Lambda$ be two countable finitely generated groups, a map $f\colon \Gamma\to\Lambda$ is a \textbf{regular map} if it is Lipschitz and there exists $m\in \N$ such that the preimages of $f$ have size at most $m$ for some $m\in \N$, i.e.\  $\sup_{\lambda\in \Lambda}|f\inv(\{\lambda\})|\leq m$. When there exists such a map, we say that $\Gamma$ regularly embeds in $\Lambda$.
\end{definition}

\begin{remark}
	Coarse embeddings are special cases of regular maps.
\end{remark}

The second part of the following theorem is a slightly generalized version of \cite[Thm. 2.1.2]{shalomHarmonicAnalysisCohomology2004} (which was proved for coarse embeddings).

\begin{theorem}\label{thm:regular/Linfty}
	Let $\Gamma$ and $\Lambda$ be finitely generated groups. \begin{itemize}
		\item Assume that $\Gamma$ is an at most $m$-to-one $\LL^{\infty}$ measure subgroup of $\Lambda$. 
		Then $\Gamma$ regularly embeds into $\Lambda$.
		\item Conversely, if $\Gamma$ regularly embeds into $\Lambda$ and $\Gamma$ is amenable, then there exists $m\in \N$ such that $\Gamma$ is an at most $m$-to-one $\LL^\infty$ measure subgroup of $\Lambda$. 
	\end{itemize}
\end{theorem}
We deduce from the first statement of Theorem~\ref{thm:regular/Linfty} and the fact that asymptotic dimension is monotonous under regular map (see \cite{benjaminiSeparationProfileInfinite2012}) the following corollary.
\begin{corollary}\label{cor:asdimLinfty}
	The asymptotic dimension is monotonous under taking at most $m$-to-one $\LL^\infty$ measure subgroups.\qed
\end{corollary}
And from the second part of the theorem, we deduce the following result, which we announced in the introduction.
\begin{corollary}\label{cor:iso profile monotonous reg emb}
	For every $1\leq p\leq \infty$, the isoperimetric profile is monotonous under regular map between amenable groups.
\end{corollary}
\begin{proof}
	Let $\Gamma$ and $\Lambda$ be finitely generated amenable groups. 
	By the previous theorem, the existence of a regular map from $\Gamma$ 
	to $\Lambda$ implies that there is an at most $m$-to-one $\LL^\infty$ measure 
	subgroup coupling from $\Gamma$ to $\Lambda$. Such a coupling is in 
	particular $\LL^1$, so the result follows from the second item in Theorem~
	\ref{thm: monotonicity of isop under Lp}.
\end{proof}

The rest of this subsection is dedicated to the proof of the theorem.

\begin{proof}[Proof of Theorem~\ref{thm:regular/Linfty}]
	To prove the first statement, we let $(\Omega, X_\Lambda,\mu)$ be an $m$-to-one $\LL^{\infty}$ subgroup coupling from  $\Gamma$ to  $\Lambda$, and consider the associated cocycle $\alpha\colon \Gamma\times X_\Lambda\to \Lambda$. Let  $K\geq 0$ be such that $|\alpha(s,x)|_{S_{\Lambda}}\leq K$ for all $s\in S_\Gamma$ and a.e. $x\in X$.
	The assumptions imply that for a.e.\ $x\in X_\Lambda$, the map $\alpha(\cdot,x): \Gamma\to \Lambda$ is at most $m$-to-one and $K$-Lipschitz, hence is a regular map.
	
	We now prove the second statement of the theorem.
	We let $K\geq 1$ and $m\in \N$ such that there is an $m$-to-one $K$-Lipschitz maps $\Gamma\to\Lambda$. 
	We then let $F$ be a finite set of cardinality $m$, and denote by $\pi_\Lambda\colon\Lambda\times F\to\Lambda$ the projection on the first coordinate. 
	Our coupling space is
	$$
	\Omega\coloneqq \{f\colon\Gamma\to \Lambda\times F\colon f\text{ is injective and }\pi_{\Lambda}\circ f\text{ is }K\text{-Lipschitz} \}
	$$
	Note that since $F$ has cardinality $m$, every at most $m$-to-one map 
	$\Gamma\to\Lambda$ can be lifted to an injective map $\Gamma\to\Lambda\times 
	F$, so $\Omega$ is not empty. The Lipschitz condition and the fact that balls 
	in $\Lambda$ are finite ensures that $\Omega$ is locally compact for the 
	product topology, and hence is a standard Borel space. We have a $\Gamma\times\Lambda$-action given by 
	$(\gamma,\lambda)\cdot f(g)=\lambda f(\gamma\inv g)$, and a compact fundamental 
	for the $\Lambda$-action is given by $X_\Lambda=\{f\in\Omega\colon 
	\pi_\Lambda(f(1_\Gamma))=1_\Lambda\}$. A Borel fundamental domain $X_\Gamma$ 
	for the $\Gamma$-action can be obtain as follows:  we fix a well-order $<$ on 
	$\Lambda\times F$, and let $X_\Gamma$ be the set of functions $f$ which attain their $<$-minimum at 
	$1_\Gamma$.
	
	The cocycle $\alpha\colon \Gamma\times X_\Lambda\to \Lambda$ is given by $\alpha(\gamma,f)=\pi_\Lambda(f(\gamma\inv))\inv$, and the injectivity condition implies that $\Gamma$ acts freely.
	Finally for each $\gamma\in\Gamma$, the Lipschitz condition implies that $\alpha(\gamma,\cdot)$ is bounded, and since $\Gamma$ is amenable we may find a $\Gamma$-invariant measure on $X_\Lambda$ which we extend via Proposition \ref{prop: measured coupling from Borel} in order to get the desired $\LL^\infty$ subgroup coupling. Finally, the definition of $\alpha$ and the fact that $f$ is at most $m$-to-one ensures that the coupling is at most $m$-to-one.
\end{proof}

\subsection{A continuum of 3-solvable groups}\label{sec:continuum3solvable}
In this section we prove the following result, announced in the introduction. 
\begin{theorem}\label{thm:3-solvable}
	There exists an uncountable family of groups $\Gamma_i$, such that 
	\begin{itemize}
		\item[(i)] $\Gamma_i= N_i\rtimes \Z$, where $N_i$ is locally finite, and 2-step nilpotent;
		\item[(ii)] for any $i\neq j$ and any $m\geq 1$, $\Gamma_i$ is not an at most $m$-to-one $\LL^1$-measure sub-quotient (nor an $\LL^1$-measure quotient) of $\Gamma_j$.
	\end{itemize}
\end{theorem}
Note that  (i) implies in that $\Gamma_i$ is $3$-step solvable and has 
asymptotic dimension 1. Indeed, the asymptotic dimension of an extension is less or equal than that of quotient plus
that of kernel \cite[Thm.~2.3]{dranishnikovAsymptoticDimensionDiscrete2006}, and the asymptotic dimension of a locally finite group is $0$ \cite[Thm.~2.1.]{dranishnikovAsymptoticDimensionDiscrete2006}.
We deduce the result announced in the introduction.
\begin{corollary}\label{cor:3solvable}
	There exists an uncountable family of groups $\Gamma_i$, such that 
	\begin{enumerate}[(i)]
		\item $\Gamma_i= N_i\rtimes \Z$, where $N_i$ is locally finite, and 2-step nilpotent;
		\item for any $i\neq j$ and any $m\geq 1$, $\Gamma_i$ does not regularly embed into $\Gamma_j$.
	\end{enumerate}
\end{corollary}
\begin{proof}[Proof of Theorem~\ref{thm:3-solvable}] 
	To prove the theorem, we shall use \cite[Corollary 
	3.3]{erschlerIsoperimetricInequalitiesShapes2017} which is phrased in terms 
	of Følner function. Since the Følner function is equal to 
	to the generalized inverse of the $\LL^1$-isoperimetric profile, i.e.\ $\text{F\o l}_\Gamma(k)=\inf\left\{n\colon j_{1,\Gamma}(n) \geq k \right\}$, Corollary 
	\ref{cor:iso profile monotonous reg emb} implies that it is monotonous 
	under regular maps between amenable groups. 
	
	Now  by \cite[Corollary 3.3]{erschlerIsoperimetricInequalitiesShapes2017}, for any non-decreasing function $\tau\colon[1,\infty)\to [1,\infty)$ such that $\tau(n+1)-\tau(n)\leq n$ and $\tau(n)\geq n$,  there exists a group as in (i) whose Følner function is asymptotically equivalent to $\exp(\tau(n))$. Let us now explain how to get uncountably many asymptotically incomparable such functions.

	\begin{clai}
		For every $1<a<b<2$, there is a non-decreasing function $\tau_{a,b}\colon\N\to\N$
		such that for all $n\in\N$ we have 
		
		\begin{equation} 
			\tau_{a,b}(n)\geq n\;\text{ and }\;\tau_{a,b}(n+1)-\tau_{a,b}(n)\leq n\tag{*}\label{cond: EZ condition},
		\end{equation}
		and the following three conditions are met:
		\begin{enumerate}[(1)]
			\item for all but finitely many $n\in\N$
			we have $n^a\leq \tau_{a,b}(n)\leq n^b$;
			\item there are infinitely many $n\in\N$ such that $\tau_{a,b}(n)=n^b$;
			\item \label{cond: meet large intervals}the set of $n\in\N$ such that $\tau_{a,b}(n)=n^a$ contains arbitrarily large intervals.
		\end{enumerate}
		Moreover, for every $1<a'<a<b<b'<2$, the functions $\exp(\tau_{a,b})$ and $\exp(\tau_{a',b'})$ are asymptotically incomparable.
	\end{clai}
	\begin{cproof}
		Let us start by proving the existence of $\tau_{a,b}$ as above for $1<a<b<2$. First, we fix $N\in\N$ such that for all $n\geq N$, we have $(n+1)^b-n^b\leq n$, so in particular $(n+1)^a-n^a\leq n$. Since $b<2$, we can find $N_0>N$ such that
		\[
		\sum_{n=0}^{N_0-1}n\geq N_0^b
		.\]
		We then define $\tau_{a,b}$ by induction on the interval $[0,N_0]$ letting $\tau_{a,b}(0)=0$ and for all $n<N_0$,
		\[\tau_{a,b}(n+1)=\min(\tau_{a,b}(n)+n,N_0^b).\]
		Observe that by the previous inequality $\tau_{a,b}(N_0)=N_0^b$.
		
		Suppose now by induction $\tau_{a,b}$ has been defined on an interval $[0,N_k]$, satisfies \eqref{cond: EZ condition} and satisfies $\tau_{a,b}(N_k)=N_k^b$, then it suffices to explain how to extend its definition to a bigger interval $[0,N_{k+1}]$ so that it still satisfies \eqref{cond: EZ condition} and \begin{enumerate}[(1')]
			\item for all $n\in(N_k,N_{k+1}]$
			we have $n^a\leq \tau_{a,b}(n)\leq n^b$;
			\item $\tau_{a,b}(N_{k+1})=N_{k+1}^b$;
			\item there is an interval $I_k$ of size $k$ such that for all $n\in I_k$, $\tau_{a,b}(n)=n^a$.
		\end{enumerate}
		In order to do so, let $M_k$ be the first integer such that $(M_k)^a\geq N_k^b$. We then let $\tau_{a,b}(n)=N_k^b$ for all $n\in(N_k,M_k)$, and then for all $n\in [M_k, M_k+k]$ we let $\tau_{a,b}(n)=n^a$, which takes care of condition (3').
		
		Then, since $b<2$, we may and do define  $N_{k+1}$ as the least integer such that
		\[
		\tau_{a,b}(M_k+k)+\sum_{n=M_k+k}^{N_{k+1}-1}n\geq N_{k+1}^b.
		\]
		We now define by induction on $n\in[M_k+k,N_{k+1})$ 
		\[
		\tau_{a,b}(n+1)=\min(n+\tau_{a,b}(n),n^b),
		\] 
		which by the previous inequality guarantees that $\tau_{a,b}(N_{k+1})=N_{k+1}^b$, so (2') is satisfied. 
		Finally, since for all $n\geq N_k$, we have $(n+1)^a-n^a\leq n$, an inspection of the definition shows that condition \eqref{cond: EZ condition} is still satisfied, and (1') is clearly satisfied. This finishes the construction of the desired $\tau_{a,b}$.\\
		
		We now prove the incomparability statement.
		First note that we cannot have $\exp(\tau_{a',b'})$ asymptotically bounded by $\exp(\tau_{a,b})$ since we have $\tau_{a',b'}(n)=n^{b'}$ for infinitely many $n\in\N$ but $\tau_{a,b}(n)\leq n^b$ for all but finitely many $n\in\N$ and $b<b'$.
		
		Conversely, suppose that there is  $C\in\N$ such that for all $n\in\N$, we have $\exp(\tau_{a,b}(n))\leq C\exp(\tau_{a',b'}(Cn))$. Note that the set $\{Cn\colon n\in\N\}$ has to intersect infinitely many times any set which contains arbitrarily large intervals.
		Since $\tau_{a',b'}$ satisfies condition \eqref{cond: meet large intervals} we thus find infinitely may $n\in\N$ such that $\tau_{a',b'}(Cn)=(Cn)^{a'}$. On the other hand for all but finitely many $n\in\N$  we  have $\tau_{a,b}(n)\geq n^a$. We conclude that there are infinitely many $n\in\N$ such that
		\[
		\exp(n^a)\leq C\exp((Cn)^{a'}),
		\]
		a contradiction.
	\end{cproof}
	
	We now fix for every $1<a<b$ a function $\tau_{a,b}$ as above. Then it satisfies the assumption of \cite[Corollary 3.3]{erschlerIsoperimetricInequalitiesShapes2017}, so there is a group $\Gamma_{a,b}$ satisfying (i) whose Følner function is asymptotically equivalent to $\exp(\tau_{a,b})$.
	
	Now, as explained right before the previous claim, the Følner function is monotonous under regular embeddings. So by the claim for all $1<a'<a<b<b'<2$, the groups $\Gamma_{a,b}$ and $\Gamma_{a',b'}$ are not regularly embeddable one another.  Therefore the family $(\Gamma_{1+\eps,2-\eps})_{0<\eps<1/2}$ is the uncountable family we were looking for.   
\end{proof}

\subsection{Stability of the Liouville property for lamplighter groups}
The Liouville property is now completely understood for lamplighter groups. 
\begin{theorem}\label{thm:Liouville}
	Let $F$ be a non-trivial finite group, and let $H$ be a finitely generated group. Let $\mu$ be a probability measure on $\Gamma=F \wr H$ whose support is finite and generates $\Gamma$, and let $\bar{\mu}$ be the projected probability on $H$. Then $(\Gamma,\mu)$ is Liouville if and only if $(H,\bar{\mu})$ is recurrent.  
\end{theorem}

This result is due to Kaimanovich \cite{kaimanovichExampleNonLiouville} when $F$ is abelian, and to Lyons and Peres in the general case \cite{LP15}.
By a celebrated result of Varopoulos (see \cite[\S 6.2]{Grigorian}), if $\mu$ is an admissible probability measure on $G$, $(G,\mu)$ is recurrent if and only if $G$ has at most quadratic volume growth. Therefore, we deduce the following corollary.
\begin{corollary}\label{cor:Liouville}
	Let $F$ be a non-trivial finite group, and let $H$ be a finitely generated group. Let $\mu$ be an admissible probability measure on $\Gamma=F \wr H$. Then $(\Gamma,\mu)$ is Liouville if and only if $H$ has at most quadratic volume growth.
\end{corollary}
We immediately deduce that for lamplighter groups, the Liouville property is independent on the (admissible) measure $\mu$. We now state the main result of this subsection.
\begin{theorem}\label{thm:SUBQLiouville}
	For $i=1,2$, let $F_i$ be a non-trivial finite group, and let $H_i$ be a finitely generated group. Let $\mu_i$ be a probability measure on $\Gamma_i=F_i \wr H_i$ whose support is finite and generates $\Gamma_i$. Assume that $\Gamma_1$ is either an at most $m$-to-one $L^1$ measure sub-quotient, or an $L^1$ measure quotient of $\Gamma_2$. Then if $(\Gamma_2,\mu_2)$ is Liouville, then so is $(\Gamma_1,\mu_1)$.
\end{theorem}

A well-known open question asked by Benjamini in \cite{benjaminiInstabilityLiouville} is whether Liouville property is ``invariant under quasi-isometry''. We suggest the following variant.

\begin{ques}\label{q:LiouvilleQI}
	Let $G_1$ and $G_2$  be finitely generated groups equipped with admissible probability measures $\mu_1$ and $\mu_2$. Assume that $(G_2,\mu_2)$  is Liouville and that $G_1$ regularly embeds into $G_2$. Must $(G_1,\mu_1)$ be also Liouville? \end{ques}
This question is motivated by the following immediate corollary of Theorem \ref{thm:SUBQLiouville} and Theorem \ref{thm:regular/Linfty},.
\begin{corollary}\label{cor:CoarseLiouville}
	For $i=1,2$, let $F_i$ be a non-trivial finite group, and let $H_i$ be a finitely generated group. Let $\mu_i$ be a probability measure on $\Gamma_i=F_i \wr H_i$ whose support is finite and generates $\Gamma_i$. Assume that $\Gamma_1$ regularly embeds into $\Gamma_2$. Then if $(\Gamma_2,\mu_2)$ is Liouville, then so is $(\Gamma_1,\mu_1)$. 
\end{corollary}
\begin{proof}[Proof of  Theorem \ref{thm:SUBQLiouville}]
	By \cite[\S 2]{erschlerIsoperimetricProfilesFinitely2003}, there exist constants $c,C>0$ such that the F\o lner function of $\Gamma_i$ satisfies 
	\[ce^{c\text{F\o l}_{H_i}(cn)}\leq\text{F\o l}_{\Gamma_i}(n)\leq Ce^{C\text{F\o l}_{H_i}(Cn)}.\]
	This implies that \[\log \text{F\o l}_{\Gamma_i}(n)\approx \text{F\o l}_{H_i}(n).\]
	On the other hand, by Remark \ref{remark:Folner}, we have \[\text{F\o l}_{\Gamma_1}\preceq \text{F\o l}_{\Gamma_2},\]
	from which we deduce that \[\text{F\o l}_{H_1}\preceq \text{F\o l}_{H_2}.\]
	By  Corollary \ref{cor:Liouville}, we know that $\Gamma_2$ has at most quadratic growth. Hence by \cite{coulhonRandomWalksGeometry2000}, we have $\text{F\o l}_{H_2}\preceq n^2$, and therefore $\text{F\o l}_{H_1}\preceq n^2$. Hence by \cite{coulhonRandomWalksGeometry2000}, $H_1$ has at most quadratic growth and so is recurrent. So we conclude thanks to Corollary \ref{cor:Liouville}.
\end{proof}

\section{Følner tilings}\label{sec:Ftiling}
\subsection{Følner tiling sequences and orbit equivalence}\label{section:Folner tilings}
\begin{definition}\label{def: Folner tiling}
	Let $\Gamma$ be an amenable group and $(F_k)$ be a sequence of finite subsets of $\Gamma$. We call $(F_k)$ a (left) \textbf{Følner tiling sequence} if the sequence of \emph{tiles} $(T_k)$ defined inductively by by $T_0=F_0$ and $T_{k+1}=T_kF_{k+1}$ satisfies the following conditions:  
	\begin{enumerate}[(i)]
		\item \label{cond: Folner disjoint} (tiling condition) for all $k\in\N$, $T_{k+1}$ is a \emph{disjoint} union:
		$$T_{k+1}=\bigsqcup_{\gamma\in F_{k+1}}T_k\gamma;$$
		\item (Følner condition) $(T_k)$ is a left Følner sequence: for all $\gamma\in\Gamma$, $\displaystyle\lim_{k\to+\infty}\frac{\abs{\gamma T_k\setminus T_k}}{\abs{T_k}}= 0$.
	\end{enumerate}
	If in addition there exists a decreasing sequence of finite index subgroups $\Gamma_k$ such that each $F_k$ is a set of left coset representatives of $\Gamma_{k-1}$ modulo $\Gamma_k$, then we call $(F_k)$ a   \textbf{profinite Følner tiling sequence} associated to $(\Gamma_k)$. 
\end{definition}
\begin{remark}
	The first condition amounts to saying that every element of $T_k$ can uniquely be written as $f_0\cdots f_k$ where each $f_i$ belongs to $F_i$. 
\end{remark}
\begin{remark}
	In some examples it will be more convenient to consider \emph{right} Følner tiling sequences $(F_k)$, i.e. sequences such that $(F_k\inv)$ is a left Følner tiling sequence. Nevertheless, every Følner tiling sequence will be a \emph{left} Følner tiling sequence unless specified otherwise.
\end{remark}

To every Følner tiling sequence $(F_k)$, we associate a measure-preserving $\Gamma$-action constructed as follows. We consider the standard Borel probability space $(X=\prod_k F_k,\mu)$, where each factor is equipped with the normalized counting measure, and $\mu$ is the product measure. Each element $x=(x_k)_{k\in\N}$ of $X$ defines a sequence $(g_k(x))_{k\in\N}$ of elements of $\Gamma$ given by $$g_k(x)=x_0\cdots x_k\in T_k$$

Observe that by condition \eqref{cond: Folner disjoint}, each  $g_k$ is an equidistributed random element of $T_k$. In other words, the  map $g_n$ defines an isomorphism from the finite product $(X_n=\prod_{k=0}^n F_k,\mu_n)$ to $T_n$ equipped with the renormalized counting measure. 

Since $(T_k)$ is a left Følner sequence, we deduce that for every $\gamma\in \Gamma$ and almost every $x\in X$, there is $n$ such that $\gamma g_n(x)\in T_n$. We can then write uniquely $\gamma g_n(x)=x'_0\cdots x'_n$ where $x'_i\in F_i$, and we then define
$$\gamma\cdot (x_k)_{k\in\N}=(x'_0,..., x'_n,x_{n+1},x_{n+2},...).$$
Note that this does not depend on the choice of $n$ because we have $\gamma g_{n+1}(x)=\gamma g_{n}(x)x_{n+1}\in T_{n+1}$. 
In other words, the action of $\Gamma$ on $X$ is such that for every $\gamma\in \Gamma$ and a.e.\ $x\in X$, for all $n$ large enough,
\begin{equation}\label{eq:actionGamma on X}
	g_n(\gamma\cdot x)=\gamma g_n(x).
\end{equation}

We claim that up to measure zero, this group action induces the equivalence relation of equality up to a finite set of indices, also called the {\bf cofinite equivalence relation} $E_{\mathrm{cof}}$.   In particular the action is measure preserving: indeed the cofinite equivalence relation can be realized through the natural action of the direct sum $\bigoplus_k \mathfrak{S}(F_k)$, which is obviously measure preserving.

Indeed, we have just seen that for almost every $x \in X$ and all $\gamma\in\Gamma$, $x$ and $\gamma \cdot x$ are equal up to a finite set of indices.  Conversely, if $x$ and $x'$ are such that $x_j=x'_j$ for all $j\geq k+1$, the element $\gamma=g_k(x)g_k(x')\inv \in G$ satisfies $\gamma g_k(x)=g_k(x') \in T_k$, and hence $\gamma \cdot x=x'$. 

We deduce from this discussion the following proposition. 
\begin{proposition}\label{prop:tilingOE}
	Assume that $(F_k)$ is a Følner tiling sequence for $\Gamma$. Then $\Gamma$ has a measure-preserving action on the infinite product probability space $(X=\prod_k F_k,\mu)$, which almost surely generates the co-finite equivalence relation on this product.  \end{proposition}

\begin{example}
	For instance, if $\Gamma=\Z$, $F_k=\{0,2^{k}\}$, the tiles are $T_k=\{0,...,2^{k+1}-1\}$ and for this example we get the usual odometer, up to renaming each $F_k$ as $\{0,1\}$.
\end{example}

The previous example generalizes as follows.
\begin{proposition}\label{prop:tilingProfinite}
	If $(F_k)$ is a profinite Følner tiling sequence associated to $(\Gamma_k)$, then the action given in Proposition \ref{prop:tilingOE} is isomorphic to the profinite action  of $\Gamma$ on $\varprojlim \Gamma/\Gamma_k $.
\end{proposition}
\begin{proof} 
	It follows from the assumption that $T_n\simeq X_n$ is a set of left coset representatives of $\Gamma$ modulo $\Gamma_n$. Hence the restriction of the projection $\Gamma\to \Gamma/\Gamma_n$ to $T_n$ induces a bijection $\Phi_n:X_n\to \Gamma/\Gamma_n$. 
	Note that since $F_n\subseteq \Gamma_{n-1}$, we have $\pi_{n-1}(g_n(x))=\pi_{n-1}(g_{n-1}(x))$ for all $x\in X$.
	Hence, the sequence $(\Phi_n)$ induces a map $\Phi:X\to \projlim \Gamma/\Gamma_k$, $x\mapsto (\pi_n(g_n(x)))$, which is an isomorphism of probability spaces.
	By Equation (\ref{eq:actionGamma on X}), for a.e.\ $x\in X$ and all $\gamma\in \Gamma$ we have $ g_n(\gamma\cdot x)=\gamma g_n(x)$ for all large enough $n$.
	Hence $\Phi$ intertwines the two $\Gamma$-actions and we are done.
\end{proof}

Consider the following (possibly infinite) measurable distance on $X$ given by
\[\dcof(x,x')=\inf\{n\in\N \colon \forall k\geq n, x_k=x_k'\}\]
Observe that $x$ and $x'$ are equal up to a finite set of indices if and only if $\dcof(x,x')<+\infty$. Also, by the definition of our action, for every $\gamma\in\Gamma$ and almost every $x \in X$ we have that  $\dcof(\gamma \cdot x,x)>  k$ if and only if $\gamma g_k(x)\not\in T_k$. In particular, 
\begin{equation}\label{eqn: measure and distance}
	\mu\left(\{x,\; \dcof(\gamma\cdot x,x)> k\}\right)=\frac{|T_k\setminus \gamma\inv T_k|}{|T_k|} =\frac{|\gamma T_k\vartriangle T_k|}{2|T_k|}.
\end{equation}


In order to obtain quantitative statements, we introduce the following parameters.
\begin{definition}\label{def:Folner tiling quantitative}
	Let $(\eps_k)$ be a non-increasing sequence of strictly positive numbers tending to $0$, and $(R_k)$ be a sequence of positive reals. Say that a Følner tiling sequence $(F_k)$ of a finitely generated group $\Gamma$ equipped with a finite generating set $S_{\Gamma}$  is an {\bf $(\eps_k,R_k)$-Følner tiling sequence}
	if 
	\begin{enumerate}[(i)]
		\item each tile $T_k$ has $d_{S_\Gamma}$-diameter at most $R_k$,
		\item  every $s \in S_{\Gamma}$ satisfies $\abs{T_k\setminus sT_k}\leq \eps_k|T_k|$.
	\end{enumerate}
\end{definition}

We shall need the following key observation.
\begin{lemma}\label{lem:distance on X/lenght in G}
	Let $(F_k)$ be $(\eps_k,R_k)$-Følner tiling sequence of a finitely generated group $\Gamma$ equipped with a finite generating set $S_{\Gamma}$. Then
	\begin{enumerate}[(i)]
		\item for all $s\in S_{\Gamma}$, we have for all $k \geq 0$ that
		$\mu\left(\{x \in X \colon \dcof(s\cdot x,x)> k\}\right)\leq \eps_{k};$ 
		\item for almost every $x\in X$, if  $|\gamma|_{S_{\Gamma}}> 2R_k$, then  $\dcof(\gamma\cdot x,x) > k$. 
	\end{enumerate}
\end{lemma}
\begin{proof}The first item follows from Equation \eqref{eqn: measure and distance}. To prove the second item, we simply observe that if $|\gamma|_{S_{\Gamma}}>2 R_k$, then $\gamma T_k\cap T_k=\emptyset$ as $\mathrm{diam}(T_k) \leq R_k$.
\end{proof}
Consider two amenable groups $\Gamma$ and $\Gamma'$ admitting respective Følner tiling sequences $(F_k)$ and  $(F_k')$ such that $\abs{F_k}=\abs{F_k'}$. Then Proposition \ref{prop:tilingOE} provides us with free measure-preserving actions of $\Gamma$ and $\Gamma'$ on $X=\prod_k\{1,\ldots, |F_k|\}$, with same orbits. 
The following proposition relates the parameters of the Følner tiling sequences with the integrability of the cocycles.

\begin{proposition}\label{prop:tilingOEquantitative}
	Suppose that $(F_k)$ is an $(\epsilon_k,R_k)$ Følner tiling sequences for $\Gamma$ and $(F'_k)$ is an $(\epsilon'_k,R'_k)$ Følner tiling sequence for $\Gamma'$, such that $\abs{F_k}=\abs{F_k'}$ for all $k\in \N$. Let $\varphi\colon[0,\infty)\to [0,\infty)$ be a non-decreasing function such that the sequence
	$(\varphi(2R'_k)(\varepsilon_{k-1}-\varepsilon_{k}))_{k\in\N}$ is summable.
	Then the orbit equivalence coupling constructed above from $\Gamma$ to $\Gamma'$ is $(\varphi,\LL^0)$-integrable: 
	for all $s\in S_\Gamma$ \[\int_X \varphi(d_{S_{\Gamma'}}(x,s\cdot x)) d\mu(x) <+\infty.\] 
\end{proposition}
\begin{proof}
	By Lemma \ref{lem:distance on X/lenght in G}, for all $s \in S_\Gamma$ and for all $k \in \mathbb{N}$,
	$$ \mu\left(\lbrace x \in X \colon  d_{S_{\Gamma'}}(x,s\cdot x) > 2R'_k \rbrace \right)\leq\mu\left(\{x \in X \colon \dcof(s\cdot x,x)> k\}\right)\leq\eps_k .$$
	We then exploit the fact that $\varphi$ is increasing:  for $s \in S_\Gamma$ we have that
	\begin{align*}
		\int_X \varphi(d_{S_{\Gamma'}}(x,s\cdot x)) d\mu(x) & \leq  \varphi(2R'_0) + \sum_{k=1}^\infty \varphi(2R'_k) \mu\left( \lbrace x \in X \colon 2R'_{k-1} <d_{S_{\Gamma'}}(x,s\cdot x) \leq 2R'_{k} \rbrace \right) \\ 
		& \leq \varphi(R'_0) + \sum_{k=1}^\infty \varphi(2R_k') (\varepsilon_{k-1}-\varepsilon_{k}),
	\end{align*}
	which is finite by assumption.
\end{proof}

\subsection{Applications to groups with polynomial growth}\label{sec:polynomialgrowth}
We start applying Proposition \ref{prop:tilingOEquantitative} to torsion-free abelian groups. 
\begin{proposition}\label{prop: Folner tiling for Zd}
	Let $n$ be a positive integer. The group $\Z^n$ (equipped with its standard generating set) admits  a profinite ($\varepsilon_k,R_k$)-Følner tiling sequence $(F_k)$, with $|F_k|= 2^{n}$, $R_k=n2^{k+1}$ and $\varepsilon_k = 2^{-(k+1)}$ for any $k\ge 0$. 
\end{proposition}
\begin{proof}
	We let $F_k=\{0,2^{k}\}^n$ for any $k\ge 0$. One can check that $T_k=\{0,1,\ldots, 2^{k+1}-1\}^n$, which is a  coset representative for the finite index subgroup $\Gamma_k=(2^{k+1}\Z)^n$.
	The diameter of $T_k$ is bounded by $n2^{k+1}$ and its size equals $2^{n(k+1)}$. Finally take $s$ a generator of $\Z^n$. Without loss of generality, we can assume that $s$ is the first basis vector in $\mathbb{Z}^n$. Then, we have $$T_k\setminus((1,0,\ldots, 0)+T_k)=\{0\}\times \{0,1,\ldots, 2^{k+1}-1\}^{n-1},$$ whose cardinality is $2^{k+1}$ smaller than that of $T_k$, so we are done.
\end{proof}

\begin{corollary}\label{cor: Folner tiling for Zd}
	Let $n$ and $m$ be positive integers. The group $\Z^n$ (equipped with its standard generating set) admits  a ($\varepsilon_k,R_k$)-Følner tiling sequence $(F'_k)$, with $|F'_k|= 2^{nm}$, $R_k=n2^{m(k+1)}$ and $\varepsilon_k = 2^{-m(k+1)}$ for any $k\ge 0$.
\end{corollary}
\begin{proof}
	Let $(F_k)_k$ be the Følner tilling sequence given in Proposition \ref{prop: Folner tiling for Zd} and for any $k\ge 0$ let $F'_k = F_{mk}F_{mk+1}\ldots F_{mk+m-1}$. Note that  $F'_k = \{0,2^{mk},2\cdot 2^{mk},\ldots, (2^m-1)2^{mk}\}^n$ and $T'_k = \{0,1,\ldots,2^{mk+m}-1\}^n$.
	
	As $T'_k$ is the set  $T_{mk+m-1}$ from Proposition \ref{prop: Folner tiling for Zd} we have that the diameter of $T'_k$ is at most $n2^{mk+m}$ and the set $T'_k\setminus(s + T'_k)$ has cardinality at most $2^{-mk-m}|T'_k|$ for any standard generator $s$ of $\Z^n$.
\end{proof}

The following theorem is almost sharp since by our extensions of Bowen's theorem (Theorem \ref{thm:Bowen}), if $n<m$ and $p>\frac nm$, then there cannot be an $\LL^p$ measure sub-quotient coupling from $\Z^m$ to $\Z^n$. 
\begin{theorem}\label{thm:OE between Zd and Zd'}
	For every $n$ and $m$, there exists an orbit equivalence coupling from $\Z^m$ to $\Z^{n}$ which is for every $\eps>0$ a $(\varphi_\epsilon,\psi_\epsilon)$-integrable coupling, where 
	$$\varphi_\epsilon(x)=\frac{x^{n/m}}{\log(x)^{1+\eps}}\text{ and }\psi_\epsilon(x)=\frac{x^{m/n}}{\log(x)^{1+\eps}}.$$
	In particular if $n<m$, then  then there is  an $(\LL^{p},\LL^{p'})$ orbit equivalence coupling from $\Z^m$ to $\Z^n$ for all $p<\frac{n}{m}$ and $p'<\frac{m}{n}$.
\end{theorem}
\begin{proof}
	Applying Corollary \ref{cor: Folner tiling for Zd}, we find a ($\varepsilon_k,R_k$)-Følner tiling sequence $(F_k)_k$ for $\Z^m$ with $|F_k|=2^{nm}$, $R_k=m2^{n(k+1)}$ and $\varepsilon_k = 2^{-n(k+1)+1}$.  Similarly,  we find a ($\varepsilon'_k,R'_k$)-Følner tiling sequence $(F'_k)_k$ for $\Z^{n}$  with $|F'_k|=2^{nm}$ , $R'_k=n2^{m(k+1)}$ and $\varepsilon'_k = 2^{-m(k+1)+1}$.
	Note that  $\varphi_\eps(R'_k)\varepsilon_{k-1}=O(k^{-(1+\epsilon)})$ and $\psi_\eps(R_k)\varepsilon'_{k-1}=O(k^{-(1+\epsilon)})$. 
	Thus, by Proposition \ref{prop:tilingOEquantitative}, we obtain an orbit equivalence coupling from $\Z^m$ to $\Z^n$ which is $(\varphi_\eps,\psi_\eps)$-integrable. 
	Finally if $n<m$ then for all $p<n/m$ we both have that $x^p=o( \varphi_\eps(x))$ and $x^{p'}=o( \psi_\eps(x))$ as $x\to+\infty$ so our coupling is also $(\LL^p,\LL^{p'})$.
\end{proof}

\begin{remark}
	The expert reader will recognize in the above proof an explicit orbit equivalence between the dyadic $\Z^n$ and $\Z^{m}$ odometers. Moreover, using this point of view it can be shown that this coupling is not  $\psi_0$-integrable, so we ask the following refinements of question \ref{qu: optimal coupling for Z^n}.
\end{remark}

\begin{ques}\label{qu: refined optimal coupling for Z^n}
	Let $n<m$,  is there a $(\varphi,\LL^\infty)$-integrable measure equivalence coupling from $\Z^m$ to $\Z^n$, where $\varphi(x)=\frac{x^{n/m}}{\log(x)}$ ? What about a $(\varphi,\LL^0)$-integrable measure equivalence coupling ?
\end{ques}
Recall that  the Heisenberg group is the 2-step torsion-free nilpotent group that can be defined as the group of triples $(x,y,z)\in\Z^3$ equipped with the group operation 
$$(x,y,z)\cdot(x',y',z')=(x+x',y+y',z+z'+yx'),$$
which comes from its identification with the group of matrices of the form
$\begin{bmatrix} 1 & 0 & 0\\ x & 1 & 0\\ z & y & 1 \end{bmatrix}$. We equip it with the standard generating set $S=\{E_1^{\pm 1},E_2^{\pm 1}\}$, where $E_1 = (1,0,0)$ and $E_2= (0,1,0)$.

\begin{proposition}\label{prop:heis folner tiling}
	The Heisenberg group admits a profinite $(\varepsilon_k,R_k)$-Følner tiling $(F_k)$ such that $|F_k|=16$, $R_k=10 \cdot 2^{k+2}$ , and $\eps_k=2^{-k}$ for any $k \geq 0$.
\end{proposition}
\begin{proof}
	For every $k\geq 0$ let
	\[F_k=\left\{(2^kx, 2^ky, 4^kz) \colon x,y\in\{0,1\}, z\in \{0,1,2,3\}\right\} .\]
	We claim that $(F_k)$ is the profinite left Følner tiling we are looking for. 
	
	\
	
	First note that our sequence is \emph{profinite} because the set $F_k$
	is a right coset representative of $\Gamma_{k-1}$ modulo $\Gamma_k$, where $\Gamma_{k-1}$ is the (non normal) finite index subgroup 
	\[\Gamma_{k-1}=\left\{(x,y,z)\colon x\equiv y\equiv 0 \mod 2^k, z\equiv 0 \mod 4^k \right\}.\]
	
	\newcommand{\coefxA}[1]{x_{#1}}
	\newcommand{\coefzA}[1]{z_{#1}}
	\newcommand{\coefyA}[1]{y_{#1}}
	\newcommand{\coefxB}[1]{x'_{#1}}
	\newcommand{\coefzB}[1]{z'_{#1}}
	\newcommand{\coefyB}[1]{y'_{#1}}

	\
	
	We then check that $(F_k)$ satisfies the \emph{tiling condition}. Let $A=(x_A,y_A,z_A) \in T_k$.
	By our definition of $T_k$, we can write $A$ as
	\begin{equation}\label{eq:Adecomp}
		A = (x_0,y_0,z_0)(2x_1,2y_1,4z_1)\cdots(2^{k}x_k,2^ky_k, 4^{k}z_k)
	\end{equation}
	
	with $\coefxA{i},\coefyA{i},\in\{0,1\}$,
	and $\coefzA{i}\in\{0,1,2,3\}$, which yields
	\begin{eqnarray}\label{eq: unique product decomposition}
		x_A = \sum_{i=0}^{k}2^i\coefxA{i},\quad y_A=\sum_{i=0}^{k}2^i\coefyA{i},&\text{ and }&
		z_A = \sum_{i=0}^{k}4^i\coefzA{i} + \sum_{i=1}^{k}2^i\coefxA{i}\sum_{j=0}^{i-1}2^j\coefyA{j}
	\end{eqnarray}
	We  thus see that we can recover the finite sequences $(\coefxA{i})_{i=0}^k$ and $(\coefyA{i})_{i=0}^k$ from the coefficients $(x_A,y_A,z_A)$ of $A$ as the binary expansions of $x_A$ and $y_A$ respectively, and then similarly $((\coefzA{i})_{i=0}^k)$ is also completely determined by $(x_A,y_A,z_A)$. Therefore the above decomposition of $A$ as an element of $F_0\cdots F_k$ is unique for every element in $T_k$. 
	
	\
	
	We now check that the \emph{diameters of the tiles} are as claimed.
	Let $B\in F_k$, we have $B=(2^k x,2^ky,4^kz)$ for some $x,y\in \{0,1\}$, $z\in \{0,1,2,3\}$
	and a straightforward computation shows  $B=E_1^{2^kx} E_2^{2^ky} [E_1^{2^k},E_2^{2^kz}]$.
	We thus have that $F_k$ is contained in the ball of radius $10 \cdot 2^k$ around the identity. 
	Hence $T_k$ is contained in a ball of radius $10\cdot 2^{k+1}$, and so has diameter at most $10\cdot 2^{k+2}$
	as claimed. 
	
	\
	
	We finally check the \emph{quantitative Følner condition}.
	We first estimate $|T_k \setminus E_1T_k|$. To do so, let us consider an element $A\in T_k$ such that  $E_1A\in T_k$. We decompose $A$ according to (\ref{eq:Adecomp}): 
	\[                                                                                                  
	A = (x_0,y_0,z_0)(2x_1,2y_1,4z_1)\cdots(2^{k}x_k,2^ky_k, 4^{k}z_k)=(x_A,y_A,z_A).                   
	\]
	We have  $E_1A=(x_A+1,y_A,z_A)$, so if the latter belongs to $T_k$, then by Equation \eqref{eq: unique product decomposition} one of the $x_i$'s must vanish. 
	We let $m$ be the least $i$ such that $x_i=0$ and we will now count the number of possible $A$'s for a given $m\in \{0,\ldots,k\}$.
	
	First note that there are  $2^{k-m}$ possibilities for $x_0,\ldots, x_k$ since we must have $x_0=\ldots=x_{m-1}=1$
	and $x_m=0$. There are  $2^{k+1}$ possibilities for $y_0,\ldots, y_k$, and we now have 
	to understand precisely what are the conditions on the $z_i$'s: we will bound from below the number of possible parameters $\coefzA{0},...,\coefzA{k}\in\{0,1,2,3\}$ such that the element
	\[
	A = (x_0,y_0,z_0)(2x_1,2y_1,4z_1)\cdots(2^{k}x_k,2^ky_k, 4^{k}z_k)=(x_A,y_A,z_A)
	\]
	satisfies $E_1A\in T_k$. We have  $E_1A=(x_A+1,y_A,z_A)$, so if the latter belongs to $T_k$, then by Equation \eqref{eq: unique product decomposition} its decomposition 
	\[E_1A = (x'_0,y'_0,z'_0)(2x'_1,2y'_1,4z'_1)\cdots(2^{k}x'_k,2^ky'_k, 4^{k}z'_k)
	\]
	is given by $x'_i=0$ for $i<m$, $x'_m=1$ and $x'_i=x_i$ for $i>m$, $y'_i=y_i$ for all $i\in\{0,...,k\}$, and finally the $z'_i$'s are subject to the equation 
	\[\sum_{i=0}^{k}4^i\coefzB{i} +
	\sum_{i=1}^{k}2^i\coefxB{i}\sum_{j=0}^{i-1}2^j\coefyA{j}=
	\sum_{i=0}^{k}4^i\coefzA{i} +
	\sum_{i=1}^{k}2^ix_i\sum_{j=0}^{i-1}2^j\coefyA{j}.\]
	
	We thus have:
	\begin{align}\label{eq:eqz'}
		\sum_{i=0}^{k}4^i\coefzB{i} =\sum_{i=0}^{k}4^i\coefzA{i}+ \sum_{i=1}^k 2^i(x_i-x'_i)\sum_{j=0}^{i-1}2^jy_j,
	\end{align}
	
	Now recall that $x'_i=0$ for $i<m$, $x'_m=1$ and $x'_i=x_i$ for $i>m$, so we can rewrite the last term as
	$$\sum_{i=1}^k 2^i(x_i-x'_i)\sum_{j=0}^{i-1}2^jy_j=\sum_{i=1}^{m}2^i\sum_{j=0}^{i-1}2^j\coefyA{j}-2^m\sum_{j=0}^{m-1}2^j\coefyA{j}$$
	We deduce that $\sum_{i=1}^k 2^i(x_i-x'_i)\sum_{j=0}^{i-1}2^jy_j$ is negative, and its absolute value is strictly less than $4^m$. Hence
	(\ref{eq:eqz'}) has a solution, and thus  $E_1A\in T_k$, as soon as $\sum_{i=0}^k 4^iz_i\geq 4^m$.
	Taking the contrapositive, we see that there are at most $4^m$ possibilities for the sequence $(z_i)_{i=0}^k$ so that $E_1A\not\in T_k$.
	
	We conclude that for a fixed $m$, we have $2^{k-m}$ possible $x_A$'s, $2^{k+1}$ possible $y_A$'s and at most $4^{m}$ possible $z_A$'s such that $E_1A\notin T_k$, which yields a total of $2^{k-m}2^{k+1}4^{m}=2^{2k+1}2^m$ possibilities. When $x_i=1$ for every $i$, we never have $E_1A\in T_k$, and then this adds $2^{k+1}4^{k+1}=2^{3k+3}$ possibilities. 
	
	We conclude that there are at most 
	$\displaystyle 2^{3k+3}+\sum_{m=0}^{k}2^{2k+1}2^{m}<2^{3k+4}$ choices of $A$ such that $E_1A\notin T_k$, which means that 
	$\abs{E_1T_k\setminus T_k}<2^{3k+4}$.
	
	To estimate the cardinality of $E_2T_k \setminus T_k$, we proceed similarly. We consider $A\in T_k$ such that $E_2A \in T_k$.
	We let $m$ be such that $y_0=\cdots=y_{m-1}=1$, $y_m=0$. Similarly as before, if $E_2A\in T_k$ we may write it as
	\[E_2A = (x_A,y_A+1,x_A+z_A)=(x'_0,y'_0,z'_0)(2x'_1,2y'_1,4z'_1)\cdots(2^{k}x'_k,2^ky'_k, 4^{k}z'_k).
	\]
	So by Equation \eqref{eq: unique product decomposition} for all $i\in\{0,...,k\}$ we have $x'_i=x_i$,  $y'_0=\cdots =y'_{m-1}=0$, $y'_m=1$ and $y'_i=y_i$ for all $i>m$. Finally the $z'_i$'s satisfy
	\[ 
	\sum_{i=0}^{k}4^i\coefzB{i} + 
	\sum_{i=1}^{k}2^i\coefxA{i}\sum_{j=0}^{i-1}2^j\coefyB{j} =
	\sum_{i=0}^{k}2^i\coefxA{i} +
	\sum_{i=0}^{k}4^i\coefzA{i} + 
	\sum_{i=1}^{k}2^i\coefxA{i}\sum_{j=0}^{i-1}2^j\coefyA{j}.
	\]
	We may rewrite our previous equation as
	\begin{align*}
		\sum_{i=0}^{k}4^i\coefzB{i} -\sum_{i=0}^{k}4^i\coefzA{i}&=
		\sum_{i=1}^{k}2^ix_i\sum_{j=0}^{i-1}2^j(y_j-y'_j) + \sum_{i=0}^k 2^i x_i
	\end{align*}
	We then decompose the first sum in the right term, noting that for all $i>m$,  $\sum_{j=0}^{i-1}2^j(y_j-y'_j)=-1$ :
	\begin{align*}
		\sum_{i=0}^{k}4^i\coefzB{i} -\sum_{i=0}^{k}4^i\coefzA{i}&=
		\sum_{i=1}^{m}2^ix_i\sum_{j=0}^{i-1}2^j-\sum_{i=m+1}^k2^ix_i+ \sum_{i=0}^k 2^i x_i\\
		&=\sum_{i=1}^{m}2^ix_i(2^i-1)+ \sum_{i=0}^m 2^i x_i\\
		&=\sum_{i=1}^m4^i x_i+x_0\\
		&<2^{2m+1}
	\end{align*}
	So for these $x_A$ and $y_A$ there exists at most $2^{2m+1}$ values of $z_A$ such that $E_2A\notin T_k$ (namely those such that $\sum_{i=0}^{k}4^i\coefzA{i}+2^{2m+1}\geq 4^{k+1}$) and as before we conclude that there exists at most 
	$$2^{k+1}4^{k+1}+\sum_{m=0}^{k}2^{k+1}2^{k-m}2^{2m+1}  <2^{3k+4}$$
	choices of $A$ such that $E_2A\notin T_k$.
	Thus, altogether we have that $|E_iT_k\setminus T_k| \le 2^{3k+4}$ for $i\in\{1,2\}$, and since $\vert T_k\vert = 2^{4k+4}$ we conclude that 
	\[\abs{E_iT_k\setminus T_k} \le 2^{-k} |T_k|,\]
	so that we can indeed pick $\varepsilon_k=2^{-k}$.
\end{proof}

\begin{theorem}\label{thm:OE between Heis and Z4}\label{sec:ZLamplighter}
	There exists an orbit equivalence coupling between $\Z^4$ and $\operatorname{Heis}(\Z)$ which is $(\varphi_\eps,\varphi_\eps)$-integrable for all $\eps>0$, where \[\varphi_\epsilon(x)=\frac{x}{\log(x)^{1+\eps}}.\]
	In particular the coupling is $(\LL^{p},\LL^p)$ for all $p<1$.
\end{theorem}
\begin{proof}
	This follows from Proposition \ref{prop:tilingOEquantitative}, using the Følner tilings provided by Proposition \ref{prop: Folner tiling for Zd} and Proposition \ref{prop:heis folner tiling}.
\end{proof}
\begin{remark}
	As explained in the introduction, there cannot be an $\LL^1$ measure equivalence coupling between $\Z^4$ and $\operatorname{Heis}(\Z)$. We thus ask the following question.
\end{remark}

\begin{ques}
	Is there a $(\varphi_0,\varphi_0)$-integrable measure equivalence coupling between $\Z^4$ and $\operatorname{Heis}(\Z)$, where $\varphi_0(x)=\frac x{\log(x)}$ ?
\end{ques}

\subsection{A coupling between \texorpdfstring{$\Z$}{Z} and the lamplighter 
	groups}\label{sec: Zlamp}

Given two countable groups $\Lambda$, $\Gamma$ and a function $f\colon\Gamma\to\Lambda$, the \textbf{support} of $f$ is given by $\supp f=\{\gamma\in\Gamma\colon f(\gamma)\neq e_\Lambda\}$. 
The space $\Lambda^\Gamma$ is a group for pointwise multiplication. We then define $\bigoplus_{\gamma\in\Gamma}\Lambda$ as the subgroup of $\Lambda^\Gamma$ consisting of all functions which have finite support. 
Finally, we  the \textbf{wreath product} $\Lambda \wr \Gamma$ is by definition the semi-direct product
$$\Lambda \wr \Gamma\coloneqq \bigoplus_{\gamma\in\Gamma}\Lambda\rtimes\Gamma,$$
where the group multiplication is given by $(f,\gamma)\cdot(f',\gamma')=(f(\gamma\cdot f'),\gamma\gamma')$ and
$\gamma\cdot f(g)=f(\gamma\inv g)$. When the group $\Lambda$ is finite and non-trivial, such groups are sometimes referred to as lamplighter groups. Here we shall consider the special case $\Z/m\Z \wr \Z$,
for which a standard generating set is provided by the pair $\{(0,1),(\delta_0,0)\}$ where $\delta_0(n)=\left\{\begin{array}{cl}
	1 & \text{if }n=0\\
	0 & \text{otherwise.}
\end{array}
\right.$ 

The lamplighter point of view consists in viewing each element $(f,n)$ of the group as a pair where $f$ is a configuration of lamps (off at position $i$ if $f(i)=0$), and where $n$ is the position of the ``lamplighter". Multiplying $(f,n)$ on the right by the first generator amounts to moving the lamplighter from position $n$ to $n+1$. Multiplying it by the second generator amounts to switching the light at position $n$.

\begin{proposition}\label{prop:FolnerLamplighter}
	Let $m\ge 2$ be a positive integer. The group $\Z/m\Z \wr \Z$ (equipped with the  generating set $S=\{(0,1),(\delta_0,0)\}$) admits a $(\varepsilon_k,R_k)$-Følner tiling sequence $(F_k)_k$, with $|F_0|=2m^2$, and $|F_k| = 2\cdot m^{2^k}$, $R_k = (m+1)2^{k+1}$ and $\varepsilon_k = 2^{-(k+1)}$ for $k \ge 1$.
\end{proposition}
\begin{proof}
	Given the convention used to define the wreath product, it will be more convenient to construct a right Følner tiling sequence $(F_k)$. Recall  that $(F_k^{-1})$ will then be a left Følner tiling sequence.
	
	We define $F_0 =  \{(f,n)\in\Z/m\Z\wr\Z \colon \supp(f)\subseteq \{0,1\},n\in\{0,1\}\}$ and
	\begin{align*}
		F_k = &\left\{(f,0)\in\Z/m\Z\wr\Z \colon \supp(f)\subseteq [2^{k},2^{k+1}-1]\right\} \\
		& \cup\left\{(f,2^k)\in\Z/m\Z\wr\Z \colon \supp(f)\subseteq [0,2^{k}-1]\right\}.
	\end{align*}
	By induction, we show that $T_k = \{(f,n)\in\Z/m\Z\wr\Z \colon \supp(f)\subseteq [0,2^{k+1}-1],n\in[0,2^{k+1}-1]\}$. Indeed, we have that
	\begin{align*}
		T_{k+1} & = F_{k+1}T_k\\
		& = \{(f,n)\in\Z/m\Z\wr\Z \colon \supp(f)\subseteq [0,2^{k+2}-1],n\in [0,2^{k+1}-1]\}\\
		& \quad \cup \{(f,n)\in\Z/m\Z\wr\Z \colon \supp(f)\subseteq [0,2^{k+2}-1], n\in [2^{k+1},2^{k+2}-1]\}\\
		& = \{(f,n)\in\Z/m\Z\wr\Z \colon \supp(f)\subseteq [0,2^{k+2}-1],n\in [0,2^{k+2}-1]\}.
	\end{align*}
	A straightforward cardinality computation then shows that the $F_{k+1}$-translates of $T_k$ must be pairwise disjoint, so the tiling condition is satisfied.
	
	To bound the diameter of $T_k$, observe that to join two elements $(f,n)$ and $(f',n')$ in $T_n$, the lamplighter may travel from position $n$ to $n'$, passing through the whole interval $[0,2^{k+1}-1]$, while possibly switching all the lamps (at most once) along the way. Since switching each lamp uses the generator $(\delta_0,0)$ at most $\lfloor\frac{m}{2}\rfloor$ times, such a path has length at most  $\lfloor\frac{m}{2}\rfloor 2^{k+1} + 2(2^{k+1}-1)$.	
	Hence the diameter of $T_k$ is at most $\lfloor\frac{m}{2}\rfloor 2^{k+1} + 2(2^{k+1}-1)$, which is less than $(m+1)2^{k+1}$. 
	
	Finally, in order to prove the quantitative Følner condition, it suffices to show that $|T_ks\setminus T_k|\le 2^{-(k+1)}\abs{T_k}$ for all $s\in S$.
	If $s = (\delta_0,0)$, then $T_ks\setminus T_k = \emptyset$. If $s = (0,1)$, then $$T_ks\setminus T_k = \{(f,2^{k+1})\in\Z/m\Z\wr\Z \colon \supp(f)\subseteq [0,2^{k+1}-1]\}.$$
	So either way, $|T_ks\setminus T_k| \le m^{2^{k+1}} = 2^{-(k+1)}|T_k|$, which finishes the proof of the proposition.
\end{proof}

\begin{proposition}\label{prop:couplingZLamp}
	For any integer $m\ge 2$, there exists an orbit equivalence coupling from $\Z$ to $\Z/m\Z\wr\Z$ which is $(\exp,\varphi_\eps)$-integrable for all $\eps>0$, where $\varphi_\eps(x) = \frac{\log(x)}{\log(\log(x))^{1+\eps}}$.
\end{proposition}
\begin{proof}
	Let  $(F_k)_k$ be the Følner tiling sequence of $\Z/m\Z\wr\Z$ constructed in Proposition \ref{prop:FolnerLamplighter}.
	And let $(F'_k)_k$ the Følner tiling sequence of $\Z$ defined by $F'_0 = [0,2m^2-1]$ and $$F'_k = |T_{k-1}|\,\big[0,|F_k|-1\big] = \{0,m^{2^k}2^k,2m^{2^k}2^k,\ldots,(2m^{2^k}-1)m^{2^k}2^k\}.$$
	Note that $T'_k = [0,|T_k|-1]$. So $(F'_k)_k$ is a $(\varepsilon'_k,R'_k)$-Følner tiling sequence with $R'_k=|T'_k|-1\leq m^{2^{k+1}}2^{k+1}$ and $\varepsilon'_k = \frac{2}{|T'_k|} = \frac{1}{m^{2^{k+1}}2^k}$.
	Proposition \ref{prop:tilingOEquantitative} yields an explicit orbit equivalence coupling $(X,\mu)$ between the groups $\Gamma'=\Z$ and $\Gamma=\Z/m\Z\wr \Z$. Moreover, defining $z>1$ as the solution of the equation $z^{2^3(m+1)}/m=1$, we have
	\[z^{2R_k}\varepsilon'_{k-1}\leq \frac{z^{2(m+1)2^{k+2}}}{m^{2^{k}}2^{k-1}}=(z^{2^3(m+1)}/m)^{2^k} \frac{1}{2^{k-1}}= \frac{1}{2^{k-1}}, \]
	so we deduce from Proposition \ref{prop:tilingOEquantitative} that for every $s\in S_\Gamma$,
	\[\int_X z^{d_{S_{\Gamma'}}(x,s\cdot x)} d\mu(x) <+\infty. \]
	On the other hand, we have 
	\[\varphi_{\eps}(2R'_k)\varepsilon_{k-1}\leq \frac{2^{k+1}\log m+(k+2)\log 2}{\log(2^{k+1}\log m+(k+2)\log 2)^{1+\eps}}\frac{1}{2^{k}}=O\left( \frac{1}{k^{1+\eps}}\right). \]	
	Hence we deduce from Proposition \ref{prop:tilingOEquantitative} that for every $s\in S_{\Gamma'}$,
	\[\int_X \varphi(d_{S_{\Gamma}}(x,s\cdot x)) d\mu(x) <+\infty. \]
	This concludes the proof of the proposition.
\end{proof}

\section{Almost optimal couplings between \texorpdfstring{$\Z$}{Z} and iterated 
	wreath products}\label{sec:wreathproduct}

In this section, we show that taking wreath products is well-behaved with respect to $\varphi$-integrable orbit equivalence, and use this to find almost optimal orbit equivalence couplings between $\Z$ and iterated wreath products.

\subsection{Wreath products of measure-preserving equivalence relations}

In this section, we develop the concept of wreath products of measure-preserving equivalence relations in 
order to obtain a natural proof of the following statement: if $\Gamma_1$, $\Gamma_2$ are orbit equivalent, and $\Lambda_1$, $\Lambda_2$ are orbit equivalent, then $\Lambda_1\wr\Gamma_1$ is orbit equivalent to $\Lambda_2\wr\Gamma_2$. The measure-preserving equivalence relation that witnesses that they are orbit equivalent
will be the wreath product $\mathcal S\wr\mathcal R$ of the equivalence relation $\mathcal R$ witnessing 
the orbit equivalence between $\Gamma_1$ and $\Gamma_2$ with the equivalence relation $\mathcal S$
witnessing the orbit equivalence between $\Lambda_1$ and $\Lambda_2$. We will follow the conventions concerning wreath products of groups given at the beginning of Section \ref{sec: Zlamp}. Let us also recall from Definition \ref{def: pmp eq rel} that a measure-preserving equivalence relation is an equivalence relation which comes from the orbits of a measure-preserving action of a countable group, in particular it has countable classes.

Let $\mathcal R$ be a measure-preserving equivalence relation on $(X,\mu)$, let $\mathcal S$ be a measure-preserving equivalence relation on $(L,\nu)$. Consider the space $L^{\mathcal R}$ which consists of couples $(x,(l_y)_{y\in [x]_\mathcal R})$ where $l_y\in L$ for every $y\in [x]_{\mathcal R}$. This space can be equipped with a natural standard Borel space structure which we do not make explicit for now since in our concrete case it will be easy to describe.
Moreover, one can endow it with the following probability measure $\eta$ given by
$$\eta=\int_X\nu^{\otimes[x]_\mathcal R}\, d\mu(x).$$
We then equip the space $(L^{\mathcal R}, \eta)$ with the \textbf{wreath product} $\mathcal S\wr \mathcal R$ of $\mathcal S$ by $\mathcal R$ defined by saying two couples $(x,(l_y)_{y\in [x]_\mathcal R})$ and $(x',(l'_y)_{y\in [x]_\mathcal R})$ are $\mathcal S\wr \mathcal R$-equivalent as soon as $(x,x')\in\mathcal R$ and the following two conditions are satisfied:
\begin{itemize}
	\item for all but finitely many $y\in [x]_{\mathcal R}$, we have $l_y=l'_y$;
	\item for all $y\in [x]_{\mathcal R}$, we have $(l_y,l'_y)\in\mathcal S$.
\end{itemize}
One can check that $\mathcal S\wr \mathcal R$ is indeed a measure-preserving equivalence relation.

Let us make all this completely explicit in the case of interest to us: we now suppose that $\mathcal R$ comes from a fixed \emph{free} $\Gamma$-action. We then have a cocycle map 
$c\colon\mathcal R\to\Gamma$ which takes $(x,y)\in\mathcal R$ to the unique $\gamma\in\Gamma$ satisfying $\gamma\cdot x=y$. The standard Borel space structure on $L^\mathcal R$ then comes from its natural identification to the standard Borel space $X\times L^\Gamma$ given by the bijection $\Phi\colon X\times L^\Gamma\to L^{\mathcal R}$ defined by 
\[
\Phi(x,(l_g)_{g\in\Gamma})=(x, (l_{c(x,y)})_{y\in [x]_\mathcal R}).
\]
The inverse map is given by $\Phi\inv(x,(l_y)_{y\in [x]_\mathcal R})=(x,(l_{g\cdot x})_{g\in\Gamma})$. Note that $\Phi_*(\mu\otimes \nu^{\otimes\Gamma})=\eta$. We have a natural $\Gamma$-action on $L^\mathcal R$ given by 
$\gamma\cdot (x,(l_y)_{y\in [x]_\mathcal R})=(\gamma\cdot x, (l_y)_{y\in [x]_\mathcal R})$. The corresponding action on $X\times L^\Gamma$ is given by 
$$\gamma\cdot (x,(l_g)_{g\in\Gamma})=(\gamma\cdot x, (l_{g\gamma})_{g\in\Gamma}).$$

Finally, suppose that the equivalence relation $\mathcal S$ comes from a free $\Lambda$-action on $L$. We then also have a natural $\bigoplus_{\gamma\in \Gamma}\Lambda$-action on $X\times L^\Gamma$ given by: for all $f\in \bigoplus_{\gamma\in \Gamma}\Lambda$, 
$$f\cdot (x,(l_g)_{g\in\Gamma})=(x, (f(g\inv)\cdot l_g)_{g\in\Gamma}),$$
or from the $L^\mathcal R$ viewpoint, $f\cdot (x,(l_y)_{y\in[x]_\mathcal R})=(x,(f(c(y,x))\cdot l_y)_{y\in[x]_\mathcal R})$.

Let us check that the $\Gamma$ and $\bigoplus_{\gamma\in \Gamma}\Lambda$-actions actually extend to a measure-preserving action of $\Lambda\wr\Gamma$ on $L^\mathcal R$ via the following formula:
\begin{equation}\label{eq: action wr product}
	\left(f,\gamma\right)\cdot\left(x,\left(l_y\right)_{y\in[x]_{\mathcal R}}\right)=\left(\gamma\cdot x, \left(f\left(c\left(y,\gamma\cdot x\right)\right)\cdot l_y\right)_{y\in[x]_{\mathcal R}}\right).
\end{equation}
On the one hand, we have 
\begin{align*}
	\left(f_1,\gamma_1\right)\cdot\left[\left(f_2,\gamma_2\right)\cdot \left(x,\left(l_y\right)_{y\in\mathcal R}\right)\right]&=\left(f_1,\gamma_1\right)\cdot\left(\gamma_2\cdot x, \left(f_2\left(c\left(y,\gamma_2\cdot x\right)\right)\cdot l_y\right)_{y\in [x]_{\mathcal R}}\right)\\
	&= \left(\gamma_1\gamma_2\cdot x, \left(f_1\left(c\left(y,\gamma_1\gamma_2\cdot x\right)\right)f_2\left(c\left(y,\gamma_2\cdot x\right)\right)l_y\right)_{y\in 
		[x]_{\mathcal R}}\right),
\end{align*}
while on the other hand
\[
[\left(f_1,\gamma_1\right)\left(f_2,\gamma_2\right)]\cdot \left(x, \left(l_y\right)_{y\in\mathcal R}\right)
=
\left(\gamma_1\gamma_2\cdot x, \left (f_1\left(c\left(y,\gamma_1\gamma_2\cdot x\right)\right)f_2\left(\gamma_1\inv c\left(y,\gamma_1\gamma_2\cdot x\right)\right)\cdot l_y\right)_{y\in[x]_{\mathcal R}}\right).
\]
Since $\gamma_1\inv c(y,\gamma_1\gamma_2\cdot x)=c(y,\gamma_2\cdot x)$, the last terms of the two above equations are equal, thus we do have an action of the wreath product $\Lambda\wr\Gamma$ on $L^{\mathcal R}$.
Since the $\Lambda$-action induces the equivalence relation  $\mathcal S$ while the $\Gamma$-action induces the equivalence relation $\mathcal R$, the equivalence relation induced by this $\Lambda\wr \Gamma$-action is equal to $\mathcal S\wr\mathcal R$.
Moreover this action is measure-preserving, so $\mathcal S\wr\mathcal R$ is a measure-preserving equivalence relation. Finally, note that this action
is free as a direct consequence of the freeness of the initial $\Gamma$ and $\Lambda$-actions. We have shown the following result.

\begin{proposition}
	Let $\Gamma\act(X,\mu)$ freely, let $\Lambda\act(L,\nu)$ freely, where $(X,\mu)$ and $(L,\nu)$ are two standard probability spaces and the actions are measure-preserving. Denote by $\mathcal R$ and $\mathcal S$ the respective associated equivalence relations, and by $c\colon\mathcal R\to \Gamma$ the cocycle defined by $c(x,\gamma\cdot x)=\gamma$. Then the measure-preserving free $\Gamma\wr\Lambda$-action on $L^{\mathcal R}$ given by 
	\[
	\left(f,\gamma\right)\cdot\left(x,\left(l_y\right)_{y\in[x]_{\mathcal R}}\right)
	=\left(\gamma\cdot x, \left(f\left(c\left(y,\gamma\cdot x\right)\right)\cdot l_y\right)_{y\in[x]_{\mathcal R}}\right)
	\]
	induces the equivalence relation $\mathcal S\wr\mathcal R$. \qed
\end{proposition}

\subsection{Wreath products and quantitative orbit equivalence}\label{sec:wreathproductOE}

It follows directly from the previous proposition that when $\Gamma_1$, $\Gamma_2$ are orbit equivalent, and $\Lambda_1$, $\Lambda_2$ are orbit equivalent, then $\Lambda_1\wr\Gamma_1$ is orbit equivalent to $\Lambda_2\wr\Gamma_2$.
We now give a quantitative version of this fact. 

For this, it is useful to identify $\Gamma$ and $\Lambda$ to subgroups of $\Gamma\wr\Lambda$ as follows: first we embed $\Lambda$ in $\bigoplus_{\gamma\in \Gamma}\Lambda$ by associating to every $\lambda\in\Lambda$ the function 
\[
\iota(\lambda)\colon \gamma\mapsto\left\{\begin{array}{lc}
	\lambda & \text{ if }\gamma=e_\Gamma \\ 
	e_\Lambda & \text{ otherwise.}
\end{array} \right.
\]
Then we embed $\Lambda$ in $\Lambda\wr\Gamma$ via $\lambda\mapsto(e_\Gamma,\iota(\lambda))$, and we embed $\Gamma$ into $\Lambda\wr\Gamma$ via $\gamma\mapsto(\gamma,\iota(e_\Lambda))$. It is well known that if $S_\Gamma$ is a generating set for $\Gamma$ and $S_\Lambda$ is a generating set for $\Lambda$, then through the above identification $S_\Lambda\cup S_\Gamma$ is a generating set for the wreath product $\Lambda\wr\Gamma$.


\begin{proposition}\label{prop:wreath product}
	Let $(X,\mu)$  be an orbit equivalence coupling between finitely generated groups $\Gamma_1=\la S_{\Gamma_1}\ra$ and $\Gamma_2=\la S_{\Gamma_2}\ra$, and let $(L,\nu)$ be an orbit equivalence coupling between $\Lambda_1=\la S_{\Lambda_1}\ra$ and $\Lambda_2=\la S_{\Lambda_2}\ra$. Denote by $\mathcal R$ the equivalence relation generated by the $\Gamma_i$-action.
	
	Then the orbit coupling $L^{\mathcal R}$ satisfies that for every $(x,l)\in L^{\mathcal R}$, every $\gamma_1\in\Gamma_1$ and every $\lambda_1\in\Lambda_1$,	
	\begin{align*}
		d_{S_{\Lambda_2}\cup S_{\Gamma_2}}((x,l),\gamma_1\cdot(x,l))& =d_{S_{\Gamma_2}}(x,\gamma_1\cdot x)\\
		d_{S_{\Lambda_2}\cup S_{\Gamma_2}}((x,l),\lambda_1\cdot(x,l))& =d_{S_{\Lambda_2}}(l_{e_\Gamma},\lambda_1\cdot l_{e_\Gamma})	
	\end{align*}
\end{proposition}
\begin{proof}
	By definition for all $i\in\{1,2\}$, all $\gamma_i\in\Gamma_i$, and all $(x,l)\in L^{\mathcal R}$, we have $\gamma_i\cdot(x,l)=(\gamma_i\cdot x, l)$, so the first equation is clear. The second equation follows similarly by noting that if $\lambda_i\in\Lambda_i$, then their action on $(x,l)$ only changes the value at $e_\Gamma$ of $l$ according to the $\lambda_i$-action on $L$.
\end{proof}
\begin{corollary}\label{cor: wreath product}
	If $\Gamma_1$ and $\Gamma_2$ admit a $(\varphi,\psi)$-integrable orbit equivalence coupling, and if $\Lambda_1$ and $\Lambda_2$ admit a $(\varphi,\psi)$-integrable orbit equivalence coupling, then the wreath products $\Lambda_1\wr\Gamma_2$ and $\Lambda_2\wr\Gamma_2$ also admit a $(\varphi,\psi)$-integrable orbit equivalence couplings.\qed
\end{corollary}
In combination with Theorem \ref{thm:OE between Zd and Zd'} we find the following:

\begin{corollary}
	Let $a,b\in\N$ with $a<b$ and let $\Delta$ be any finitely generated group, then there is an $(\LL^p,\LL^{1/p})$-orbit equivalence coupling from $\Delta\wr\Z^b$ to $\Delta\wr\Z^a$ for every $p<\frac{a}{b}$.  
	\qed\end{corollary}
More generally, we have:
\begin{corollary}\label{cor:iteratedwreathZ^ab}
	Given $d\geq 1$, and a sequence of finitely generated groups $\Delta_n$, we define $H_n(d)$ inductively as follows: $H_0(d)=\Z^d$ and $H_{n+1}(d)=\Delta_{n+1}\wr H_n(d)$. Then  if $a<b$, there is an $(\LL^p,\LL^{1/p})$-orbit equivalence coupling from $H_n(b)$ to $H_n(a)$ for every $p<\frac{a}{b}$.
	\qed\end{corollary}
Recall that we denote by $\log^{\circ n}$ the function $\log $ composed with itself $n$ times. Now assume that $\Delta_n=\Z/2\Z$ for all $n\geq 1$. By \cite{erschlerIsoperimetricProfilesFinitely2003}, the isoperimetric profile of $H_n(d)$ is $\approx (\log^{\circ n})^{1/d}$.
Hence deduce from the second part of Theorem \ref{thmintroIsop} that the existence of a $(\LL^p,\LL^0)$ measure equivalence coupling from $H_n(b)$ to $H_n(a)$ forces the inequality $p\leq a/b$: this shows that  Corollary \ref{cor:iteratedwreathZ^ab} is (almost) sharp in that case.

We denote by $\LL^{<\infty}$ the intersection of $\LL^p$ for all $1\leq p<\infty$.
\begin{corollary}\label{cor:iteratedwreath}
	Let $m\geq 2$, and let $G_n$ be defined inductively by $G_0=\Z$, and $G_{n+1}=\Z/m\Z\wr G_n$. For any integer $n\ge 1$, there exists an orbit equivalence coupling from $\Z$ to $G_n$ which is $(\LL^{<\infty},\varphi_{n,\eps})$-integrable for every $\eps>0$, where $\varphi_{n,\eps}=\log^{\circ n}/(\log^{\circ (n+1)})^{1+\eps}$.
\end{corollary}
\begin{proof}
	We first claim that for all $n\geq 1$, there exists an  orbit equivalence coupling from $G_{n-1}$ to $G_{n}$ which is $(\LL^{<\infty},\varphi_{1,\eps})$-integrable for all $\eps>0$: this follows by induction using Proposition \ref{prop:wreath product}, the case $n=1$ resulting from Proposition \ref{prop:couplingZLamp}.
	
	We now pass to the proof of the corollary, which we also do by induction on $n$.
	The base case $n=1$ follows from the claim, so we take $n\geq 2$ and assume that we have a  orbit equivalence coupling from $\Z$ to $G_{n-1}$ which is $(\LL^{<\infty},\varphi_{n-1,\eps})$-integrable for all $\eps>0$.
	Using Proposition \ref{prop:CombineIntCouplingsConcave}, we compose this coupling with the coupling from the claim between $G_{n-1}$ and $G_{n}$, and obtain a $(\varphi_{n-1,\eps}\circ\varphi_{1,\eps}, \LL^{<\infty})$-coupling from $\Z$ to $G_{n}$. Note that for all $k\geq 1$ 
	\[\log^{\circ k}\circ \varphi_{1,\eps}(t)=\log^{\circ k}(\log t/(\log^{\circ 2}(t))^{1+\eps})\sim  \log^{\circ (k+1)}(t)\] as $t$ tends to $+\infty$. Hence, $ \varphi_{n-1,\eps}\circ\varphi_{1,\eps}(t)\sim \varphi_{n,\eps}(t)$, and we are done.
\end{proof}

\section{Unstable properties under exponential couplings}\label{sec: unstable}

\subsection{Finite presentability is unstable}
\label{sec:FP_Unstable}

In this section we prove that being finitely presented is not preserved by strongly exponentially integrable orbit equivalence couplings (see Definition~\ref{def:strong}). We do this by constructing an explicit strongly exponentially integrable orbit coupling between the wreath product $\Z/k\Z\wr\Z$ and  the Baumslag-Solitar group $\BS(1,k)$, for every $k\ge 2$. 

Recall from the previous section that $\Z/k\Z\wr\Z = \bigoplus_{\Z}\Z/k\Z \rtimes \Z$ where $\Z$ acts by a shift on $\bigoplus_{\Z}\Z/k\Z$, that is	\[((x_i)_{i\in\Z},m_x)((y_i)_{i\in\Z},m_y) = ((x_i+y_{i-m_x})_{i\in\Z},m_x+m_y).\]
Also recall that $\BS(1,k)=\Z[1/k]\rtimes\Z$, where $\Z[1/k]=\{\frac{x}{k^m}\colon x,m\in\Z\}$ and $\Z$ acts on $\Z[1/k]$ by multiplication by $1/k$, that is 
\[
\left(\frac{x}{k^{m_x}},n_x\right)\left(\frac{y}{k^{m_y}},n_y\right) = \left(\frac{x}{k^{m_x}}+\frac{y}{k^{m_y+n_x}},n_x+n_y\right).
\]

It is well known that the second group is finitely presented, while the first one is not by a result of Baumslag \cite{baumslagWreathProductsFinitely1961}. In particular these groups are not quasi-isometric.

We will consider actions of these two groups on the standard Borel space $X=\prod_{\Z}\Z/k\Z$ equipped with the measure $\mu$ defined as the infinite product of the normalized counting measure. 
The action of $\Z/k\Z\wr\Z$ is the standard one, that is \[((x_i)_{i\in\Z},m)\cdot(y_i)_{i\in\Z} = (x_i+y_{i-m})_{i\in\Z}.\]

For the Baumslag-Solitar group $\BS(1,k)$ we first define the action of the subgroup $\Z[1/k]=\la (k^m,0)\colon m\in\Z\ra$ as follows: for all $m\in\Z$, we decompose the space $X$ as \[
X=\prod_{i<m}\Z/k\Z\times\prod_{i\geq m}\Z/k\Z,
\] and then $(k^m,0)$ acts trivially on the first factor, and as the $k$-adic odometer on the second factor.
To be completely explicit,
$(k^m,0)\cdot (x_i)_{i\in\Z}=(y_i)_{i\in\Z}$, where:
\begin{itemize}
	\item $x_i=y_i$ for every $i<m$;
	\item  if $x_i=k-1$ for every $i\ge m$, then $y_i=0$ for every $i\ge m$;
	\item else, letting $N$ be the smallest integer $\geq m$ such that $x_N\neq k-1$, we let $y_i=0$ for $m\le i<N$, $y_N=x_N+1$, and $y_i=x_i$ for $i>N$.
\end{itemize}
Note that this action is realized by letting $\Z[1/k]$ act on $\Q_k$ by addition and extending this action to the bi-infinite product $X=\prod_{\Z}\Z/k\Z$.
Also note that the equivalence relation generated by the action of $\Z[1/k]$ is the cofinite relation up to measure zero. 

As with $\Z/k\Z\wr\Z$ we then let $\Z$ act by  $(0,1)\cdot(x_i)_{i\in\Z}=(x_{i-1})_{i\in\Z}$. This defines an (essentially) free measure preserving action of the group $\Z[1/k]\rtimes\Z$ on $(X,\mu)$. Since the natural action of $\oplus_{\Z}\Z/k\Z$ induce the cofinite equivalence relation, we deduce that the two actions we have built yield an orbit equivalence coupling $(X,\mu)$ from $\Z/k\Z\wr\Z$ to $\BS(1,k)$. 

\begin{theorem}\label{thm:Lamplighter et BS}
	The orbit equivalence coupling $(X,\mu)$ we just constructed is an $(\LL^\infty,\exp^\diamond)$ orbit equivalence coupling from $\Z/k\Z\wr\Z$ to  $\BS(1,k)$. 
\end{theorem}
\begin{proof}
	We equip $\Z/k\Z\wr\Z$ with the generating subset $S=\{\delta_0, 1_{\Z}\}$ and $\BS(1,k)$ with the generating subset $T= \{ 1_{\Z[1/k]}, 1_{\Z}\}$. We denote by $S^\pm=S\cup S\inv$ and $T^\pm=T\cup T\inv$ the corresponding symmetric generating sets.

	Let us start by showing that our orbit equivalence coupling is $\LL^\infty$ as a coupling from $\Z/k\Z\wr\Z$ to $\BS(1,k)$. For this, we only need to check that $d_{T^\pm}(s\cdot x,x)$ is uniformly bounded for all $s\in S$. Note that the generator of the $\Z$ copy in the two groups acts exactly the same. So for $s= 1_\Z$, we have $d_{T^\pm}(s\cdot x,x)=1$. Then observe that the generator $\delta_0$ changes only $x_0$, which is achieved by the action of $s^m$ for some $m\in\{-k+1,\cdots,k-1\}$, where $s=1_{\Z[1/k]}$. Hence $d_{T^\pm}(\delta_0\cdot x, x)\leq k-1$, which proves that our coupling is $\LL^\infty$ in one direction.
	
	To prove that the coupling is strongly exponential in the other direction, we need to obtain estimates for every $g\in \BS(1,k)$.
	So let $g\in \BS(1,k)$ , we will show that $g$ satisfies the following estimate: 
	for all $M\geq 0$,
	\begin{equation}\label{eq: estimate for BS and lamplighter}
		\mu\Bigl(\left\{x\in X\colon d_{S^\pm}(g\cdot x,x)	\geq
		(k+1)\left(2\abs g_{T}+2M+3\right)\right\}\Bigr)\leq k^{-M+1}
	\end{equation}
	Let $n=\abs g_{T^\pm}$, write $g=(z,j)$, with $|j|\leq n$  and $z\in\Z[1/k]$. Observe that $|(0,j)|_{S^\pm}=|j|$. Hence, for a.e.\ $x\in X$, $d_{S^\pm}((0,j)\cdot x,x)=|j|$. By the triangle inequality, it is therefore sufficient to prove that
	\begin{equation}\label{eq: easier estimate for BS and lamplighter}
		\mu\Bigl(\left\{x\in X\colon d_{S^\pm}(z\cdot x,x)\geq k(2|g|_{T}+2M+3)\right\}\Bigr)\leq k^{-M+1}
	\end{equation}
	By symmetry we may assume that $z>0$. A straightforward induction on $n$ then shows that we can always write $z$ as $z=\sum_{i=-n}^n a_ik^i$, where $a_i\in\{0,...,k-1\}$.

	Let $x\in X$. Observe that the coefficients of $z\cdot x$
	that differ from those of $x$ are contained in the interval $[-n,\infty)$. Given $m\ge M$, consider the event $(z\cdot x)_{n+m+2}\neq x_{n+m+2}$, and observe that its occurrence forces the coefficients of $x_{n+i}$ for $i=1,2, \ldots, m+1$ to be equal to $k-1$. Letting
	\[E=\{x\in X\mid  \textnormal{ there exists $m\geq M$ such that } (z\cdot x)_{n+m+2}\neq x_{n+m+2}\},\] 
	we thus have $\mu(E)\leq \sum_{m\geq M}k^{-m}\leq k^{-M+1}$.
	
	Now note that $|f|_{S^\pm}\le k(2n+M+2)$ for every $f\in \bigoplus_\Z\Z/k\Z$ with $\supp(f)\subseteq[-n,n+M+1]$, so 
	\[\left\{x\in X\colon d_{S^\pm}(z\cdot x,x)\geq k(2|g|_{T}+2M+3)\right\}\subseteq E.\]
	Thus the estimate \eqref{eq: easier estimate for BS and lamplighter} holds, and hence so does \eqref{eq: estimate for BS and lamplighter}.
	
	Let us finally check that this estimate yields the desired result. For a given $\delta>0$ and $g\in \BS(1,k)$, by the integration by slices inequality and \eqref{eq: estimate for BS and lamplighter} we have 
	\begin{align*}
		\int_X \exp(\delta d_{S^\pm}(x,g\cdot x)d\mu(x) & \leq\sum_{M\geq 0} k^{-M+1}\exp\bigl(\delta k(2\abs g_T+2M+3)\bigr)\\
		&= \sum_{M\geq 0} k\exp\bigl(\delta k(2\abs g_T+M(2-\log k/\delta)+3)\bigr)\\
		&= k\exp(3\delta k)\exp\bigl( 2\delta k\abs g_T\bigr)\sum_{M\geq 0} \exp\bigl(\delta kM(2-\log k/\delta)\bigr)
	\end{align*}
	So given $\eps>0$, we take $\delta$ such that $2\delta k<\eps$ and $2-\log k/\delta<0$, and then we will have a positive constant $C>0$ such that for every $g\in \BS(1,k)$, 
	\[
	\int_X \exp(\delta d_{S^\pm}(x,g\cdot x))d\mu(x) <C\exp(\eps \abs g_T),
	\]
	which finishes the proof that the coupling  is strongly $\exp$-integrable from $\BS(1,k)$ to $\Z/k\Z\wr \Z$. 	
\end{proof}

\begin{corollary}\label{cor:finitePresentabilityUnstable}
	Finite presentability is unstable under the equivalence relation of strongly exponential orbit equivalence. In particular, it is unstable under $\LL^{<\infty}$ orbit equivalence.
\end{corollary}

\subsection{Finite asymptotic dimension is unstable}

We now consider the following group:  $$\Gamma_1=\bigoplus_{\Z}\Z[1/k]\rtimes \Z^2,$$ where the first coordinate of $\Z^2$ acts by shift on the direct sum $\oplus_{\Z}\Z[1/k]$, and the second coordinate multiplies each factor by the corresponding power of $k$. 
Recall the action of $\Z[1/k]$ on $X=\prod_{\Z}\Z/k\Z$ as defined in section \ref{sec:FP_Unstable} induces the cofinite equivalence relation up to measure zero. For every $n\in \Z$, we let $X_n$ be a copy of $X$. We deduce an action of the direct sum $\oplus_{n\in \Z}\Z[1/k]$ on $Y=\prod_{n\in \Z}X_n=\prod_{\Z^2}\Z/k\Z$ which induces the cofinite equivalence relation up to measure zero.
The cofinite equivalence relation is also induced by the natural action of the group $\oplus_{\Z^2}\Z/k\Z$  on $Y$. 

Combining these actions with the Bernoulli action of $\Z^2$ produces an orbit equivalence coupling between $\Gamma_1$ and the wreath product $$\Gamma_2=\Z/k\Z\wr\Z^2.$$

We equip $\Gamma_1$ with the generating subset $S_1= \{\pm 1_{\Z[1/k]_0}\}\cup S_{\Z^2}$, and $\Gamma_2$ with $S_2=\{\delta_0\}\cup S_{\Z^2}$, where $S_{\Z^2}$ is the canonical generating set of $\Z^2$.

\begin{theorem}\label{thm:G1 et G2}
	The coupling we just defined is an $(\LL^\infty,\exp)$-integrable orbit equivalence coupling from $\Gamma_2$ to $\Gamma_1$. Actually the exponential integrability from $\Gamma_1$ to $\Gamma_2$ is uniform in the following sense: there exists a constant $c>0$ such that for all $g\in \Gamma_1$ 
	\[\int_X \exp(cd_{S_2}(x,g\cdot x))d\mu(x)<+\infty.\]
\end{theorem}
\begin{proof}
	The proof that the coupling is $\LL^\infty$ from $\Gamma_2$ to $\Gamma_1$ is exactly the same as in the proof of Theorem \ref{thm:Lamplighter et BS}, so we don't write it down.
	Also similar to the proof of Theorem \ref{thm:Lamplighter et BS}, we observe that the two respective copies of $\Z^2$ act the same. As a warm up, let us prove the exponential integrability from $\Gamma_1$ to  $\Gamma_2$. 
	Thanks to Proposition \ref{prop: only check integrability on generators} only needs to consider the action of the generators $\pm 1_{\Z[1/k]_0}$. Observe that the coupling induces on every $X_n$ an orbit-equivalence coupling between $\Z[1/k]_0$ and $\oplus_{\Z\times\{0\}}\Z/k\Z$ that coincides with that of section \ref{sec:FP_Unstable}. Hence the result follows from Theorem \ref{thm:Lamplighter et BS}. 
	
	The uniform exponential integrability requires a bit more work. It will result from the following statement: given $g\in \Gamma_1$ and $M\geq 0$ we have that
	\begin{equation}\label{eq:uniformExp}
		\mu\left(\left\{y\in Y\colon d_{S_2}(y,g\cdot y)\geq \left(k+1\right)\left(2M+(2|g|+5)(2|g|+1)\right) + |g|\right\}\right)\leq  \int_{M}^{+\infty}t^{2|g|+1}k^{-t}dt,
	\end{equation}
	where we write $|g|$ instead of $|g|_{S_1}$ in order to simplify notation.
	Let us first see why the result follows from this inequality. Denoting $C=(k+1)(2|g|+5)(2|g|+1)+|g|$ and $T=2(k+1)M$, 
	we obtain that for all $T\geq 0$,
	
	\[\mu\left(\left\{y\in Y\colon d_{S_2}(y,g\cdot y)\geq T+C\right\}\right)\leq  \int_{T/(2k+2)}^{+\infty}t^{2|g|+1}k^{-t}dt.\]
	For all $\eps>0$, the latter is at most $k^{-\left(\frac{1-\eps}{2k+2}\right)T}K_{\eps}$, 
	where $K_{\eps}=\int_0^{+\infty}t^{2|g|+1}k^{-\eps t}dt<+\infty$.
	Integrating over $T\geq 0$, it follows that
	\[\int_{0}^{+\infty}\left(k^{\left(\frac{1-2\eps}{2k+2}\right)T}\mu\left(\left\{y\in Y\colon d_{S_2}(y,g\cdot y)\geq T+C\right\}\right)\right)dT\leq
	K_\eps\int_{0}^{+\infty}k^{-\left(\frac{\eps}{2k+2}\right)T}dT <+\infty.\]
	Using the change of variable $T\leftarrow T+C$ and the integration by slice formula, we deduce that 	\[\int_X k^{\left(\frac{1-2\eps}{2k+2}\right)d_{S_2}(x,g\cdot x)}d\mu(x)<+\infty,\]
	Since $\eps$ is arbitrary, we finally deduce that 
	\[\int_X e^{cd_{S_2}(x,g\cdot x)}d\mu(x)<+\infty,\]
	for all $c<\log k/(2k+2)$.
	
	We now prove that \eqref{eq:uniformExp} holds. 
	We will use the fact that the $n^\text{th}$ copy $\Z[1/k]$ and $\oplus_{\{n\}\times\Z}\Z/k\Z$ have the same orbits. Moreover they only act non-trivially on the $n^\text{th}$ factor $X_n$, and their actions are those defined in Section \ref{sec:FP_Unstable}.
	Let $y=(y_{i,j})$ and $g=((x_{i}),(m,n))$. Note that $|n|+|m|\leq |g|$. Hence
	\[d_{S_2}(y,g\cdot y)\leq |g|+\sum_{i=-|g|}^{|g|} d_T((y_{i-n,j-m})_{j\in \Z},x_i\cdot(y_{i-n,j-m})_{j\in \Z}), \]	
	where $T$ is the generating subset of $\Z/k\Z\wr\Z$ from Section \ref{sec:FP_Unstable}.
	Now for every $-|g|\leq i\leq |g|,$ let $M_i\in \R$ be such that 
	\[d_T((y_{i-n,j-m})_{j\in \Z},x_i\cdot(y_{i-n,j-m})_{j\in \Z})=(k+1)(2|g|+2M_i+5).\]
	We deduce that $\sum_{i=-|g|}^{|g|}M_i\ge M$. Note that $|g|\geq |(x_i)_{i\in \Z}|_{S}$,  where $S$ is the generating subset of $\BS(1,k)$ from Section \ref{sec:FP_Unstable}.
	On the other hand, by (\ref{eq: estimate for BS and lamplighter}) we have that for all $(t_i)\in \R_+^{2|g|+1}$, \[\mu\left(\left\{y\in Y\colon d_T((y_{i-n,j-m})_{j\in \Z},x_i\cdot(y_{i-n,j-m})_{j\in \Z})\geq (k+1)(2|g|_S+2t_i+5)\right\}\right)\leq k^{-t_i}.\]
	For every $t\geq 0$ and $m\in \N$, denote by $v_m(t)=\vol\{(t_i)\in \R_+^{m}, \; \sum_{i}t_i=t\},$ and observe that  $v_m(t)\leq t^m$. Finally, integrating over all values for $t_i=\max\{M_i,0\}$ for which $t=\sum_i t_i\geq M$, we obtain
	
	\begin{eqnarray*}
		\mu\left(\left\{y\in Y\colon d_{S_2}(y,g\cdot y)\geq (k+1)(2M+(2|g|+5)(2|g|+1)) + |g|\right\}\right) & \leq & \int_{t\geq M}k^{-t}v_{2|g|+1}(t)dt\\
		& \leq & \int_{t\geq M}t^{2|g|+1}k^{-t}dt
	\end{eqnarray*}
	which ends the proof of (\ref{eq:uniformExp}) and therefore of the theorem.
\end{proof}

Recall that if $H$ is a subgroup of a countable group $G$, then $\asdim H\leq \asdim G$ 
\cite[Cor.~63]{BelDranishnikovAsdimSurvey2008}. 
We deduce that $\Gamma_1$ has infinite asymptotic dimension, as it contains a subgroup isomorphic to $\bigoplus_{n\in \N} \Z$, which has infinite asymptotic dimension by \cite[Thm.~74]{BelDranishnikovAsdimSurvey2008}. On the other hand, by \cite[Thm.~3.4]{dranishnikovAsymptoticDimensionDiscrete2006} the asymptotic dimension of $\Gamma_2$ is at most $2$. It is actually equal to $2$, as $\Gamma_2$ contains $\Z^2$, which has asymptotic dimension $2$ by \cite[Thm.~74]{BelDranishnikovAsdimSurvey2008}. We deduce the following corollary.
\begin{corollary}\label{cor:asympdimUnstable}
	There exist a group with asymptotic dimension $2$, and a group of infinite asymptotic dimension that admit an exponential-integrable orbit equivalence measure coupling.
\end{corollary}
	\bibliographystyle{alpha}
	\bibliography{zoterobib}
	
\end{document}